\documentclass[a4paper,reqno]{amsart}
\usepackage[T1]{fontenc}
\usepackage[latin1]{inputenc} \usepackage{amssymb,
  amsmath, amsfonts, amsthm} \usepackage{url}
\usepackage{color} \usepackage{euscript}
\usepackage{verbatim} \usepackage{upref}
\usepackage{graphicx}

\renewcommand{\phi}{\varphi}
\newcommand{\fracpar}[2]{\frac{\partial{#1}}{\partial{#2}}}

\newcommand{\F}{\mathcal{F}}
\newcommand{\Winf}{{W^{1,\infty}(\Real)}}
 
\newcommand{\muac}{\mu_{\text{\rm ac}}}
\newcommand{\D}{\mathcal{D}}
\newcommand{\G}{\mathcal{G}}
\newcommand{\I}{\mathbf{I}}

\newcommand{\lP}{\mathbf{P}}

\newcommand{\Ltwo}{{L^2}}
\newcommand{\Hone}{{H^1}}
\newcommand{\Linf}{{L^\infty}}
\newcommand{\sfrac}[2]{\mbox{\footnotesize
    $\displaystyle\frac{#1}{#2}$}}
\newcommand{\abs}[1]{\left\vert#1\right\vert}
\newcommand{\R}{\mathbb R}
\newcommand{\Real}{\mathbb R}

\newcommand{\e}{\ensuremath{\mathrm{e}}}

\newcommand{\epsi}{\varepsilon}

\newcommand{\norm}[1]{\left\Vert#1\right\Vert}
\newcommand{\snorm}[1]{\Vert#1\Vert}

\newcommand{\id}{\text{Id}}

\newcommand{\bigo}[1]{\mathcal{O}(#1)}
\DeclareMathOperator{\sgn}{sgn}
      
\newtheorem{theorem}{Theorem}[section]

\newtheorem{lemma}[theorem]{Lemma}
\newtheorem{proposition}[theorem]{Proposition}

\theoremstyle{definition}
\newtheorem{definition}[theorem]{Definition}

\begin{document}  
\title[Numerical schemes for hyperelastic
rod]{Convergent Numerical Schemes for the 
  Compressible Hyperelastic Rod Wave Equation}

\author[Cohen]{David Cohen}
\author[Raynaud]{Xavier Raynaud}            
\address[Raynaud]{\newline 
Center of Mathematics for Applications\\           
University of Oslo\\
NO--$0316$ Oslo \\
Norway}
\email{xavierra@cma.uio.no}   
\urladdr{http://folk.uio.no/xavierra/}   
\address[Cohen]{\newline Mathematisches Institut\\
Universit\"at Basel\\
CH--$4051$ Basel\\
Switzerland}
\email{david.cohen@unibas.ch}
\urladdr{http://www.math.unibas.ch/~cohen/}

          
\date{\today}   

\begin{abstract}
  We propose a fully discretised numerical scheme
  for the hyperelastic rod wave equation on the
  line. The convergence of the method is
  established. Moreover, the scheme can handle the
  blow-up of the derivative which naturally occurs
  for this equation. By using a time splitting
  integrator which preserves the invariants of the
  problem, we can also show that the scheme preserves
  the positivity of the energy density.
\end{abstract}

\maketitle

\section{Introduction}\label{sect-intro}

We consider the compressible hyperelastic rod wave equation
\begin{equation}\label{eq:hr}
  u_t-u_{xxt}+3uu_x-\gamma
  (2u_xu_{xx}+uu_{xxx})=0.
\end{equation}
The equation is obtained by Dai in \cite{dai98} as
a model equation for an infinitely long rod
composed of a general compressible hyperelastic
material. The author considers a far-field, finite
length, finite amplitude approximation for a
material where the first order dispersive terms
vanish. The function $u=u(t,x)$ represents the radial
stretch relative to a prestressed state. The
parameter $\gamma\in\R$ is a constant which
depends on the material and the prestress of the
rod and physical values lie between -29.4760 and
3.4174. For materials where first order dispersive
terms cannot be neglected, the KdV equation
\begin{equation*}
  u_t+uu_x+u_{xxx}=0
\end{equation*}
applies and only smooth solitary waves exists. In
contrast, the hyperelastic rod equation
\eqref{eq:hr} admits sharp crested solitary waves.

The Cauchy problems of the hyperelastic rod wave
equation on the line and on the circle are studied
in \cite{ConStr:00} and \cite{yin:03},
respectively. The stability of a class of solitary
waves for the rod equation on the line is
investigated in \cite{ConStr:00}. In
\cite{lenells:06}, Lenells provides a
classification of all traveling waves. In
\cite{ConStr:00,yin:03}, the authors establish,
for a special class of initial data, the global
existence in time of strong solutions. However, in
the same papers, they also present conditions on
the initial data for which the solutions blow up
and, in that case, global classical solutions no
longer exist. The way the solution blows up is
known: In the case $\gamma>0$, there is a point
$x\in\Real$ and a blow-up time $T$ for which
$\lim_{t\to T}u_x(t,x)=-\infty$ for 
(for $\gamma<0$, we have $\lim_{t\to T}u_x(t,x)=\infty$).

To handle the blow-up, weak solutions have to be
considered but they are no longer unique. For
smooth solutions, the energy $\int_\Real
(u^2+u_x^2)\,dx$ is preserved and $H^1(\Real)$ is
a natural space for studying the solutions. After
blow-up, there exist two consistent ways to
prolong the solutions, which lead to
\textit{dissipative} and \textit{conservative}
solutions. In the first case, the energy which is
concentrated at the blow-up point is dissipated
while, in the second case, the same energy is
restored. The global existence of dissipative
solution is established in \cite{CHK2}. In the
present article, we consider the conservative
solutions, whose global existence is established
in \cite{horay07}.

There are only a few works in the literature which
are concerned with numerical methods for the
hyperelastic rod wave equation. In
\cite{mayama08}, the authors consider a Galerkin
approximation which preserves a discretisation of
the energy. In \cite{cohray:10}, a
Hamiltonian-preserving numerical method and a
multisymplectic scheme are derived.  In both
works, no convergence proofs are provided and the
schemes cannot handle the natural blow-up of the
solution.

In this paper, we propose a fully discretised
numerical scheme which can compute the solution on any
finite time interval. In particular, it can
approach solutions which have locally unbounded
derivatives (the condition $u_x\in L^2(\Real)$
allows for an unbounded derivative in
$L^\infty(\Real)$). A standard space discretisation 
of \eqref{eq:hr} cannot give us global solutions. 
To obtain these solutions, we follow 
the framework given in \cite{horay07}. 
With a coordinate transformation
into Lagrangian coordinates, we first rewrite the
problem as a system of ordinary differential
equations in a Banach space
(Sections~\ref{sect-semigroup}
and~\ref{sect-bana}). We establish new decay
estimates (Section~\ref{sect-decay}) which allow
us to consider solutions defined on the whole real
line. We discretise the system of equations in
space (Section~\ref{sect-semi}) and time
(Section~\ref{sect-time}) and study the 
convergence of the numerical solution 
in Section~\ref{sect-full}. In
Section~\ref{sect-appdata}, we explain how to
define a converging sequence of initial data. This
construction can be applied to any initial data in
$H^1(\Real)$. Finally, in
Section~\ref{sect-numexp}, numerical experiments
demonstrate the validity of our theoretical
results. Moreover, the time splitting
discretisation enables the scheme to preserve
invariants and we can use this property to prove
that the scheme preserves the positivity of a
discretisation of the energy density
$u^2+u_x^2\,dx$, see Theorem \ref{th:fullinv}.

The results of this paper are also valid for the
generalised hyperelastic rod wave equation
\begin{equation}\label{eq:ghr}
u_t-u_{xxt}+\sfrac{1}{2}g(u)_x-\gamma
(2u_xu_{xx}+uu_{xxx})=0, \quad
u|_{t=0}=u_0.
\end{equation}
However, for simplicity only the
numerical discretisation of equation \eqref{eq:hr}
will be analysed.  Equation \eqref{eq:ghr} was first
introduced in \cite{CHK2}; it defines a whole
class of equations, depending on the choice of the
(locally uniformly Lipschitz) function $g$ and the
value of the parameter $\gamma$, which contains
several well-known nonlinear dispersive equations.
Taking $\gamma=1$ and $g(u)=2\kappa u+3u^2$ (with
$\kappa\geq0$), equation \eqref{eq:ghr} reduces to
the Camassa--Holm equation \cite{camassa-holm93};
For $g(u)=3u^2$, equation \eqref{eq:ghr} becomes
the hyperelastic rod wave equation \eqref{eq:hr};
For $g(u)=2u+u^2$ and for $\gamma=0$, equation
\eqref{eq:ghr} leads to the Benjamin-Bona-Mahony
(BBM) equation (or regularised long wave)
\cite{bbm72}.

\section{The Semigroup of Conservative Solutions}
\label{sect-semigroup}

The purpose of this section is to recall the main
results of \cite{horay07} about the conservative
solutions of the hyperelastic rod wave equation
\eqref{eq:hr}.  The total energy for the
hyperelastic rod wave equation is given by the
$H^1$ norm, which is preserved in time for smooth
solutions. An important feature of this equation
is that it allows for the concentration of the
energy density $(u^2+u_x^2)\,dx$ on set of zero
measure. To construct a semigroup of conservative
solution, it is necessary to keep track of the
energy when it concentrates. This justifies the
introduction of the set $\D$ defined as follows.
\begin{definition} \label{def:D}
The set $\D$ is composed of all pairs $(u,\mu)$
such that $u$ belongs to $\Hone(\R)$ and $\mu$ is a
positive finite Radon measure whose absolute
continuous part, $\muac$, satisfies
\begin{equation*}
\muac=(u^2+u_x^2)\,dx.
\end{equation*}
\end{definition}
The measure $\mu$ represents the energy density
and the set $\D$ allows $\mu$ to have a singular
part.  The solutions of \eqref{eq:hr} are
constructed via a change of coordinates, from
Eulerian to Lagrangian coordinates. An extra
variable which account for the energy is
necessary. Let us sketch this construction. We
apply the inverse Helmholtz operator
$(\id-\partial_{xx})^{-1}$ to \eqref{eq:hr} and
obtain the system of equations
\begin{subequations}
  \label{eq:hr2}
  \begin{align}
    \label{eq:hr21}
    u_t+\gamma uu_x+P_x&=0\\
    \label{eq:hr22}
    P-P_{xx}&=\sfrac{3-\gamma}{2}u^2+\sfrac{\gamma}{2}u_x^2.
  \end{align}
\end{subequations}
By using the Green function of the Helmholtz
operator, we can write $P$ in an explicit form,
i.e.,
\begin{equation}\label{eq:P}
  P(t,x)=\sfrac{1}{2}\int_\R\e^{-|x-z|}
  \bigl(\sfrac{3-\gamma}{2}u^2+\sfrac{\gamma}{2}u_x^2\bigr)
  (t,z)\,dz.
\end{equation}
We also define
\begin{equation}\label{eq:Q}
  Q(t,x):=P_x(t,x)=-\sfrac{1}{2}\int_\R\sgn(x-z)\e^{-|x-z|}
  \bigl(\sfrac{3-\gamma}{2}u^2+\sfrac{\gamma}{2}u_x^2\bigr)
  (t,z)\,dz.
\end{equation}
Next, we introduce the characteristics $y(t,\xi)$
defined as the solutions of
\begin{equation*}
  y_t(t,\xi)=\gamma u(t,y(t,\xi))
\end{equation*}
with $y(0,\xi)$ given. The variable $y(t,\xi)$
corresponds to the trajectory of a particle in the
velocity field $\gamma u$. However, the Lagrangian
velocity will be defined as
\begin{equation*}
  U(t,\xi)=u(t,y(t,\xi)).
\end{equation*}
From \eqref{eq:hr21}, we get
\begin{equation*}
  U_t(t,\xi)=u_t(t,y)+y_tu_x(t,y)=(u_t+\gamma uu_x)(t,y)=-P_x(t,y)=-Q(t,y).
\end{equation*}
As it can be checked directly from \eqref{eq:hr2},
smooth solutions satisfy the following transport
equation for the energy density:
\begin{equation}\label{eq:trp}
  \bigl(u^2+u_x^2\bigr)_t+\bigl(\gamma u(u^2+u_x^2)\bigr)_x
  =\bigl(u^3-2Pu\bigr)_x.
\end{equation}
After introducing the cumulative energy $H(t,\xi)$
as
\begin{equation*}
  \label{eq:H}  
  H(t,\xi):=\int_{-\infty}^{y(t,\xi)}(u^2+u_x^2)\,dx,
\end{equation*}
we can rewrite the transport equation
\eqref{eq:trp} as
\begin{equation*}
  H_t(t,\xi)=U^3(t,\xi)-2P(t,y)U(t,\xi).
\end{equation*}
To obtain a system of differential equations in
terms only of the Lagrangian variables
$X=(y,U,H)$, we have to express \eqref{eq:P} and
\eqref{eq:Q} in terms of these new variables. This
can be done (see \cite{horay07} for the details)
and we obtain
\begin{equation*}
  P(t,\xi)=\sfrac{1}{2}\int_\R
  \e^{-\sgn(\xi-\eta)(y(\xi)-y(\eta))}
  \Bigl(\sfrac{3-2\gamma}{2}U^2y_\xi+
  \sfrac{\gamma}{2}H_\xi\Bigr)(\eta)\,d\eta,
\end{equation*}
\begin{equation*}
  Q(t,\xi)=-\sfrac{1}{2}\int_\R\sgn(\xi-\eta)
  \e^{-\sgn(\xi-\eta)(y(\xi)-y(\eta))}
  \Bigl(\sfrac{3-2\gamma}{2}U^2y_\xi+
  \sfrac{\gamma}{2}H_\xi\Bigr)(\eta)\,d\eta.
\end{equation*}
Finally, we obtain the following system of
differential equations
\begin{subequations}
  \label{eq:genorg}
  \begin{align}
    \label{eq:genorg1}
    y_t&=\gamma U\\
    \label{eq:genorg2}
    U_t&=-Q\\
    \label{eq:genorg3}
    H_t&=U^3-2PU,
  \end{align}
\end{subequations}
which we rewrite in the compact form
\begin{equation*}
  X_t=F(X).
\end{equation*}
The derivatives of $P$ and $Q$ are given by
\begin{equation*}
  P_\xi(t,\xi)=Q(t,\xi)y_\xi
\end{equation*}
and 
\begin{equation}
  \label{eq:Qder}
  Q_\xi(t,\xi)=\sfrac{\gamma}{2}H_\xi+
  \sfrac{3-2\gamma}{2}U^2y_\xi-Py_\xi.
\end{equation}
Thus, after differentiating \eqref{eq:genorg}, we obtain
\begin{subequations}
  \label{eq:genorgder}
  \begin{align}
    \zeta_{\xi t}&=\gamma U_\xi\ (\text{or }y_{\xi t}=\gamma U_\xi),\\
    U_{\xi t}&=\frac{\gamma}{2}H_\xi+\left(\frac{3-2\gamma}{2}U^2-P\right)y_\xi,\\
    H_{\xi
      t}&=-2Q\,Uy_\xi+\left(3U^2-2P\right)U_\xi,
  \end{align}
\end{subequations}
where we denote $y(\xi)=\zeta(\xi)+\xi$.

The mapping $F$ is a mapping from $E$ to $E$, where
$E$ is a Banach space that we now define. We
denote by $V$ the space defined as
\begin{equation*}
V=\{f\in C_b(\Real) \ |\ f_\xi\in\Ltwo(\R)\},
\end{equation*}
where $C_b(\Real)=C(\Real)\cap\Linf(\R)$. The
space $V$ is a Banach space for the norm
$\norm{f}_V:=\norm{f}_\Linf+\norm{f_\xi}_\Ltwo$.
The Banach space $E$ is then defined as
\begin{equation*}
E=V\times \Hone\times V
\end{equation*} 
with norm 
$\norm{f}_E:=\norm{f}_V+\norm{f}_\Hone+\norm{f}_V$.
In \cite{horay07}, the existence of short-time
solutions of \eqref{eq:genorg} 
is established by a standard contraction
argument in $E$. The solutions of
\eqref{eq:genorg} are not in general global in
time but for initial data $(\zeta_0,U_0,H_0)$ which
belongs to the set $\F$, which we now define, they
are.
\begin{definition}
\label{def:F}
The set $\F$ consists of all $(\zeta,U,H)\in E$
such that
\begin{subequations}
\label{eq:lagcoord}
\begin{align}
\label{eq:lagcoord1}
&(\zeta,U,H)\in \left[\Winf\right]^3\quad\text{and}\quad
\lim_{\xi\rightarrow-\infty}H(\xi)=0\\
\label{eq:lagcoord2}
&y_\xi\geq0, H_\xi\geq0, y_\xi+H_\xi\geq c
\text{  almost everywhere, for some constant $c>0$}\\
\label{eq:lagcoord3}
&y_\xi H_\xi=y_\xi^2U^2+U_\xi^2\text{ almost
  everywhere.}
\end{align}
\end{subequations}
\end{definition}
The set $\F$ is preserved by the flow, that is, if
$X(0)\in\F$ and $X(t)$ is the solution to
\eqref{eq:genorg} corresponding to this initial
value, then $X(t)\in\F$ for all time $t$. The
properties of the set $\F$ can then be used to
establish apriori estimates on the solutions and
show that they exit globally in time, see
\cite{horay07} for more details. We denote by
$S_t$ the semigroup of solutions in $\F$ given by
the solutions of \eqref{eq:genorg}.

Given an initial data $(u,\mu)\in\D$, we have to
find the corresponding initial data in $\F$; we
have to define a mapping between Eulerian and
Lagrangian variables. To do so, we set
\begin{subequations}
\label{eq:Ldef}
\begin{align}
\label{eq:Ldef1}
y(\xi)&=\sup\left\{y\ |\ \mu((-\infty,y))+y<\xi\right\},\\
\label{eq:Ldef2}
H(\xi)&=\xi-y(\xi),\\
\label{eq:Ldef3}
U(\xi)&=u\circ{y(\xi)}.
\end{align}
\end{subequations}
We define $X=L(u,\mu)$ and $L$ maps Eulerian to
Lagrangian variables. When $\mu=\muac$ (no energy
is concentrated), equation \eqref{eq:Ldef1}
simplifies and we get
\begin{equation*}
  y(\xi)+\int_{-\infty}^{y(\xi)}(u^2+u_x^2)(x)\,dx=\xi.
\end{equation*}
Reciprocally, we define the mapping $M$ from
Lagrangian to Eulerian variables: Given
$X=(y,U,H)\in\F$, we recover $(u,\mu)=M(X)\in\D$
by setting
\begin{subequations}
\label{eq:umudef}
\begin{align}
\label{eq:umudef1}
&u(x)=U(\xi)\text{ for any }\xi\text{ such that  }  x=y(\xi),\\
\label{eq:umudef2}
&\mu=y_\#(H_\xi\,d\xi).
\end{align}
\end{subequations}
Here, $y_\#(H_\xi\,d\xi)$ denotes the push-forward
of the measure $H_\xi\,d\xi$ by the mapping $y$.

In conclusion, the construction of the global
conservative solutions is based on the change of
variable from Eulerian to Lagrangian. However,
this change of variable is not bijective. The
discrepancy between the two sets of variables is
due to the freedom of relabeling in Lagrangian
coordinates. The relabeling functions can be
identified as a group, which basically consists of
the diffeomorphisms of the line with some
additional assumptions (see \cite{horay07}). Given
$X=(y,U,H)$, the element $X\circ f=(y\circ
f,U\circ f,H\circ f)$ is called the relabeled
version of $X$ with respect to the relabeling
function $f$. We can check that
\begin{equation*}
  M(X)=M(X\circ f),
\end{equation*}
that is, several configurations in Lagrangian
variables correspond to the same Eulerian
configuration. In this article, we will not be too
concerned with the aspects of relabeling
invariance. They have however to be taken into
account to establish the semigroup property of the
semigroup of solutions in Eulerian variables. We
also use the relabeling invariance in Section
\ref{sect-numexp} to construct initial data in a
convenient way for some particular initial
conditions. We define the semigroup of solutions
$T_t$ in Eulerian coordinates as
\begin{equation}
  \label{eq:defTt}
  T_t=M\circ S_t\circ L.
\end{equation}
Finally, we recall the following main result from
\cite{horay07}.
\begin{theorem}
  \label{th:main} The mapping
  $T\colon\D\times\Real_+\to\D$, where $\D$ is
  defined by Definition \ref{def:D}, defines a
  continuous semigroup of conservative solutions
  of the hyperelastic rod wave equation
  \eqref{eq:hr}, that is, given $(\bar
  u,\bar\mu)\in\D$, if we denote by
  $t\mapsto(u(t),\mu(t))=T_t(\bar u,\bar\mu)$ the
  corresponding trajectory, then $u$ is a weak
  solution of the hyperelastic rod wave equation
  \eqref{eq:hr2}.
\end{theorem}




The function $y(t,\xi)$ gives the trajectory of a
particle which evolves in the velocity field given
by $\gamma u(t,x)$. If $u$ is smooth, then it is
Lipschitz in the second variable and the mapping
$\xi\to y(t,\xi)$ remains a diffeomorphism. We
denote its inverse by $x\to y^{-1}(t,x)$. In this
case, the density $\rho(t,x)$ is given by
\begin{equation}
  \label{eq:defrho}
  \rho(t,x)=\frac1{y_\xi(t,y^{-1}(t,x))}.
\end{equation}
We can also recover the energy density as
\begin{equation}
  \label{eq:defendens}
  (u^2+u_x^2)(t,x)=\frac{H_\xi}{y_\xi}(t,y^{-1}(t,x))).
\end{equation}
In the following sections, we design numerical schemes which
preserve the positivity of the particle and energy
densities as defined in \eqref{eq:defrho} and
\eqref{eq:defendens}.

\section{Equivalent System of ODEs in a Banach
  space}\label{sect-bana}

In this section, we reformulate the hyperelastic
rod wave equation \eqref{eq:hr2} as a system of
ordinary differential equations in a Banach space
as this was done in \cite{horay07} but where we
decouple the functions $y$, $U$ and $H$ and their
derivatives $y_\xi$, $U_\xi$ and $H_\xi$ that 
we denote by $q$, $w$ and $h$. We set
$\zeta(t,\xi):=y(t,\xi)-\xi$ and
$v(\xi,1)=q(t,\xi)-1$. If we assume that
\begin{equation}
  \label{eq:assumder}
  q=y_\xi,\quad w=U_\xi\,\text{ and } h=H_\xi
\end{equation}
then \eqref{eq:genorg} and \eqref{eq:genorgder}
rewrite
\begin{subequations}
  \label{eq:sysban}
  \begin{align}
    \label{eq:sysban1}
    \zeta_t=y_t&=\gamma U,\\
    \label{eq:sysban2}
    U_t&=-Q,\\
    \label{eq:sysban3}
    H_t&=U^3-2PU,\\
    \label{eq:sysban4}
    v_t=q_t&=\gamma w,\\
    \label{eq:sysban5}
    w_t&=\sfrac{\gamma}{2}h+
    \bigl(\sfrac{3-2\gamma}{2}U^2-P\bigr)q,\\ 
    \label{eq:sysban6}
    h_t&=-2QUq+\bigl(3U^2-2P\bigr)w,
  \end{align}
\end{subequations}
where $P$ and $Q$ are given by 
\begin{equation}\label{eq:P3}
  P=\sfrac{1}{2}\int_\R
  \e^{-\sgn(\xi-\eta)(y(\xi)-y(\eta))}
  \Bigl(\sfrac{3-2\gamma}{2}U^2q+
  \sfrac{\gamma}{2}h\Bigr)(\eta)\,d\eta
\end{equation}
and 
\begin{equation}\label{eq:Q3}
  Q=-\sfrac{1}{2}\int_\R\sgn(\xi-\eta)
  \e^{-\sgn(\xi-\eta)(y(\xi)-y(\eta))}
  \Bigl(\sfrac{3-2\gamma}{2}U^2q+
  \sfrac{\gamma}{2}h\Bigr)(\eta)\,d\eta.
\end{equation}
Note that equation \eqref{eq:sysban} is semilinear
in the variables $(q,w,h)$. Since the terms $P$ and $Q$
have similar structure, in the remaining of the
paper most of the proofs will be
established just for one of them. Now, we do not
require \eqref{eq:assumder} to hold any longer
and, setting $Y:=(\zeta,U,H,v,w,h)$, we obtain the
system of differential equations
\begin{equation*}
Y_t(t)=G(Y(t)),
\end{equation*}
where $G$ is defined by \eqref{eq:sysban}. In the
remaining, we will sometimes abuse the notation
and write $Y=(y,U,H,q,w,h)$ instead of
$Y=(\zeta,U,H,v,w,h)$. Then, we implicitly assume
the relations $y(\xi)=\zeta(\xi)+\xi$ and
$q=v+1$. The variables $y$ and $q$ are the
physical ones but do not have the proper
decay/boundedness properties at infinity and this
is why $\zeta$ and $v$ have to be introduced. The
system \eqref{eq:sysban} is defined in the Banach
space $F$, where $F$ is given by
\begin{equation*}
F:=L^\infty(\R)\times \bigl(L^\infty(\R)\cap L^2(\R)\bigr)
\times L^\infty(\R)\times L^2(\R)\times L^2(\R) \times
L^2(\R).
\end{equation*}
For any $Y=(\zeta,U,H,v,w,h)\in F$ we use the
following norm on $F$:
\begin{equation*}
  \norm{Y}_F=\norm{\zeta}_{L^\infty}+\norm{U}_{L^2}+\norm{U}_{L^\infty}+\norm{H}_{L^\infty}+
  \norm{v}_{L^2}+\norm{w}_{L^2}+\norm{h}_{L^2}.
\end{equation*}
The following proposition holds.
\begin{proposition}
  \label{prop:LipG}
  The mappings $P:F\to\Hone(\R)$ and
  $Q:F\to\Hone(\R)$ belongs to $C^1(F,H^1(\Real))$
  and $G:F\to F$ belongs to $C^1(F,F)$. Moreover,
  given $M>0$, let
    \begin{equation*}
      B_M=\{X\in F\ |\ \norm{X}_F\leq M\}.
    \end{equation*}
    There exists a constant $C(M)$ which only
    depends on $M$ such that
    \begin{equation}
      \label{eq:bdderPQ}
      \norm{P(Y)}_{H^1}+\norm{Q(Y)}_{H^1}+\norm{\fracpar{P}{Y}(Y)}_{L(F,H^1)}+\norm{\fracpar{Q}{Y}(Y)}_{L(F,H^1)}\leq C(M)
    \end{equation}
    and
    \begin{equation}
      \label{eq:bdderG}
      \norm{G(Y)}_{F}+\norm{\fracpar{G}{Y}(Y)}_{L(F,F)}\leq C(M)
    \end{equation}
    for all $Y\in B_M$.
\end{proposition}

Here, abusing slightly the notations, we denote by
the same letter $P$ the function $P(t,\xi)$ and
the mapping $Y\mapsto P$. The same holds for
$Q$. The norms $L(F,H^1(\Real))$ and $L(F,F)$ are
the operator norms.

\begin{proof}
  First we prove that the mappings $Y\mapsto P$
  and $Y\mapsto Q$ as given by \eqref{eq:P3} and
  \eqref{eq:Q3} belong to $C^1(F,L^\infty(\R)\cap
  L^2(\R))$. We rewrite $Q$ as
  \begin{align}
    \notag Q(X)(\xi)&= -\sfrac{\e^{-\zeta(\xi)}}{
      2}\int_\Real\chi_{\{\eta<\xi\}} (\eta)
    \e^{-(\xi-\eta)}\e^{\zeta(\eta)}\\ \notag
    &\qquad\qquad\times
    \Bigl(\sfrac{3-2\gamma}{2}U^2q+
    \sfrac{\gamma}{2}h\Bigr)(\eta)(\eta)\,d\eta\\
    \notag
    &\quad+\sfrac{\e^{\zeta(\xi)}}{2}
    \int_\Real\chi_{\{\eta>\xi\}}(\eta)\e^{(\xi-\eta)}
    \e^{-\zeta(\eta)}\\
    \label{eq:Qsum}
    &\qquad\qquad\times\Bigl(\sfrac{3-2\gamma}{2}U^2q+
    \sfrac{\gamma}{2}h\Bigr)(\eta)\,d\eta,
  \end{align}
  where $\chi_B$ denotes the indicator function of
  a given set $B$. We decompose $Q$ into the sum
  $Q_1+Q_2$, where $Q_1$ and $Q_2$ are the
  operators corresponding to the two terms in the 
  sum on the
  right-hand side of \eqref{eq:Qsum}. Let
  $h(\xi)=\chi_{\{\xi>0\}}(\xi)\e^{-\xi}$ and $A$
  be the map defined by $A\colon v\mapsto h\star
  v$. Then, $Q_1$ can be rewritten as
  \begin{equation}
    \label{eq:Q1def}
    Q_1=-\sfrac{\e^{-\zeta(\xi)}}{2}A\circ
    R(Y)(\xi),
  \end{equation}
  where $R$ is the operator from $F$ to $L^2(\R)$
  given by 
  \begin{equation*}
    R(Y)(\xi)=\e^{\zeta(\xi)}\Bigl(\sfrac{3-2\gamma}{2}U^2(1+v)+
    \sfrac{\gamma}{2}h\Bigr)(\xi).
  \end{equation*}
  The mapping $A$ is a continuous linear mapping
  from $L^2(\R)$ into $L^2(\R)\cap L^\infty(\R)$ as, from
  Young inequalities, we have
  \begin{equation}
    \label{eq:young1}
    \norm{h\star v}_{L^2}\leq\norm{h}_{L^1}\norm{v}_{L^2}
    \text{ and }
    \norm{h\star v}_{L^\infty}\leq\norm{h}_{L^2}\norm{v}_{L^2}.
  \end{equation}
  For any $Y\in B_M$, we have
  \begin{align*}
    \norm{Q_1}_{L^2\cap L^\infty}\leq C(M)
    \norm{A\circ R}_{L^2\cap L^\infty}\leq
    C(M)\norm{R}_{L^2}\leq C(M)
  \end{align*}
  for some constant $C(M)$ which depends only on
  $M$. From now on, we denote generically by
  $C(M)$ such constant even if its value may
  change from line to line. The same result holds
  for $Q$ and $P$. Since $R$ is composed of sums
  and products of $C^1$ maps, the fact that
  $R\colon F\rightarrow L^2$ is $C^1$ follows
  directly from the following short lemma whose
  proof is essentially the same as the proof of
  the product rule for derivatives in $\Real$.
  \begin{lemma}  
    \label{lem:Blip}
    Let $1\leq p\leq\infty$. If $K_1\in C^1(F,\Linf(\R))$
    and $K_2\in C^1(F,L^p(\R))$, then the product
    $K_1K_2$ belongs to $C^1(F,L^p(\R))$ and
    \begin{equation*}
      \fracpar{(K_1K_2)}{Y}(Y)[\bar Y]=K_1(Y)\fracpar{K_2}{Y}(Y)[\bar Y]+K_2(Y)\fracpar{K_1}{Y}(Y)[\bar Y].
    \end{equation*}
  \end{lemma}
  With this lemma in hands, we thus obtain that
  \begin{equation*}
    \fracpar{R}{Y}(Y)[\bar Y]=\e^{\zeta}\left(\frac{3-2\gamma}{2}\left(\bar\zeta U^2(1+v)+2U\bar U(1+v)+U^2\bar v\right)+\frac{\bar\zeta\gamma}2h+\frac{\gamma}2\bar h\right)
  \end{equation*}
  and
  \begin{equation*}
    \norm{\fracpar{R}{Y}(Y)}_{L(F,L^2)}\leq C(M).
  \end{equation*}
  Then, $Q_1$ is in $C^1(F,L^2(\R)\cap L^\infty(\R))$,
  \begin{equation*}
    \fracpar{Q_1}{Y}(Y)[\bar Y]=\frac{\e^{-\zeta}}{2}(\bar\zeta A(R(Y))-A(\fracpar{R}{Y}(Y)[\bar Y]))
  \end{equation*}
  and
  \begin{equation*}
    \norm{\fracpar{Q_1}{Y}(Y)}_{L(F,L^2\cap L^\infty)}\leq C(M).
  \end{equation*}
  We obtain the same result for $Q_2$, $Q$ and
  $P$. We differentiate $Q$ and get
  \begin{equation}
    \label{eq:diffQ}
    Q_\xi=\sfrac{\gamma}{2}h+
    \sfrac{3-2\gamma}{2}U^2q-Pq,
  \end{equation}
  see \eqref{eq:Qder}. Hence, the mapping
  $Y\mapsto Q_\xi$ is differentiable,
  \begin{equation*}
    \fracpar{Q_\xi}{Y}(Y)[\bar Y]=\frac{\gamma}2\bar h+\frac{3-2\gamma}{2}(2U\bar U+U^2\bar q)-\fracpar{P}{Y}(Y)[\bar Y]q-P\bar q,
  \end{equation*}
  and
  \begin{equation*}
    \norm{\fracpar{Q_\xi}{Y}(Y)}_{L(F,L^2)}\leq C(M).
  \end{equation*}
  It follows that $Q$ belongs to $C^1(F,H^1(\R))$
  and $\norm{\fracpar{Q}{Y}}_{L(F,H^1)}\leq
  C(M)$. The same result holds for $P$ and
  \eqref{eq:bdderPQ} is proved. By using Lemma
  \ref{lem:Blip}, we get that $G\in C^1(F,F)$ and
  this proves \eqref{eq:bdderG}.
\end{proof}

By using Proposition \ref{prop:LipG} and the
standard contraction argument, we prove the
existence of short-time solutions to
\eqref{eq:sysban}:
\begin{theorem}\label{th:exist}
  For any initial values
  $Y_0=(\zeta_0,U_0,H_0,v_0,w_0,h_0) \in F$, there
  exists a time $T$, only depending on the norm of
  the initial values, such that the system of
  differential equations \eqref{eq:sysban} admits
  a unique solution in $C^1([0,T],F)$. Moreover,
  for any two solutions $Y_1$ and $Y_2$ such that
  $\sup_{t\in[0,T]}\norm{Y_1(t)}_F\leq M$ and
  $\sup_{t\in[0,T]}\norm{Y_2(t)}_F\leq M$, then
  \begin{equation}
    \label{eq:stabor}
    \sup_{t\in[0,T]}\norm{Y_1(t)-Y_2(t)}_F\leq C(M)\norm{Y_1(0)-Y_2(0)}_F,
  \end{equation}
  where the constant $C(M)$ depends only on $M$.
\end{theorem}
\begin{proof}
The stability result \eqref{eq:stabor} is a direct
application of Proposition \ref{prop:LipG} and
Gronwall's Lemma. 
\end{proof}
The system of differential equations \eqref{eq:sysban} 
in the Banach space
$F$ has an interesting geometric property: it
possesses an invariant. In fact, the following
quantity
\begin{equation*}
  I(Y):=U^2q^2+w^2-qh
\end{equation*}
is conserved along the exact solution of the
problem as we show now.  For any $Y(t)$ solution
of \eqref{eq:sysban}, we have
\begin{equation}
  \label{eq:compIpres}
  \begin{aligned}
    \sfrac{d}{dt}I(Y(t))&=
    2UU_tq^2+2U^2qq_t+2ww_t-q_th-qh_t=
    -2UQq^2+2U^2q\gamma w\\
    &+2w\Bigl(\sfrac{\gamma}{2}h+
    \bigl(\sfrac{3-2\gamma}{2}U^2-P\bigr)q\Bigr)-\gamma
    wh -q\bigl(-2QUq+(3U^2-2P)\bigr)=0.
  \end{aligned}
\end{equation}
Additionally, we have
\begin{lemma}\label{lem:presprop}
  The following properties are preserved
  (independently one of each other) by the
  governing equations \eqref{eq:sysban}
  \begin{enumerate}
  \item[(i)] $q$, $w$, $h$ belongs to $L^\infty(\R)$.
  \item[(ii)] $qh=U^2q^2+w^2$ (or $I(Y)=0$).
  \item[(iii)] $qh=U^2q^2+w^2$ (or $I(Y)=0$) and
    $q\geq0,\ h\geq0,\ q+h\geq c$ almost
    everywhere for some constant $c>0$.
  \item[(iv)]The functions $y,U$ and $H$ are
    differentiable and $y_\xi=q$, $U_\xi=w$ and
    $H_\xi=h$.
  \end{enumerate}
\end{lemma}
\begin{proof} We consider the short time solution
  given by Theorem~\ref{th:exist} on the time interval
  $[0,T]$. Given an initial data $Y_0$ which
  satisfies (i), then we use Gronwall's Lemma and
  the semi-linearity of
  \eqref{eq:sysban4}-\eqref{eq:sysban6} with
  respect to $q$, $w$, $h$ and obtain that
  \begin{equation*}
    \norm{q(t)}_{L^\infty}+\norm{w(t)}_{L^\infty}+\norm{h(t)}_{L^\infty}\leq C(\norm{q_0}_{L^\infty}+\norm{w_0}_{L^\infty}+\norm{h_0}_{L^\infty})
  \end{equation*}
  for a constant $C$ which only depends on
  $\norm{Y_0}_F$. We have already seen that
  $I(Y(t))$ is preserved, so that the condition
  (ii) is satisfied. We consider a fixed
  $\xi$. Let us denote $\bar T=\sup\{t\in[0,T]\ |\
  (q+h)(\bar t,\xi)>0\text{ for all } \bar
  t\in[0,t]\}$. Since
  \begin{equation}
    \label{eq:presI}
    qh=q^2U^2+w^2
  \end{equation}
  for all time $t\in[0,T]$, the product $qh$ is
  positive and therefore $q\geq0$ and $h\geq 0$
  for $t\in[0,\bar T]$. From the governing
  equations \eqref{eq:sysban}, we get
  \begin{align*}
    \sfrac{d}{dt}\big(\sfrac{1}{q+h}\big)\leq C\sfrac{\abs{w}+q}{(q+h)^2}
  \end{align*}
  for some constant $C$ which depends only on
  the norm of the initial data. Since $q$ and
  $h$ are positive, it follows from
  \eqref{eq:presI} that $\abs{w}\leq
  \sfrac12(q+h)$ and therefore
  $\sfrac{d}{dt}
  \big(\sfrac{1}{q+h}\big)\leq
  \sfrac{C}{q+h}$. Gronwall's inequality yields
  $\sfrac{1}{q+h}(t)\leq\sfrac{1}{q_0+h_0}\e^{C\bar
    T}$. Hence, we have $\bar T=T$ and
  \begin{equation*}
    \e^{-CT}(q_0+h_0)\leq (q+h)(t). 
  \end{equation*}
  In particular it implies that $\bar T=T$ and
  there exists a constant $c>0$ such that $q+h\geq c$
  for almost every $\xi$ and $t\in[0,T]$. Thus we
  have proved that $Y(t)$ satisfies the condition
  (iii). The last property follows from
  \cite{horay07}, where it is proved that there
  exists a unique solution to the system
  \begin{equation}
    \label{eq:syseq1}
    \begin{aligned}
      \zeta_t&=\gamma U,\\
      U_t&=-Q,\\
      H_t&=U^3-2PU
    \end{aligned}
  \end{equation}
  in the Banach space $V\times H^1(\R)\times V$, where
  $V=\{f\in L^\infty(\R)\ |\ f_\xi\in L^2(\R)\}$. By
  differentiation of \eqref{eq:syseq1}, we obtain
  that $(y,U,H,y_\xi,U_\xi,H_\xi)(t)$ satisfies
  \eqref{eq:sysban}. Therefore by uniqueness of
  the solutions, for an initial data $Y_0$ which
  satisfies (iii), we obtain that the property
  (iv) is satisfied.
\end{proof}

Having a closer look at Lemma~\ref{lem:presprop}, 
we now define the following set.

\begin{definition}
  \label{def:G} The set $\G$ consists of the
  elements $(y,U,H,q,w,h)\in F$ which satisfy the
  conditions (i), (iii) and (iv).
\end{definition}

As a consequence of Lemma \ref{lem:presprop}, the
set $\G$ is preserved by the system. For any
initial data in $\G$, the solution of
\eqref{eq:sysban} coincide with the solutions that
are obtained in \cite{horay07}. In particular, we
prove in the same way as in \cite{horay07} that
\begin{theorem}
  \label{th:globsol} For initial data in $\G$,
  the solutions to \eqref{eq:sysban} are global in
  time.
\end{theorem}
We denote by $S_t$ the semigroup of solutions to
\eqref{eq:sysban} in $\G$. Thus, we slightly abuse
the notations, as $S_t$ was already introduced in
the previous section, see \eqref{eq:defTt}, but,
as we explained, the two semigroups are
essentially the same. Note that global existence
can only be established for initial data in $\G$
and do not hold in general for initial data in
$F$.

\section{Decay at infinity}
\label{sect-decay}

The terms $P$ and $Q$, as given by \eqref{eq:P3}
and \eqref{eq:Q3} which appear in the governing
equations \eqref{eq:sysban} are global in the
sense that they are not compactly supported even
if $Y$ is. Consequently the set of compactly
supported functions is not preserved by the
system. However, we identify in this section
decay properties which are preserved by the
system. We denote by $F^e$, the subspace of $F$ of
functions with exponential decay defined as
\begin{equation*}
  F^e=\{Y\in F\ |\ q,w,h\in L^\infty(\R),\quad 
  \e^{\abs{\xi}}U,\,\e^{\abs{\xi}}w\in L^2(\Real),\, \e^{\abs{\xi}}h\,\in L^1(\R)\}.
\end{equation*}
We define the following norm on $F^e$
\begin{multline*}
  \norm{Y}_{F^e}=\norm{Y}_F+\norm{q}_{L^\infty}+\norm{w}_{L^\infty}+\norm{h}_{L^\infty}\\
  +\snorm{\e^{\frac{\abs{\xi}}2}U}_{L^2}+
  \snorm{\e^{\frac{\abs{\xi}}2}w}_{L^2}+\norm{\e^{\abs{\xi}}h}_{L^1}.
\end{multline*}
Given $\alpha>1$, we denote by $F^{\alpha}$, the
subspace of $F$ of functions with polynomial decay
defined as
\begin{multline*}
  F^\alpha=\{Y\in F\ |\ q,w,h\in
  L^\infty(\R),\quad
  (1+\abs{\xi})^{\frac{\alpha}2}
  U,\,(1+\abs{\xi})^{\frac{\alpha}2} w\in L^2(\Real),\\
  (1+\abs{\xi})^\alpha h\,\in L^1(\R)\}.
\end{multline*}
We define the following norm on $F^\alpha$
\begin{multline*}
  \norm{Y}_{F^\alpha}=\norm{Y}_F+\norm{q}_{L^\infty}+\norm{w}_{L^\infty}+\norm{h}_{L^\infty}\\
  +\snorm{(1+\abs{\xi})^{\frac{\alpha}2}U}_{L^2}+
  \snorm{(1+\abs{\xi})^{\frac{\alpha}2}w}_{L^2}+\snorm{(1+\abs{\xi})^\alpha
    h}_{L^1}.
\end{multline*}
\begin{theorem}
  \label{th:decay} The spaces $F^e$ and $F^\alpha$
  are preserved by the flow of \eqref{eq:sysban}. Considering the short-time
  solutions given by Theorem \ref{th:exist}, we
  have that
  \begin{enumerate}
  \item[(i)] If $Y_0\in F^e$, then
    $\sup_{t\in[0,T]}\norm{Y(t,\cdot)}_{F^e}\leq
    C$,\\
  \item[(ii)] If $Y_0\in F^\alpha$, then
    $\sup_{t\in[0,T]}\norm{Y(t,\cdot)}_{F^\alpha}\leq
    C$,
  \end{enumerate}
  for a constant $C$ which only depends on $T$ and
  $\norm{Y_0}_{F^e}$ (case (i)) or $T$ and
  $\norm{Y_0}_{F^\alpha}$ (case (ii)).
\end{theorem}
\begin{proof} Let us prove the case (i). First, we
  establish $L^1$ bounds on the solutions. By
  applying the Cauchy--Schwartz inequality, we get
  \begin{equation*}
    \int_\Real \abs{U_0(\xi)}\,d\xi=
    \int_\Real \e^{-\abs{\sfrac{\xi}2}}
\e^{\abs{\sfrac{\xi}2}}\abs{U_0(\xi)}\,
d\xi\leq\sqrt{2}\norm{\e^{\abs{\xi}}U_0^2(\xi)}_{L^1}^{\sfrac12},
  \end{equation*}
  which implies that $U_0\in L^1(\R)$ and
  $\norm{U_0}_{L^1}\leq C$ for some constant $C$
  which depends only on
  $\norm{\e^{\abs{\xi}}U_0^2(\xi)}_{L^1}$. Similarly
  we get that $w_0\in L^1(\R)$ and
  $\norm{w_0}_{L^1}\leq C$ for some constant $C$
  which depends only on
  $\norm{\e^{\abs{\xi}}w_0^2(\xi)}_{L^1}$. We
  denote generically by $C$ such a constant, which
  depends only on $T$ and $\norm{Y_0}_{F^e}$.
  From Theorem \ref{th:exist} and Lemma
  \ref{lem:presprop}, we get that
  \begin{equation*}
    \norm{q(t,\cdot)}_{L^\infty}+\norm{w(t,\cdot)}_{L^\infty}+\norm{h(t,\cdot)}_{L^\infty}\leq C.
  \end{equation*}
  By following the same argument as in the proof
  of Proposition \ref{prop:LipG}, from \eqref{eq:Q1def}
  to \eqref{eq:young1}, but, instead, using the
  Young inequality $\norm{\kappa\star
    r}_{L^1}\leq\norm{\kappa}_{L^1}\norm{r}_{L^1}$,
  we obtain that
  \begin{equation}
    \label{eq:estL1Q}
    \norm{Q(t,\cdot)}_{L^1}\leq C(\norm{h(t,\cdot)}_{L^1}+1)
  \end{equation}
  for a constant $C$ which depends only on
  $\norm{Y(t)}_{F^e}$ and, therefore, only on
  $\norm{Y_0}_{F^e}$ and $T$. The same estimate
  holds for $P$, that is,
  \begin{equation}
    \label{eq:estL1P}
    \norm{P(t,\cdot)}_{L^1}\leq C(\norm{h(t,\cdot)}_{L^1}+1).
  \end{equation}
  Let us denote
  \begin{equation*}
    J(t):=\norm{U(t,\cdot)}_{L^1}+
    \norm{w(t,\cdot)}_{L^1}+\norm{h(t,\cdot)}_{L^1}.
  \end{equation*}
  From the governing equations \eqref{eq:sysban},
  after using \eqref{eq:estL1Q} and
  \eqref{eq:estL1P}, we get
  \begin{equation*}
    J(t)\leq J(0)+C+C\int_{0}^t J(\tau)\,d\tau.
  \end{equation*}
  Hence, by applying Gronwall's Lemma, we get
  that, for $t\in[0,T]$,
  \begin{equation}
    \label{eq:L1bound}
    J(t)=\norm{U(t,\cdot)}_{L^1}+
    \norm{w(t,\cdot)}_{L^1}+\norm{h(t,\cdot)}_{L^1}\leq C
  \end{equation}
  for another constant $C$. Let $L(t)$ denotes
  \begin{equation}
    \label{eq:defI}
    L(t)=\norm{\e^{\abs{\xi}}U^2(t,\cdot)}_{L^1}
    +\norm{\e^{\abs{\xi}}w^2(t,\cdot)}_{L^1}+
    \norm{\e^{\abs{\xi}}h(t,\cdot)}_{L^1}.
  \end{equation}
  From the definition of $Q$, we get that
  \begin{equation}
    \label{eq:grosestQ}
    Q(t,\xi)\leq C \int_\Real \e^{-\abs{\xi-\eta}}(U^2+h)(t,\eta)\,d\eta
  \end{equation}
  so that
  \begin{align*}
    \e^{\abs{\xi}}Q(t,\xi)&\leq C \int_\Real \e^{\abs{\xi}}\e^{-\abs{\xi-\eta}}\e^{-\abs{\eta}}\e^{\abs{\eta}}(U^2+h)(t,\eta)\,d\eta\\
    &\leq C L(t)
  \end{align*}
  because $\abs{\xi}-\abs{\eta}\leq\abs{\xi-\eta}$
  and therefore
  \begin{equation}
    \label{eq:decexpQ}
    \norm{\e^{\abs{\xi}} Q(t,\cdot)}_{L^\infty}\leq C L(t).
  \end{equation}
  Similarly, we get that 
  \begin{equation}
    \label{eq:decexpP}
    \norm{\e^{\abs{\xi}}
      P(t,\cdot)}_{L^\infty}\leq C L(t).
  \end{equation}
  From the governing equations \eqref{eq:sysban},
  we get that
  \begin{align}
    \notag
    \norm{\e^{\abs{\xi}}U^2(t,\xi)}_{L^1}&\leq
    \norm{\e^{\abs{\xi}}U_0^2}_{L^1}+
    \int_0^t\norm{2\e^{\abs{\xi}}QU(\tau,\cdot)}_{L^1}
    \,d\tau\\
    \notag &\leq
    \norm{\e^{\abs{\xi}}U_0^2}_{L^1}+2\int_0^t
    \norm{\e^{\abs{\xi}}Q(\tau,\cdot)}_{L^\infty}
    \norm{U(\tau,\cdot)}_{L^1}\,d\tau\\
    \label{eq:estidecu2}
    &\leq
    \norm{\e^{\abs{\xi}}U_0^2}_{L^1}+C\int_0^t
    I(\tau)\,d\tau,
  \end{align}
  by using the $L^1$ apriori estimates 
  \eqref{eq:L1bound} and \eqref{eq:decexpQ}. From
  \eqref{eq:sysban}, we also obtain that
  \begin{align*}
    \norm{\e^{\abs{\xi}}h(t,\xi)}_{L^1}&\leq
    \norm{\e^{\abs{\xi}}h_0}_{L^1}+\int_0^t
    (2\norm{\e^{\abs{\xi}}Q(\tau,\cdot)}_{L^\infty}
    \norm{U(\tau,\cdot)}_{L^1})\,d\tau+\\
    &\quad \int_0^t
    (C\norm{\e^{\abs{\xi}}U^2(\tau,\cdot)}_{L^1}+
    \norm{\e^{\abs{\xi}}P(\tau,\cdot)}_{L^\infty}
    \norm{w(\tau,\cdot)}_{L^1})\,d\tau
  \end{align*}
  which, after using the $L^1$ estimates
  \eqref{eq:L1bound}, \eqref{eq:decexpQ} and
  \eqref{eq:decexpP}, yields
  \begin{equation}
    \label{eq:estidech}
    \norm{\e^{\abs{\xi}}h(t,\xi)}_{L^1}\leq\norm{\e^{\abs{\xi}}h_0}_{L^1}+C+
    C\int_0^tI(\tau)\,d\tau.
  \end{equation}
  Similarly we get that
  \begin{align}
    \notag \norm{\e^{\abs{\xi}}w^2(t,\xi)}_{L^1}&
    \leq\norm{\e^{\abs{\xi}}w_0^2}_{L^1}+
    \int_0^t\sfrac\gamma2\norm{\e^{\abs{\xi}}h(\tau,\xi)}_{L^1}
    \,d\tau\\
    \notag &\quad
    +\int_{0}^t(C\norm{\e^{\abs{\xi}}U^2(\tau,\cdot)}_{L^1}+
    C\norm{\e^{\abs{\xi}}P(\tau,\cdot)}_{L^\infty}
    \norm{w(\tau,\cdot)}_{L^1})\,d\tau.\\
    \label{eq:estidecw2}
    &\leq\norm{\e^{\abs{\xi}}w_0^2}_{L^1}+
    C+C\int_0^tI(\tau)\,d\tau.
  \end{align}
  After summing \eqref{eq:estidecu2},
  \eqref{eq:estidech} and \eqref{eq:estidecw2}, we
  get $L(t)\leq L(0)+C+C\int_{0}^tL(\tau)\,d\tau$
  and the result follows by applying Gronwall's
  inequality. We now turn to case (ii). We
  introduce the quantity
  \begin{equation*}
    K(t)=\snorm{(1+\abs{\xi})^\alpha U^2(t,\cdot)}_{L^1}
    +\snorm{(1+\abs{\xi})^\alpha w^2(t,\cdot)}_{L^1}+
    \snorm{(1+\abs{\xi})^\alpha h(t,\cdot)}_{L^1}.
  \end{equation*}
  From \eqref{eq:grosestQ}, we get
  \begin{equation}
    \label{eq:estxialQinf}
    (1+\abs{\xi})^\alpha Q\leq
    C\int_\Real(1+\abs{\xi})^\alpha\e^{-\abs{\xi-\eta}}(1+\abs{\eta})^{-\alpha}(1+\abs{\eta})^{\alpha}(U^2q+h)\,d\eta.
  \end{equation}
  Since
  $\abs{\xi}\leq\abs{\xi-\eta}+\abs{\eta}\leq(1+\abs{\xi-\eta})(1+\abs{\eta})$,
  we have $(1+\abs{\xi})\leq
  2(1+\abs{\xi-\eta})(1+\abs{\eta})$ and
  \begin{equation}
    \label{eq:ineqal}
    (1+\abs{\xi})^\alpha\leq
    2^\alpha(1+\abs{\xi-\eta})^\alpha(1+\abs{\eta})^\alpha.     
  \end{equation}
  Then, it follows from \eqref{eq:estxialQinf}
  that
  \begin{align}
    \notag
    (1+\abs{\xi})^\alpha Q&\leq
    C\int_\Real\e^{-\abs{\xi-\eta}}(1+\abs{\xi-\eta})^{\alpha}(1+\abs{\eta})^{\alpha}(U^2q+h)\,d\eta\\
    \label{eq:estQKt}
    &\leq
    C\norm{\e^{-z}(1+\abs{z})^\alpha}_{L^\infty}K(t)\leq
    CK(t)
  \end{align}
  so that $\norm{(1+\abs{\xi})^\alpha
    Q}_{L^\infty}\leq CK(t)$. We have to
  estimate
  $\norm{(1+\abs{\xi})^{\alpha}Q}_{L^1}$. We have
  \begin{align}
    \notag
    \norm{(1+\abs{\xi})^{\alpha}Q}_{L^1}&\leq\int_{\Real^2}(1+\abs{\xi})^\alpha\e^{-\abs{\xi-\eta}}(1+\abs{\eta})^{-\alpha}(1+\abs{\eta})^{\alpha}(U^2q+h)\,d\eta
    d\xi\\
    \notag&=\int_{\Real^2}(1+\abs{\eta+z})^\alpha\e^{-\abs{z}}(1+\abs{\eta})^{-\alpha}(1+\abs{\eta})^{\alpha}(U^2q+h)\,d\eta dz\\
    \notag&\leq
    2^\alpha\int_{\Real^2}(1+\abs{z})^\alpha\e^{-\abs{z}}(1+\abs{\eta})^{\alpha}(U^2q+h)\,d\eta
    dz\quad \text{(by \eqref{eq:ineqal})}\\
    \notag&\leq
    C\int_{\Real}(1+\abs{z})^\alpha\e^{-\abs{z}}\,dz\int_\Real(1+\abs{\eta})^{\alpha}(U^2+h)\,d\eta\\
    \label{eq:bdQusq}&\leq CK(t).
  \end{align}
  Hence,
  \begin{equation}
    \label{eq:bdalpQ}
    \norm{(1+\abs{\xi})^{\alpha}Q}_{L^1\cap L^{\infty}}\leq CK(t)
  \end{equation}
  and the same bound holds for $P$. From the
  governing equations, we obtain
    \begin{align*}
      \notag \norm{(1+\abs{\xi})^\alpha
        U^2(t,\xi)}_{L^1}&\leq
      \norm{(1+\abs{\xi})^\alpha U_0^2}_{L^1}+
      \int_0^t\norm{2(1+\abs{\xi})^\alpha
        QU(\tau,\cdot)}_{L^1}
      \,d\tau\\
      \notag &\leq \norm{(1+\abs{\xi})^\alpha
        U_0^2}_{L^1}+2\int_0^t
      \norm{(1+\abs{\xi})^\alpha Q^2(\tau,\cdot)}_{L^1}\,d\tau\\
      &\quad+2\int_0^t\norm{(1+\abs{\xi})^\alpha U^2(\tau,\cdot)}_{L^1}\,d\tau\\
    \label{eq:estidecu2}
    &\leq
    \norm{(1+\abs{\xi})^\alpha U_0^2}_{L^1}+C\int_0^t
    K(\tau)\,d\tau,
  \end{align*}
  by \eqref{eq:bdalpQ}, as
  $\norm{Q}_{L^\infty}\leq C$, see
  \eqref{eq:bdderPQ}. In a similar way, one proves
  that
  \begin{equation*}
    \norm{(1+\abs{\xi})^\alpha
      w^2(t,\xi)}_{L^1}\leq\norm{(1+\abs{\xi})^\alpha w_0^2}_{L^1}+C\int_0^t
    K(\tau)\,d\tau
  \end{equation*}
  and 
  \begin{equation*}
    \norm{(1+\abs{\xi})^\alpha
      h(t,\xi)}_{L^1}\leq\norm{(1+\abs{\xi})^\alpha h_0}_{L^1}+C\int_0^t
    K(\tau)\,d\tau
  \end{equation*}
  so that
  \begin{equation*}
    K(t)\leq K(0)+C\int_0^tK(\tau)\,d\tau 
  \end{equation*}
  and the result follows from Gronwall's Lemma.
\end{proof}
For later use, we note that, in this proof, we
have established that
\begin{equation}
  \label{eq:estdecPQ}
  \norm{\e^{\abs{\xi}}Q}_{L^\infty}+
  \norm{\e^{\abs{\xi}}P}_{L^\infty}\leq C(\norm{Y}_{F^e})
\end{equation}
and 
\begin{equation}
  \label{eq:estdecPQ2}
  \norm{(1+\abs{\xi})^\alpha Q}_{L^\infty\cap L^1}+
  \norm{(1+\abs{\xi})^\alpha P}_{L^\infty\cap L^1}\leq C(\norm{Y}_{F^\alpha})
\end{equation}
for some given increasing function $C$, see
\eqref{eq:decexpQ}, \eqref{eq:decexpP} and
\eqref{eq:bdalpQ}.

\section{Semi-Discretisation in space}\label{sect-semi}

The first step towards a discretisation of
\eqref{eq:sysban} is to consider
step-functions. We consider an equally-spaced grid
on the real line defined by the points
$$
\xi_i=i\Delta\xi,
$$
where $\Delta\xi$ is the grid step and
$i=0,\pm1,\pm2,\ldots$. We introduce the space
\begin{multline*}
  F_{\Delta\xi}=\{ Y\in F\,:\:\text{each component of }
  Y\text{ consists of }\\\text{piecewise constant
    functions in each intervals
    $[\xi_i,\xi_{i+1})$}\}.
\end{multline*}
The system \eqref{eq:sysban} does not preserve the
set $F_{\Delta\xi}$ of piecewise constant
function. Thus, we define
\begin{equation}\label{eq:Pbar}
P_{\Delta\xi}( Y)(\xi)=\sum_{i=-\infty}^\infty 
P( Y)(\xi_i)
\chi_{[\xi_i,\xi_{i+1})}(\xi),
\end{equation}
\begin{equation}\label{eq:Qbar}
  Q_{\Delta\xi}( Y)(\xi)=\sum_{i=-\infty}^\infty 
  Q( Y)(\xi_i)
  \chi_{[\xi_i,\xi_{i+1})}(\xi)
\end{equation}
and consider a second system of differential
equations
\begin{equation}
\label{eq:sysbar}
\begin{aligned}
  \zeta_t&=\gamma U\\
  U_t&=-Q_{\Delta\xi}\\
  H_t&= U^3-2P_{\Delta\xi} U\\
  q_t&=\gamma  w\\
  w_t&=\sfrac{\gamma}{2} h+
  \bigl(\sfrac{3-2\gamma}{2} U^2-P_{\Delta\xi}\bigr) q\\
  h_t&=-2Q_{\Delta\xi} U q+ \bigl(3
  U^2-2P_{\Delta\xi}\bigr) w,
\end{aligned}
\end{equation}
or, shortly,
$$
 Y_t(t)=G_{\Delta\xi}( Y(t)).
$$
Like in the preceding section, we show that this
system of differential equations possesses a
short-time solution, an invariant and that it
solution converges to the solution of
\eqref{eq:sysban} as $\Delta\xi\to0$. In the
next theorem we prove, by a contraction argument,
the short-time existence of solutions to
\eqref{eq:sysbar}.
\begin{theorem}\label{th:existbar}
  For any initial value $ Y_0=( y_0, U_0, H_0,
  q_0, w_0, h_0)\in F$, there exists a time $T$,
  only depending on $\norm{Y_0}_F$, such that the
  system of differential equations
  \eqref{eq:sysbar} admits a unique solution in
  $C^1([0,T], F)$.
\end{theorem}
This theorem is a consequence of point (i) in the
following lemma.
\begin{lemma}
  \label{lem:LipGbar}
  The following statements hold
  \begin{enumerate}
  \item[(i)]The mapping $G_{\Delta\xi}:F\to F$
    belongs to $C^1(F,F)$ and
    \begin{equation}
      \label{eq:bdderGxi}
      \norm{G_{\Delta\xi}(Y)}_{F}+\norm{\fracpar{G_{\Delta\xi}}{Y}(Y)}_{L(F,F)}\leq C(M),
    \end{equation}
    for any $Y\in B_M$.
  \item[(ii)] For any $Y\in F$, we have
    \begin{equation}
      \label{eq:estdiffGbarG}
      \norm{G(Y)-G_{\Delta\xi}(Y)}_{F}\leq C\sqrt{\Delta\xi}
    \end{equation}
    for some constant $C$ which only depends on
    $\norm{Y}_F$.
  \end{enumerate}
\end{lemma}
\begin{proof} For any function $f\in H^1(\R)$, let
  $\lP(f)$ be the function defined as
  $\displaystyle\lP(f)(\xi)=\sum_{i=-\infty}^\infty
  f(\xi_i)\chi_{[\xi_i,\xi_{i+1})}(\xi)$.  Thus,
  we can rewrite $Q_{\Delta\xi}(Y)$ and $P_{\Delta\xi} (Y)$ as
  \begin{equation*}
    Q_{\Delta\xi}(Y)=\lP[Q(Y)]\text{ and }P_{\Delta\xi}(Y)=\lP[P(Y)].
  \end{equation*}
  Let us prove that $\lP$ is a continuous mapping
  from $H^1(\R)$ to $L^\infty(\R)\cap L^2(\R)$. 
  By using the
  Sobolev embedding theorem of $H^1(\R)$ into
  $L^\infty(\R)$, we get
  \begin{equation*}
    \norm{\lP(f)}_{L^\infty}\leq\norm{f}_{L^\infty}\leq C\norm{f}_{H^1}
  \end{equation*}
  for some constant $C$, so that $\lP$ is
  continuous from $H^1(\R)$ into $L^\infty(\R)$. 
  The $L^2$ norm of $\lP(f)$ is given by
  \begin{equation*}
    \norm{\lP(f)}_{L^2}^2=\sum_{i=-\infty}^\infty \Delta\xi f(\xi_i)^2.
  \end{equation*}
  We have, for all $\xi\in[\xi_i,\xi_{i+1})$, that
  \begin{align*}
    f(\xi_i)^2&=f(\xi)^2-2\int_{\xi_i}^{\xi}
    f(\eta)f_\xi(\eta)\,d\eta\\
    &\leq
    f(\xi)^2+\int_{\xi_i}^{\xi_{i+1}}
    f^2(\eta)\,
    d\eta+\int_{\xi_i}^{\xi_{i+1}}
    f_\xi^2(\eta)\,d\eta
  \end{align*}
  which, after integration over
  $[\xi_i,\xi_{i+1})$, yields
  \begin{equation*}
    \Delta \xi f(\xi_i)^2\leq\int_{\xi_i}^{\xi_{i+1}}
    f(\eta)^2\,d\eta+
    \Delta\xi\bigl(\int_{\xi_i}^{\xi_{i+1}}f^2(\eta)\,d\eta+
    \int_{\xi_i}^{\xi_{i+1}}f_\xi^2(\eta)\,d\eta\bigr).
  \end{equation*}
  Hence,
  \begin{equation*}
    \norm{\lP(f)}_{L^2}^2\leq(1+\Delta\xi)
    \norm{f}_{L^2}^2+\Delta\xi\norm{f_\xi}_{L^2}^2
  \end{equation*}
  and the mapping $\lP$ is continuous from
  $H^1(\R)$ to $L^2(\R)$. Since $Q_{\Delta\xi}$
  and $P_{\Delta\xi}$ are compositions of a
  continuous linear map $\lP$ and a $C^1$ map,
  they are also $C^1$ and
  \begin{equation*}
    \fracpar{P_{\Delta\xi}}{Y}(\bar Y)=\lP(\fracpar{P}{Y}(Y)[\bar Y])
  \end{equation*}
  for all $\bar Y\in F$.  The same holds for $Q$
  so that \eqref{eq:bdderGxi} follows from Lemma
  \ref{lem:Blip}. Let us prove point (ii). First
  we note that \eqref{eq:estdiffGbarG} follows
  directly from the definitions of $G$,
  $G_{\Delta\xi}$ and the estimate
  \begin{equation}
    \label{eq:estdiffGbarG01}
    \norm{Q(Y)-Q_{\Delta\xi}(Y)}_{L^2\cap L^\infty}+\norm{P(Y)-P_{\Delta\xi}(Y)}_{L^2\cap L^\infty}\leq C\sqrt{\Delta\xi}.
  \end{equation}
  Let us prove \eqref{eq:estdiffGbarG01}. We
  estimate $\norm{\id-\lP}_{L(H^1,L^\infty\cap
    L^2)}$, where the norm here is the operator
  norm from $H^1(\R)$ to $L^\infty(\R)\cap
  L^2(\R)$.  Let us consider $f\in H^1(\R)$, we
  have
  \begin{equation*}
    \norm{f-\lP(f)}_{L^\infty}\leq
    \sup_{i}
    \norm{f(\xi)-f(\xi_i)}_{L^\infty([\xi_i,\xi_{i+1}])}.
  \end{equation*}
  For any $\xi\in[\xi_i,\xi_{i+1})$, we have
  $\abs{f(\xi)-f(\xi_i)}
  \leq\sqrt{\Delta\xi}\norm{f_\xi}_{L^2}$, by the
  Cauchy--Schwartz inequality. Hence,
  \begin{equation*}
    \norm{f-\lP(f)}_{L^\infty}\leq
    \sqrt{\Delta\xi}\norm{f_\xi}_{L^2}\leq\sqrt{\Delta\xi}\norm{f}_{H^1}.
  \end{equation*}
  We have
  \begin{align*}
    \int_{\xi_i}^{\xi_{i+1}}
    \abs{f(\xi)-\lP(f)(\xi)}^2\,d\xi
    &=\int_{\xi_i}^{\xi_{i+1}}
    \abs{\int_{\xi_i}^{\xi}f_\xi(\eta)
    \,d\eta}^2\,d\xi\\ &\leq\int_{\xi_i}^{\xi_{i+1}}
    ((\xi-\xi_i)\int_{\xi_i}^{\xi}f_\xi^2(\eta)
    \,d\eta)\,d\xi\\
    &\leq\int_{\xi_i}^{\xi_{i+1}}f_\xi^2(\eta)
    \,d\eta\int_{\xi_i}^{\xi_{i+1}}
    (\xi-\xi_{i})\,d\xi\\
    &=\sfrac{(\Delta\xi)^2}{2}
    \int_{\xi_i}^{\xi_{i+1}}f_\xi^2(\eta)\,d\eta.
  \end{align*}
  Hence,
  \begin{equation}
    \label{eq:L2normProj}
    \norm{f-\lP(f)}_{L^2}\leq\sfrac{\Delta\xi}{\sqrt2}
    \norm{f}_{H^1}
  \end{equation}
  and we have proved that $\norm{\id-\lP}_{L^2\cap
    L^\infty}\leq C\sqrt{\Delta\xi}$ for some
  constant $C$. Then, we have
  \begin{align*}
    \norm{Q(Y)-Q_{\Delta\xi}(Y)}_{L^2\cap L^\infty}\leq
    C\sqrt{\Delta\xi}\norm{Q(Y)}_{H^1}\leq
    C'\sqrt{\Delta\xi}
  \end{align*}
  for another constant $C'$ which depends only on
  $\norm{Y}_F$.  One proves in the same way the
  same estimate for $P$ and thus we obtain
  \eqref{eq:estdiffGbarG01}.
\end{proof}

Concerning our new system of equations
\eqref{eq:sysbar}, it is not difficult to show in
the same way as in \eqref{eq:compIpres} that
\begin{equation*}
  I_{\Delta\xi}(Y):= U^2 q^2+ w^2- q h
\end{equation*}
is also a conserved quantity along the exact
solution of our problem. The system
\eqref{eq:sysbar} is introduced because it allows
for a space discretisation of the original system
\eqref{eq:sysban}. Indeed, the set of piecewise
constant functions is preserved:
\begin{lemma}
  The set $F_{\Delta\xi}$ is preserved, that is,
  if $Y_0\in F_{\Delta\xi}$ and $Y(t)$ is the
  solution of \eqref{eq:sysbar} with initial data
  $Y_0$, then $Y(t)\in F_{\Delta\xi}$ for all
  $t\in[0,T]$.
\end{lemma}
The proof of this lemma is straightforward. We can
now compare solutions of \eqref{eq:sysbar} and of the
original system \eqref{eq:sysban}.
\begin{theorem}\label{th:convbar}
  Given $M>0$ and $Y_0,Y_{0,\Delta\xi}\in F$. Let
  $Y(t)$ be the short-time solution of
  \eqref{eq:sysban} with initial data $Y_0$ and
  $Y_{\Delta\xi}(t)$ be the short-time solution of
  \eqref{eq:sysbar} with initial data $Y_{0,\Delta\xi}$
  in the interval $[0,T]$. If we have
  \begin{equation*}
    \norm{Y(t)}_F\leq M\text{ and }\norm{Y_{\Delta\xi}(t)}_F\leq M\quad\text{ for all }t\in[0,T], 
  \end{equation*}
  then we also have
  \begin{equation}
    \label{eq:bdcompsys}
    \norm{Y(t)-Y_{\Delta\xi}(t)}_{F}\leq 
    \bigl(\norm{Y_0-Y_{0,\Delta\xi}}+CT\sqrt{\Delta\xi}\bigr)\e^{CT}
     \quad\text{for all }t\in[0,T]
  \end{equation}
  with some constant $C$ which depends only on $M$.
\end{theorem}
\begin{proof}
  The proof of this theorem is a consequence of Lemma~\ref{lem:LipGbar} 
  and of Gronwall's Lemma. We have
  \begin{align*}
    Y(t)-Y_{\Delta\xi}(t)&=Y_0-Y_{0,\Delta\xi}+
    \int_{0}^t\bigl(G(Y(\tau))-G_{\Delta\xi}(Y_{\Delta\xi}(\tau))\bigr)\,d\tau\\
    &=Y_0-Y_{0,\Delta\xi}+\int_{0}^t\bigl(G(Y(\tau))-G(Y_{\Delta\xi}(\tau))
    +G(Y_{\Delta\xi}(\tau))-G_{\Delta\xi}(Y_{\Delta\xi}(\tau))\bigr)\,
    d\tau
  \end{align*}
  which yields, after using Proposition
  \ref{prop:LipG} and Lemma \ref{lem:LipGbar},
  \begin{equation*}
    \norm{Y(t)-Y_{\Delta\xi}(t)}_F\leq\norm{Y_0-Y_{0,\Delta\xi}}_F+C\int_{0}^t\norm{Y(\tau)-Y_{\Delta\xi}(\tau)}_F\,d\tau+CT\sqrt{\Delta\xi},
  \end{equation*}
  for some constant $C$ which depends only on
  $M$. Then, \eqref{eq:bdcompsys} follows from
  Gronwall's Lemma.
\end{proof}

Lemma \ref{lem:presprop} and Theorem
\ref{th:decay} show that there exist properties of
the initial data that are preserved by the system
\eqref{eq:sysban}. The same results - with the
exception of property (iv) in Lemma
\ref{lem:presprop} - hold for the system
\eqref{eq:sysbar}. This is the content of the
following theorem.

\begin{theorem}
  \label{th:proppresbar}
  We consider an initial data $Y_0\in F$ and the
  corresponding short time solution $Y(t)$ of
  \eqref{eq:sysbar} given by Theorem
  \ref{th:existbar}.
  \begin{enumerate}
  \item[(i)] If $q_0$, $w_0$, $h_0$ belongs to
    $L^\infty(\R)$ then
    \begin{equation*}
      \sup_{t\in[0,T]}(\norm{q(t,\cdot)}_{L^\infty}+\norm{w(t,\cdot)}_{L^\infty}+\norm{h(t,\cdot)}_{L^\infty})\leq C
    \end{equation*}
    for some constant $C$ which depends only on
    $T$ and $\norm{Y_0}_{F^e}$.
  \item[(ii)] If we have $qh=U^2q^2+w^2$ for $t=0$ 
   (or $I(Y_0)=0$) 
    then this holds for all $t\in[0,T]$.
  \item[(iii)] If we have $qh=U^2q^2+w^2$ 
   (or $I(Y)=0$) and
    $q\geq0,\ h\geq0,\ q+h\geq c$ almost
    everywhere for some constant $c>0$, then the
    same relations holds for all $t\in[0,T]$.
  \item[(iv)] If $Y_0\in F^e$, then
    \begin{equation}
      \label{eq:decGbarexp}
      \sup_{t\in[0,T]}\norm{Y(t,\cdot)}_{F^e}\leq C,
    \end{equation}
    if $Y_0\in F^\alpha$, then
    \begin{equation}
      \label{eq:decGbarpol}
      \sup_{t\in[0,T]}\norm{Y(t,\cdot)}_{F^\alpha}\leq C,
    \end{equation}
    where the constant $C$ depends only on $T$ and
    $\norm{Y_0}_{F^e}$, and $T$ and
    $\norm{Y_0}_{F^\alpha}$, respectively.
  \end{enumerate}
\end{theorem}

\begin{proof}
  The system \eqref{eq:sysbar} is obtained from
  \eqref{eq:sysban} by simply replacing $P$ and
  $Q$ by $P_{\Delta\xi}$ and $Q_{\Delta\xi}$ as
  defined in \eqref{eq:Pbar} and
  \eqref{eq:Qbar}. Therefore, the proofs of points
  (i), (ii) and (iii) in Lemma \ref{lem:presprop},
  which do not require any special properties of
  $P$ and $Q$, apply directly to
  \eqref{eq:sysbar}. After introspection of the
  proof of Theorem \ref{th:decay}, we can see that
  in order to prove \eqref{eq:decGbarexp}, we need
  to prove that the estimates \eqref{eq:estL1Q},
  \eqref{eq:estL1P}, \eqref{eq:decexpQ},
  \eqref{eq:decexpP}, which hold for $P$ and $Q$,
  also hold for $P_{\Delta\xi}$ and
  $Q_{\Delta\xi}$, namely,
  \begin{equation}
    \label{eq:L1PQbar}
    \norm{Q_{\Delta\xi}(t,\cdot)}_{L^1}\leq C(\norm{h(t,\cdot)}_{L^1}+1),\quad
    \norm{P_{\Delta\xi}(t,\cdot)}_{L^1}\leq C(\norm{h(t,\cdot)}_{L^1}+1)
  \end{equation}
  and
  \begin{equation}
    \label{eq:decexpQPbar}
    \norm{\e^{\abs{\xi}}Q_{\Delta\xi}(t,\cdot)}_{L^\infty}\leq C L(t),\quad \norm{\e^{\abs{\xi}}
      P_{\Delta\xi}(t,\cdot)}_{L^\infty}\leq C L(t),
  \end{equation}
  where $L(t)$ is defined in \eqref{eq:defI} and
  $C$ is a constant which depends only on $T$ and
  $\norm{Y_0}_{F^e}$. We denote generically by $C$
  such constant. In the same way that we obtained
  \eqref{eq:L2normProj}, we now get that, for any
  $f\in W^{1,1}(\R)$,
  \begin{align*}
    \int_{\xi_i}^{\xi_{i+1}}
    \abs{f(\xi)-\lP(f)(\xi)}\,d\xi
    &=\int_{\xi_i}^{\xi_{i+1}}
    \abs{\int_{\xi_i}^{\xi}f_\xi(\eta)\,d\eta}
    \,d\xi\\
    &\leq\Delta\xi\int_{\xi_i}^{\xi_{i+1}}
    \abs{f_\xi(\eta)}\,d\eta
  \end{align*}
  and therefore
  \begin{equation}
    \label{eq:L1normproj}
    \norm{f-\lP(f)}_{L^1}\leq\Delta\xi\norm{f_\xi}_{L^1}.
  \end{equation}
  We obtain, after using successively
  \eqref{eq:L1normproj}, \eqref{eq:estL1Q},
  \eqref{eq:diffQ} and \eqref{eq:estL1P}, that
  \begin{align*}
    \norm{Q_{\Delta\xi}}_{L^1}&\leq\norm{Q_{\Delta\xi}-Q}_{L^1}+\norm{Q}_{L^1}\\
    &\leq\Delta\xi\norm{Q_\xi}_{L^1}+C(\norm{h}_{L^1}+1)\\
    &=\Delta\xi\norm{\sfrac{\gamma}{2}h+
      \sfrac{3-2\gamma}{2}U^2q-Pq}_{L^1}+C(\norm{h}_{L^1}+1)\\
    &\leq C\norm{P}_{L^1}+C(\norm{h}_{L^1}+1)\\
    &\leq C(\norm{h}_{L^1}+1).
  \end{align*}
  We handle in the same way $\norm{P_{\Delta\xi}}_{L^1}$
  and this concludes the proof of
  \eqref{eq:L1PQbar}. For any $\xi\in\Real$, we
  have $\xi\in[\xi_i,\xi_{i+1})$ for some
  $i$. Then,
  \begin{equation*}
    \e^{\abs{\xi}}Q_{\Delta\xi}(t,\xi)=\e^{\abs{\xi}-\abs{\xi_i}}\e^{\abs{\xi_i}}
    Q(t,\xi_i)\leq \e^{\Delta\xi}\norm{\e^\xi Q(t,\xi)}_{L^\infty}\leq CL(t)
  \end{equation*}
  by \eqref{eq:decexpQ} and, therefore,
  $\norm{\e^{\abs{\xi}}
    Q_{\Delta\xi}(t,\cdot)}_{L^\infty}\leq C
  L(t)$. Similarly, we obtain the corresponding
  result for $P_{\Delta\xi}$ so that
  \eqref{eq:decexpQPbar} is proved. Again, after
  introspection of the proof of Theorem
  \ref{th:decay}, we can check that, in order to
  prove \eqref{eq:decGbarpol}, we need to prove
  that
  \begin{equation}
    \label{eq:needprpol}
    \norm{(1+\abs{\xi})^\alpha Q_{\Delta\xi}(t,\cdot)}_{L^\infty\cap L^1}+
    \norm{(1+\abs{\xi})^\alpha P_{\Delta\xi}(t,\cdot)}_{L^\infty\cap L^1}\leq C K(t).
  \end{equation}
  We have 
  \begin{equation*}
    \norm{(1+\abs{\xi})^\alpha Q_{\Delta\xi}(t,\cdot)}_{L^\infty}\leq\norm{(1+\abs{\xi})^\alpha Q(t,\cdot)}_{L^\infty}\leq CK(t)
  \end{equation*}
  by \eqref{eq:estQKt}. Since
  $\e^{\xi-\eta}\leq\e^{\Delta\xi}\e^{\xi_i-\eta}$
  for any $(\xi,\eta)\in[\xi_i,\xi_{i+1}]^2$, we
  get
  \begin{align*}
    \norm{(1+\abs{\xi})^\alpha Q_{\Delta\xi}(t,\cdot)}_{L^1}&\leq\sum_{i=-\infty}^{\infty}\int_{\xi_i}^{\xi_{i+1}}\int_\Real(1+\abs{\xi})^{\alpha}\e^{-\abs{\xi_i-\eta}}(U^2q+h)\,d\eta d\xi\\
    &\leq \e^{\Delta\xi}\sum_{i=-\infty}^{\infty}\int_{\xi_i}^{\xi_{i+1}}\int_\Real(1+\abs{\xi})^{\alpha}\e^{-\abs{\xi-\eta}}(U^2q+h)\,d\eta d\xi\\
    &=\e^{\Delta\xi}\int_\Real\int_\Real(1+\abs{\xi})^{\alpha}\e^{-\abs{\xi-\eta}}(U^2q+h)\,d\eta
    d\xi \leq CK(t),
  \end{align*}
  by \eqref{eq:bdQusq}. The corresponding
  results for $P$ are established in the same way and
  this concludes the poof of \eqref{eq:needprpol}.
\end{proof}
In order to complete the discretisation in space,
we have to consider a finite subspace of
$F_{\Delta\xi}$. Given any integer $N$, we denote
$R=N\Delta\xi$ and we introduce the subset $F_{R}$
of $F$ defined as
\begin{align*}
  F_R=\{Y&\in  F\ :&\\
  &U(\xi)=q(\xi)=w(\xi)=h(\xi)=0,&\text{ for all
  }\xi\in(-\infty,-R)\cup[R,\infty),\\
  &\zeta(\xi)=\zeta_{\infty},\,
  H(\xi)=H_\infty,&\text{ for all }\xi\in[R,\infty),\\
  &\zeta(\xi)=\zeta_{{-\infty}},\,H(\xi)=0&\text{ for all
  }\xi\in(-\infty,-R),\\
  &\text{where $\zeta_{\pm\infty}$ and $H_\infty$
    are constants} \}.
\end{align*}
The set $F_R$ basically corresponds to functions
with compact support ($U$, $q$, $w$ and $h$ vanish
outside a compact set). We do not require that the
functions $\zeta$ and $H$ have compact support
($\zeta$ and $H$ belongs to $L^\infty$ with no
extra decay condition) but we impose that they are
constant outside the compact interval $[-R,R]$. We
denote $F_{\{\Delta\xi,R\}}=F_R\cap
F_{\Delta\xi}$.  The set $F_{\{\Delta\xi,R\}}$ is
not preserved by the flow of \eqref{eq:sysbar}
because, as mentioned earlier, $P$ and $Q$ do not
preserve compactly supported functions. That is
why we introduce the cut-off versions of $P$ and
$Q$ given by
\begin{equation*}
  P_{\{\Delta\xi,R\}}(Y)(\xi)=\sum_{i=-N}^{N-1}
  P(Y)(\xi_i)
  \chi_{[\xi_i,\xi_{i+1})}(\xi),
\end{equation*}
\begin{equation*}
  Q_{\{\Delta\xi,R\}}(Y)(\xi)=\sum_{i=-N}^{N-1} 
  Q(Y)(\xi_i)
  \chi_{[\xi_i,\xi_{i+1})}(\xi)
\end{equation*}
and define a third system of differential
equations
\begin{equation}
  \label{eq:sysbarr}
  \begin{aligned}
    \zeta_t&=\gamma U,\\
    U_t&=-Q_{\{\Delta\xi,R\}},\\
    H_t&= U^3-2P_{\{\Delta\xi,R\}} U,\\
    q_t&=\gamma  w,\\
    w_t&=\sfrac{\gamma}{2} h+
    \bigl(\sfrac{3-2\gamma}{2} U^2-P_{\{\Delta\xi,R\}}\bigr) q,\\
    h_t&=-2Q_{\{\Delta\xi,R\}} U q+ \bigl(3 U^2-2P_{\{\Delta\xi,R\}}\bigr)
    w,
  \end{aligned}
\end{equation}
or, shortly,
\begin{equation*}
  Y_t=G_{\{\Delta\xi,R\}}(Y).
\end{equation*}
It is clear from the definition that the system
\eqref{eq:sysbarr} preserves $F_{\{\Delta\xi,R\}}$ and
therefore, since $F_{\{\Delta\xi,R\}}$ is of finite
dimension, the system \eqref{eq:sysbarr} is a 
space discretisation of \eqref{eq:sysban} 
which allows for numerical computations. 
To emphasize that
we are now working in finite dimension, we denote
\begin{equation*}
  Y_i(t)=Y_{\{\Delta\xi,R\}}(t,\xi_i),
\end{equation*}
$\zeta_i= \zeta_{\{\Delta\xi,R\}}(t,\xi_i)$,
$U_i=U_{\{\Delta\xi,R\}}(t,\xi_i)$ and so on for
$H_i,q_i,w_i,h_i,P_i$ and $Q_i$
for
$i=\{-N,\ldots,N-1\}$. We have
\begin{equation*}
  Y_{\{\Delta\xi,R\}}(t,\xi)=\sum_{i=-N}^{N-1}
  Y_i(t)\chi_{[\xi_i,\xi_{i+1})}(\xi).
\end{equation*}
Again, we can show that  
\begin{equation}\label{eq:IDR}
I^i_{\{\Delta\xi,R\}}(Y):= U_i^2 q_i^2+ w_i^2- q_i h_i
\end{equation}
are conserved quantities along the exact solution
of problem \eqref{eq:sysbarr}.  Finally, note that
$F_{\{\Delta\xi,R\}}$ is contained in $F^e$ and
$F^\alpha$. Concerning the exact solution of
\eqref{eq:sysbarr}, we have the following theorem.

\begin{theorem}\label{th:existbarr}
  For an initial values $ Y_0=( y_0,
  U_0,  H_0, q_0, w_0, h_0)\in
  F$, there exists a time $T$, only depending on
  the norm of the initial values, such that the
  system of differential equations
  \eqref{eq:sysbarr} admits a unique solution in
  $C^1([0,T], F)$.
\end{theorem}
This theorem is a consequence of point (i) in the
following lemma.
\begin{lemma}
  \label{lem:LipGbarr}
  The following statements holds
  \begin{enumerate}
  \item[(i)]The mapping $G_{\{\Delta\xi,R\}}:F\to F$
    belongs to $C^1(F,F)$ and
    \begin{equation}
      \label{eq:bdderGxiR}
      \norm{G_{\{\Delta\xi,R\}}(Y)}_{F}+\norm{\fracpar{G_{\{\Delta\xi,R\}}}{Y}(Y)}_{L(F,F)}\leq C(M),
    \end{equation}
    for any $Y\in B_M$.
  \item[(ii)] For any $Y\in F^e$, we have
    \begin{equation}
      \label{eq:estdiffGbarGr}
      \norm{G_{\{\Delta\xi,R\}}(Y)-G_{\Delta\xi}(Y)}_{F}\leq C \e^{-R},
    \end{equation}
    for some constant $C$ which only depends on
    $\norm{Y}_{F^e}$.
  \item[(iii)] For any $Y\in F^\alpha$, we have
    \begin{equation}
      \label{eq:estdiffGbarpol}
      \norm{G_{\{\Delta\xi,R\}}(Y)-G_{\Delta\xi}(Y)}_{F}\leq C
      (\sqrt{\Delta\xi}+\frac1{R^{\alpha/2}}),
    \end{equation}
    for some constant $C$ which only depends on
    $\norm{Y}_{F^\alpha}$.
  \end{enumerate}
\end{lemma}
Note that for $Y(t)$ solution of \eqref{eq:sysbarr}, we have
\begin{equation*}
  \sup_{t\in[0,T]}\norm{Y(t,\cdot)}_{F^e}\leq C\text{ and }\sup_{t\in[0,T]}\norm{Y(t,\cdot)}_{F^\alpha}\leq C,
\end{equation*}
where $C$ depends on $\norm{Y_0}_{F^e}$ and
$\norm{Y_0}_{F^\alpha}$, respectively. This
follows from \eqref{eq:decGbarexp},
\eqref{eq:decGbarpol}, \eqref{eq:estdiffGbarGr}
and \eqref{eq:estdiffGbarpol}.
\begin{proof}[Proof of Lemma \ref{lem:LipGbarr}]
  For any function $f\in L^\infty(\R)\cap
  L^2(\R)$, let $\lP_R(f)$ be the function defined
  as $\lP_R(f)(\xi)=f(\xi)\chi_{[-R,R)}$.  Thus,
  we can rewrite $Q_{\{\Delta\xi,R\}}(Y)$ and
  $P_{\{\Delta\xi,R\}}(Y)$ as
  \begin{equation*}
    Q_{\{\Delta\xi,R\}}(Y)=\lP_R[Q_{\Delta\xi}(Y)]\text{ and }P_{\{\Delta\xi,R\}}(Y)=\lP_R[P_{\Delta\xi}(Y)].
  \end{equation*}
  The operator $\lP_R$ is a projection from
  $L^\infty(\R)\cap L^2(\R)$ into itself and
  therefore its norm is smaller than one. Hence,
  \eqref{eq:bdderGxiR} follows from
  \eqref{eq:bdderGxi}. Let us prove (ii).  We
  consider $Y\in F^e$. We have to prove
  \begin{equation}
    \label{eq:estdiffGbarGr01}
    \norm{Q_{\{\Delta\xi,R\}}(Y)-Q_{\Delta\xi}(Y)}_{L^2\cap L^\infty}+\norm{P_{\{\Delta\xi,R\}}(Y)-P_{\Delta\xi}(Y)}_{L^2\cap L^\infty}\leq 
    C \e^{-R}.
  \end{equation}
  By \eqref{eq:estdecPQ}, we have
  $\norm{\e^{\abs{\xi}}Q}_{L^\infty}+
  \norm{\e^{\abs{\xi}}P}_{L^\infty}\leq C$. Hence,
  \begin{equation*}
    \norm{Q_{\{\Delta\xi,R\}}-Q_{\Delta\xi}}_{L^\infty}=\sup_{\abs{\xi_i}\geq R}\abs{Q(\xi_i)}\leq C\sup_{\abs{\xi_i}\geq R}\e^{-\abs{\xi_i}}=C\e^{-R}.
  \end{equation*}
  We have
  \begin{equation*}
    \norm{Q_{\{\Delta\xi,R\}}-Q_{\Delta\xi}}_{L^2}^2=\Delta\xi \sum_{\abs{\xi_i}\geq R}Q(\xi_i)^2\leq C\Delta\xi \sum_{\abs{i\Delta\xi}\geq R} \e^{-2\abs{i\Delta\xi}}\leq C\sfrac{2\Delta\xi}{1-\e^{-2\Delta\xi}}\e^{-2R}
  \end{equation*}
  and therefore
  $\norm{Q_{\{\Delta\xi,R\}}-Q_{\Delta\xi}}_{L^2}\leq
  C \e^{-R}$. We prove in the same way the
  corresponding result for $P$ and it concludes
  the proof of \eqref{eq:estdiffGbarGr01}. The
  estimate \eqref{eq:estdiffGbarGr} follows from
  \eqref{eq:estdiffGbarGr01}. Let us prove
  (iii). We consider $Y\in F^\alpha$. We have to
  prove that 
  \begin{multline}
    \label{eq:estdiffPQpol}
    \norm{Q_{\{\Delta\xi,R\}}(Y)-Q_{\Delta\xi}(Y)}_{L^2\cap
      L^\infty}\\+\norm{P_{\{\Delta\xi,R\}}(Y)-P_{\Delta\xi}(Y)}_{L^2\cap
      L^\infty}\leq C
    (\sqrt{\Delta\xi}+\frac1{R^{\alpha/2}}).
  \end{multline}
  By \eqref{eq:estdecPQ2}, we have
  $\norm{(1+\abs{\xi})^\alpha Q}_{L^\infty}+
  \norm{(1+\abs{\xi})^\alpha P}_{L^\infty}\leq
  C$. Hence,
  \begin{equation}
    \label{eq:Qdifpol1}
    \norm{Q_{\{\Delta\xi,R\}}-Q_{\Delta\xi}}_{L^\infty}=\sup_{\abs{\xi_i}\geq R}\abs{Q(\xi_i)}\leq C\sup_{\abs{\xi_i}\geq R}(1+\abs{\xi_i})^{-\alpha}=C(1+R)^{-\alpha}.
  \end{equation}
  We have
  \begin{align}
    \notag
    \norm{Q_{\{\Delta\xi,R\}}-Q_{\Delta\xi}}_{L^2(\Real)}&=\norm{Q_{\Delta\xi}}_{L^2(\Real\setminus[-R,R])}\\
    \notag
    &\leq\norm{Q_{\Delta\xi}-Q}_{L^2(\Real\setminus[-R,R])}+\norm{Q}_{L^2(\Real\setminus[-R,R])}\\
    \label{eq:Qdifpol2}
    &\leq
    C(\sqrt{\Delta\xi}+\norm{Q}_{L^2(\Real\setminus[-R,R])}),
  \end{align}
  from \eqref{eq:estdiffGbarG01}. Since
  \begin{align*}
    \norm{Q}_{L^2(\Real\setminus[-R,R])}^2&\leq
    (1+\abs{R})^{-\alpha}\int_{\Real\setminus[-R,R]}(1+\abs{\xi})^{\alpha}Q^2\,d\xi\\
    &\leq C(1+R)^{-\alpha},\quad\text{ by
      \eqref{eq:estdecPQ2}},
  \end{align*}
  the estimate \eqref{eq:estdiffPQpol} follows
  from \eqref{eq:Qdifpol1} and
  \eqref{eq:Qdifpol2}.
\end{proof}
Again, the system \eqref{eq:sysbarr} preserves
properties of the initial data:
\begin{theorem}
  We consider an initial data $Y_0\in F$ and the
  corresponding short time solution $Y(t)$ of
  \eqref{eq:sysbarr} given by Theorem
  \ref{th:existbarr}. Then, $Y(t)$ satisfy points
  (i)-(iv) as given in Theorem
  \ref{th:proppresbar}.
\end{theorem}
Finally, for any initial data in $Y_0\in F^e$, resp. 
$Y_0\in F^\alpha$, we
obtain the following error estimate for bounded
solutions.
\begin{theorem}
  Given $Y_0$ and $Y_{0,\Delta\xi,R}$ in $F^e$, let $Y(t)$
  and $Y_{\{\Delta\xi,R\}}(t)$ be the short-time solutions of
  \eqref{eq:sysban} and \eqref{eq:sysbarr},
  respectively, with initial data $Y_0$ and $Y_{0,\Delta\xi,R}$, respectively. If we have
  \begin{equation*}
    \norm{Y(t)}_{F^e}\leq M\text{ and }\norm{Y_{\{\Delta\xi,R\}}(t)}_F\leq M\quad\text{ for all }t\in[0,T], 
  \end{equation*}
  then we have
  \begin{equation}
    \label{eq:compares}
    \sup_{t\in[0,T]}\norm{Y(t,\cdot)-Y_{\{\Delta\xi,R\}}(t,\cdot)}_{F}\leq 
    C\big(\norm{Y_0-Y_{0,\Delta\xi,R}}_{F}+\sqrt{\Delta\xi}+\e^{-R}\big),
  \end{equation}
  where the constant $C$ depends only on $M$. For
  $Y_0$ and $Y_{0,\Delta\xi,R}$ in $F^\alpha$, we
  have that if
  \begin{equation*}
    \norm{Y(t)}_{F^\alpha}\leq M\text{ and }\norm{Y_{\{\Delta\xi,R\}}(t)}_F\leq M\quad\text{ for all }t\in[0,T], 
  \end{equation*}
  then
  \begin{equation}
    \label{eq:compares2}
    \sup_{t\in[0,T]}\norm{Y(t,\cdot)-Y_{\{\Delta\xi,R\}}(t,\cdot)}_{F}\leq 
    C\big(\norm{Y_0-Y_{0,\Delta\xi,R}}_{F}+\sqrt{\Delta\xi}+\frac1{R^{\alpha/2}}\big).
  \end{equation}
\end{theorem}

\begin{proof}
  We have 
  \begin{multline}
    \label{eq:diffbd}
    \norm{Y(t,\cdot)-Y_{\{\Delta\xi,R\}}(t,\cdot)}_{F}\leq
    \norm{Y_0-Y_{0,\Delta\xi,R}}_{F}+\\\int_{0}^t\norm{G(Y(\tau,\cdot))-
      G_{\{\Delta\xi,R\}}(Y_{\{\Delta\xi,R\}}(\tau,\cdot))}_{F}\,
    d\tau.
  \end{multline}
  By Proposition~\ref{prop:LipG} and Lemmas~
  \ref{lem:LipGbar} and~\ref{lem:LipGbarr}, we get
  \begin{align*}
    \|G(Y(\tau,\cdot))-G_{\{\Delta\xi,R\}}(Y_{\{\Delta\xi,R\}}&(\tau,\cdot))\|_{F}\\
    &\leq
    \norm{G(Y(\tau,\cdot))-G_{\Delta\xi}(Y(\tau,\cdot))}_{F}\\
    &\quad+\norm{G_{\Delta\xi}(Y(\tau,\cdot))-G_{\{\Delta\xi,R\}}(Y(\tau,\cdot))}_{F}\\
    &\quad+\norm{G_{\{\Delta\xi,R\}}(Y(\tau,\cdot))-G_{\{\Delta\xi,R\}}(Y_{\{\Delta\xi,R\}}(\tau,\cdot))}_{F}\\
    &\leq C\big((\Delta\xi)^{\sfrac{1}{2}}+
    \e^{-R}+\norm{Y(\tau,\cdot)-Y_{\{\Delta\xi,R\}}(\tau,\cdot)}_{F}\big)
  \end{align*}
  for a constant $C$ which depends only on
  $M$. Hence, \eqref{eq:compares} follows from
  \eqref{eq:diffbd} after applying Gronwall's
  Lemma. The proof of \eqref{eq:compares2} is
  similar.
\end{proof}


\section{Approximation of the initial data and
  Convergence of the Semi-Discrete solutions }
\label{sect-appdata}

\subsection{Approximation of the initial data}
\label{subsec:appinitdat}

The construction of the initial data
$Y_{0,\Delta_\xi,R}$ is done in two steps. First,
we change variable from Eulerian to Lagrangian,
that is, we compute $Y_0\in\G$ such that
$X=(y_0,U_0,H_0)\in\F$ satisfies
\begin{equation}
  \label{eq:U0u0}
  U_0=u_0\circ y_0.
\end{equation}
In the new set of variables, we can solve
\eqref{eq:sysban} or, rather, its discretisation
\eqref{eq:sysbarr}. Note that, given $u_0\in H^1(\R)$,
there exists several $Y_0\in\G$ such that
\eqref{eq:U0u0} holds. This is a consequence of
relabeling invariance and this fact will be used
in the numerical examples of Section
\ref{sect-numexp}. Here, we present a framework
valid for general initial data in $H^1(\R)$. In
Section \ref{sect-semigroup}, we define the
mapping $L$ from $\D$ to $\F$.  For $u_0\in H^1(\R)$
and $\mu_0$ absolutely continuous, it simplifies
and reads
\begin{subequations}
  \label{eq:initdatex}
  \begin{equation}
    y_0(\xi)+\int_{-\infty}^{y_0(\xi)}(u_0^2+u_{0,x}^2)\,dx=\xi,
  \end{equation}
  \begin{equation}
    U_0=u_0\circ y\quad\text{ and }\quad H_0=\id{}-y_0.
  \end{equation}
  Then, we set
  \begin{equation}
    q_0=y_{0,\xi},\quad w=U_{0,\xi},\quad h=H_{0,\xi}.
  \end{equation}
\end{subequations}
As earlier, we denote $v_0=1-q_0$ and
$\zeta_0=\id-y_0$. We have
\begin{equation}
  \label{eq:stposq}
  h_0q_0=q_0^2U_0^2+w_0^2,\quad q_0+h_0=1,\quad q_0>0,\quad  h_0\geq0 
  \quad \text{ for almost every $\xi\in\Real$.}
\end{equation}
The element $Y_0=(y_0,U_0,H_0,q_0,w_0,h_0)$
belongs to $\G$. The second step consists of
computing an approximation of $Y_0$ in
$F_{\{\Delta\xi,R\}}$. In the following theorem,
we show how the change of variable given by
\eqref{eq:initdatex} deal with the decay
conditions. For simplicity, we drop the subscript
zero in the notation. Let us introduce the Banach
spaces $H^{1,e}$ and $H^{1,\alpha}$ as the
subspaces of $H^1$ with respective norms
\begin{equation*}
  \norm{u}_{H^{1,e}}^2=\norm{\e^{\abs{\frac{\xi}2}}u}_{L^2}^2+\norm{\e^{\abs{\frac{\xi}2}}u_x}_{L^2}^2
\end{equation*}
and
\begin{equation*}
  \norm{u}_{H^{1,\alpha}}^2=\norm{(1+\abs{\xi})^{\frac{\alpha}{2}}u}_{L^2}^2+\norm{(1+\abs{\xi})^{\frac{\alpha}{2}}u_x}_{L^2}^2.
\end{equation*}
\begin{theorem}
  Given $u$ and $Y$ as given by
  \eqref{eq:initdatex}, we have
  \begin{enumerate}
  \item[(i)] $u\in H^{1,e}$ if and only
    if $Y\in F^e$,
  \item[(ii)] $u\in H^{1,\alpha}$ if and only if
    $Y\in F^\alpha$.
  \end{enumerate}
\end{theorem}
\begin{proof}
  Let us assume that $u\in H^{1,e}$. By
  definition, we have $h=(u^2+u_x^2)\circ
  yy_\xi$. Hence,
  \begin{align*}
    \int_{\Real}\e^{\abs{\xi}}h(\xi)\,d\xi&=\int_{\Real}\e^{\abs{\xi}}(u^2+u_x^2)\circ
    y(\xi)y_\xi(\xi)\,d\xi\\
    &=\int_{\Real}\e^{\abs{y^{-1}(x)}}(u^2+u_x^2)(x)\,dx\\
    &=\int_{\Real}\e^{\abs{y^{-1}(x)-x}}\e^{\abs{x}}(u^2+u_x^2)(x)\,dx\\
    &\leq\e^{\norm{y(\xi)-\xi}_{L^\infty}}\int_{\Real}\e^{\abs{x}}(u^2+u_x^2)(x)\,dx<\infty.
  \end{align*}
  Using \eqref{eq:stposq}, we get
  \begin{equation*}
    \int_{\Real}\e^{\abs{\xi}}w^2(\xi)\,d\xi\leq\norm{q}_{L^\infty}\int_{\Real}\e^{\abs{\xi}}h(\xi)\,d\xi<\infty.
  \end{equation*}
  In order to prove that
  $\int_{\Real}\e^{\abs{\xi}}U^2\,d\xi$ is finite,
  we decompose the integral as follows:
  \begin{equation*}
    \int_{\Real}\e^{\abs{\xi}}U^2\,d\xi=\int_{\{\xi\in\Real|q<\frac12\}}\e^{\abs{\xi}}U^2\,d\xi+\int_{\{\xi\in\Real|q>\frac12\}}\e^{\abs{\xi}}U^2\,d\xi.
  \end{equation*}
  We have
  \begin{align*}
    \int_{\{\xi\in\Real|q<\frac12\}}\e^{\abs{\xi}}U^2\,d\xi&\leq
    \norm{U}_{L^\infty}^2\int_{\{\xi\in\Real|q<\frac12\}}\e^{\abs{\xi}}\,d\xi\\
    &\leq \norm{U}_{L^\infty}^2\int_{\{\xi\in\Real|h>\frac12\}}\e^{\abs{\xi}}\,d\xi,\quad\text{ as $q+h=1$},\\
    &\leq2
    \norm{U}_{L^\infty}^2\int_{\{\xi\in\Real|h>\frac12\}}h\e^{\abs{\xi}}\,d\xi\leq
    C\int_{\Real}\e^{\abs{\xi}}h\,d\xi<\infty
  \end{align*}
  and
  \begin{align*}
    \int_{\{\xi\in\Real|q>\frac12\}}\e^{\abs{\xi}}U^2\,d\xi&\leq 2 \int_{\{\xi\in\Real|q>\frac12\}}\e^{\abs{\xi}}\frac{U^2}{q}\,d\xi\\
    &\leq2
    \int_{\{\xi\in\Real|q>\frac12\}}\e^{\abs{\xi}}qh\,d\xi,\quad\text{
      as $U^2\leq qh$ by \eqref{eq:stposq}},\\
    &<\infty.
  \end{align*}
  Hence,
  $\int_{\Real}\e^{\abs{\xi}}U^2\,d\xi<\infty$. Let
  us now assume that $Y\in F^e$. Then, 
  \begin{align*}
    \int_{\Real}\e^{\abs{x}}(u^2+u_x^2)(x)\,dx&=\int_\Real\e^{\abs{y(\xi)}}(u^2+u_x^2)(y(\xi))y_\xi(\xi)\,d\xi\\
    &=\int_\Real\e^{\abs{y(\xi)}}h(\xi)\,d\xi\\
    &\leq\int_\Real\e^{\abs{y(\xi)-\xi}}\e^{\abs{\xi}}h(\xi)\,d\xi\\
    &\leq
    \e^{(\norm{y(\xi)-\xi}_{L^\infty})}\int_{\Real}\e^{\abs{\xi}}h(\xi)\,d\xi<\infty
  \end{align*}
  and $u_0\in H^{1,e}$. The case (ii) is proved in
  the same way.
\end{proof}
As a consequence of this theorem and Theorem
\ref{th:decay}, we obtain 
\begin{theorem}
  \label{th:predeceul}
  The spaces $H^{1,e}$ and $H^{1,\alpha}$ are
  preserved by the hyperelastic rod equation: If
  $u_0\in H^{1,e}$, then $u(t,\cdot)\in H^{1,e}$
  for all positive time and, similarly, if $u_0\in
  H^{1,\alpha}$, then $u(t,\cdot)\in H^{1,\alpha}$
  for all positive time.
\end{theorem}
To the best of our knowledge, these decay results 
are new, even for the
Camassa-Holm equation (case $\gamma=1$). They have
to be compared with \cite{MR2160386} where it is
established that the only solution which has
compact support for all positive time is the zero
solution, i.e., the compactness of the support
(which is a kind of decay condition) is
\textit{not} preserved by the equation.

Let us now construct the approximating sequence
for the initial data. From \eqref{eq:stposq}, we
get that
\begin{equation*}
  0\leq q\leq 1,\quad 0\leq h\leq 1
\end{equation*}
and
\begin{equation}
  \label{eq:boundUxi}
  U_\xi=w\leq \sqrt{hq}\leq\frac12(h+q)=\frac12.
\end{equation}
Given an integer $n$, we consider $\Delta\xi$ and
$R$ such that
$\frac1n=\frac1R+\Delta\xi=\frac1R+\frac{R}{N}$ so
that $n\to\infty$ if and only if $\Delta\xi\to0$
and $R\to\infty$. We introduce the mapping
$\I_{\Delta\xi}:L^2\to L^2$ which approximates
$L^2$ functions by piecewise constant functions,
that is, given $f\in L^2$, let
\begin{equation*}
  \bar f_i=\frac{1}{\Delta\xi}\int_{\xi_i}^{\xi_{i+1}}f(\xi)\,d\xi
\end{equation*}
and set
\begin{equation*}
  \I_{\Delta\xi}(f)(\xi)=\sum_{i=-\infty}^{\infty} 
  \bar f_i\cdot
  \chi_{[\xi_i,\xi_{i+1})}(\xi).
\end{equation*}
We define $Y_n=(y_n,U_n,H_n,q_n,w_n,h_n)$ as
follows. Let
\begin{equation*}
  v_n(\xi)=\lP_R\I_{\Delta\xi}(v),\quad
  w_n(\xi)=\lP_R\I_{\Delta\xi}(w),\quad
  h_n(\xi)=\lP_R\I_{\Delta\xi}(h).
\end{equation*}
As usual, we denote $q=1+v$ and $q_n=1+v_n$. Let
us define the weighted integrals
\begin{equation*}
  U_{i,n}=\frac{\int_{\xi_i}^{\xi_{i+1}}q^2U\,d\xi}{\int_{\xi_i}^{\xi_{i+1}}q^2\,d\xi}.
\end{equation*}
We set
\begin{equation*}
  U_n(\xi)=\sum_{i=-N}^{N-1}  U_{i,n}\cdot
  \chi_{[\xi_i,\xi_{i+1})}(\xi),\quad\text{ for }i=-N,\ldots,N-1.
\end{equation*}
We define
\begin{equation*}
  H_n(\xi)=\lP\left(\int_{-\infty}^{\xi}h_n\,d\eta\right)\quad\text{ if }\quad\xi\in[-R,R]
\end{equation*}
and $H_n(\xi)=\int_{-\infty}^{-R}h_n\,d\eta$ if
$\xi\in(-\infty,-R)$,
$H_n(\xi)=\int_{R}^{\infty}h_n\,d\eta$ if
$\xi\in(R,\infty)$. For $y_n$, we set
\begin{equation*}
  y_n(\xi)=\xi-H_n(\xi)\quad\text{ if }\quad\xi\in[-R,R]
\end{equation*}
and $y_n(\xi)=\xi-H_n(-R)$ if
$\xi\in(-\infty,-R)$, $y_n(\xi)=\xi-H_n(R)$ if
$\xi\in(R,\infty)$. The definition of $\lP$ is
given in the proof of Lemma~\ref{lem:LipGbar}.
The following theorem states that $Y_n$
approximates $Y$ in $F_{\{\Delta\xi,R\}}$ and
satisfies additional properties which will be
useful in Theorem \ref{th:fullinv}, where we prove
that the positivity of the energy is preserved by the 
numerical scheme.
\begin{theorem}\label{th:init}
  Given $Y\in\G$, there exist a sequence $Y_n\in F_{\{\Delta\xi,R\}}$ 
  such that 
  \begin{subequations}
    \label{eq:condappseq}
    \begin{equation}
      \label{eq:convYn}
      \lim_{n\to\infty}\norm{Y_n-Y}_{F}=0,
    \end{equation}
    and
    \begin{equation}
      \label{eq:idynhn}
      q_nh_n\geq U_n^2q_n^2+w_n^2,\quad q_n+h_n=1, \text{  for all $n\geq 0$ 
        and for all $\xi$}.
    \end{equation}
    Moreover, we have
    \begin{equation}
      \label{eq:bdappsedec}
      \norm{Y_n}_{F^e}\leq
      C\norm{Y}_{F^e}\text{ and }\norm{Y_n}_{F^\alpha}\leq
      C\norm{Y}_{F^\alpha}
    \end{equation}
  \end{subequations}
  for $Y\in F^e$, resp. $Y\in F^\alpha$, and where 
  the constant $C$ which does not depend on $Y$
  and $n$.
\end{theorem}
\begin{proof}
  Let us first prove \eqref{eq:idynhn}. Since
  $q+h=1$ (see \eqref{eq:stposq}), we obtain
  $q_n+h_n=1$ from the definitions of $v_n$
  (recall that $q_n=1+v_n$) and $h_n$. We consider
  a fix given interval $I=[\xi_i,\xi_{i+1}]$ and,
  for convenience, denote by an integral without
  boundary the weighted integral $\int
  f(\xi)\,d\xi=\frac1{\Delta\xi}\int_{\xi_i}^{\xi_{i+1}}f(\xi)\,d\xi$
  so that, for $\xi\in I$, $q_n=\int q\,d\xi$,
  $w_n=\int w\,d\xi$ and $h_n=\int h\,d\xi$. Using
  Jensen's inequality, we get that
  \begin{align}
    \notag
    q_n^2+U_n^2q_n^2+w_n^2&=
    \bigl(\int q\,d\xi\bigr)^2+U_n^2\bigl(\int q\,d\xi\bigr)^2
     +\bigl(\int w\,d\xi\bigr)^2\\
    \notag
    &\leq\int q^2\,d\xi+U_n^2\int q^2\,d\xi+\int
    w^2\,d\xi\\
    \notag
    &=\int q^2\,d\xi+U_n^2\int q^2\,d\xi+\int
    \bigl(q(1-q)-q^2U^2\bigr)\,d\xi\\
    \label{eq:ineqqnsq}
    &= q_n+U_n^2\int q^2\,d\xi
    -\int(q^2U^2)\,d\xi.
  \end{align}
  Using the Cauchy-Schwarz inequality and the definition 
of $U_n$, we obtain
  \begin{equation*}
    U_n^2\int q^2\,d\xi=\frac{(\int q^2U)^2\,d\xi}{\int
      q^2\,d\xi}\leq\frac{\int
      q^2\,d\xi\int q^2U^2\,d\xi}{\int
      q^2\,d\xi}=\int q^2U^2\,d\xi.
  \end{equation*}
  Hence, \eqref{eq:ineqqnsq} yields
  \begin{equation*}
    q_n^2+U_n^2q_n^2+w_n^2\leq q_n
  \end{equation*}
  which, as $q_n+h_n=1$, is equivalent to
  $q_nh_n\geq U_n^2q_n^2+w_n^2$. Let us now prove
  \eqref{eq:convYn}. A direct computation shows
  that
  \begin{equation}
    \label{eq:unifbf0}
    \norm{\lP_R\I_{\Delta\xi}(f)}_{L^2}\leq\norm{f}_{L^2},
  \end{equation}
  for any $f\in L^2(\R)$ and any $n$. Since
  $\lim_{n\to\infty}\norm{\lP_R\I_{\Delta\xi}(f)-f}_{L^2}=0$
  for any smooth function $f$ with compact
  support, we obtain, by density and
  \eqref{eq:unifbf0}, that the same result holds
  for any $f\in L^2(\R)$. Hence,
  \begin{equation*}
    \lim_{n\to\infty}\norm{q_n-q}_{L^2}=0,\quad\lim_{n\to\infty}\norm{w_n-w}_{L^2}=0\quad\text{ and }\quad\lim_{n\to\infty}\norm{h_n-h}_{L^2}=0.
  \end{equation*}
 On the interval $I=[\xi_i,\xi_{i+1}]$, we have
 \begin{equation*}
   \abs{U_n(\xi)-U(\xi)}=\abs{\frac{\int q^2(\eta)(U(\eta)-U(\xi))\,d\eta}{\int q^2\,d\eta}}\leq\frac{\Delta\xi}2
 \end{equation*}
 as $\abs{U_\xi}\leq \frac12$, see
 \eqref{eq:boundUxi}. Hence,
 $\norm{U_n-U}_{L^\infty(-R,R)}\leq\frac{\Delta\xi}2$ and
 \begin{align}
   \notag
   \norm{U_n-U}_{L^\infty}&\leq
   \norm{U_n-U}_{L^\infty(-R,R)}+\norm{U}_{L^\infty((-\infty,-R)\cup(R,\infty))}\\
   \label{eq:Linfdiff}
   &\leq\frac{\Delta\xi}2+\norm{U}_{L^\infty((-\infty,-R)\cup(R,\infty))}.
 \end{align}
 Since $U\in H^1(\Real)$,
 $\lim_{\xi\to\pm\infty}U_n=0$ and
 \eqref{eq:Linfdiff} yields
 $\lim_{n\to\infty}\norm{U_n-U}_{L^\infty}=0$. We
 have
 \begin{align}
   \notag
   \norm{U_n-\lP_R(U)}_{L^2}^2&=\sum_{i=-N}^{N-1}\int_{\xi_i}^{\xi_{i+1}}\left(\frac{\int_{\xi_i}^{\xi_{i+1}}q^2U\,d\eta}{\int_{\xi_i}^{\xi_{i+1}}q^2\,d\eta}-U(\xi)\right)^2\,d\xi\\
   \label{eq:estsuml2u}
   &\leq\sum_{i=-N}^{N-1}\frac{1}{\int_{\xi_i}^{\xi_{i+1}}q^2\,d\eta}\int_{\xi_i}^{\xi_{i+1}}\int_{\xi_i}^{\xi_{i+1}}q^2(\eta)(U(\xi)-U(\eta))^2\,d\xi
   d\eta,
 \end{align}
 after applying Cauchy-Schwarz. For $\xi,\eta\in
 I$, we have
 \begin{align*}
   \bigl(U(\xi)-U(\eta)\bigr)^2=
   \bigl(\int_{\eta}^{\xi}U_\xi(\bar\eta)\,d\bar\eta\bigr)^2\leq\Delta\xi\int_{\eta}^{\xi}U_\xi(\bar\eta)^2\,d\bar\eta\leq\Delta\xi\int_{\xi_i}^{\xi_{i+1}}U_\xi^2\,d\bar\eta.
 \end{align*}
 Hence, \eqref{eq:estsuml2u} yields
 \begin{equation*}
   \norm{U_n-\lP_R(U)}_{L^2}^2\leq\sum_{i=-N}^{N-1}\frac{(\Delta\xi)^2}{\int_{\xi_i}^{\xi_{i+1}}q^2\,d\eta}\int_{\xi_i}^{\xi_{i+1}}q^2\,d\eta\int_{\xi_i}^{\xi_{i+1}}U_\xi^2\,d\bar\eta\leq(\Delta\xi)^2\norm{U_\xi}_{L^2}^2.
 \end{equation*}
 It follows that
 \begin{align*}
   \norm{U_n-U}_{L^2}&\leq\norm{U_n-\lP_R(U)}_{L^2}+\norm{U-\lP_R(U)}_{L^2}\\
   &\leq
   \Delta\xi\norm{U_\xi}_{L^2}+\norm{U}_{L^2((-\infty,-R)\cup(R,\infty))}
 \end{align*}
 and therefore
 $\lim_{n\to\infty}\norm{U_n-U}_{L^2}=0$. The
 function $h$ belongs to $L^1(\R)$ because
 $h=h^2+U^2q^2+w^2$, by \eqref{eq:stposq}. A
 direct computation shows that
 \begin{equation}
   \label{eq:unifbf02}
   \norm{\lP_R\I_{\Delta\xi}(f)}_{L^1}\leq\norm{f}_{L^1},
 \end{equation}
 for any $f\in L^1(\R)$ and any $n$. Since
 $\lim_{n\to\infty}\norm{\lP_R\I_{\Delta\xi}(f)-f}_{L^1}=0$
 for any smooth function $f$ with compact support,
 we obtain, by density and \eqref{eq:unifbf02},
 that the same result holds for any $f\in
 L^1(\R)$. Hence,
 $\lim_{n\to\infty}\norm{h_n-h}_{L^1}=0$ and
 therefore
 \begin{equation*}
   \lim_{n\to\infty}\norm{H_n-H}_{L^\infty}=0.
 \end{equation*}
 Since $y_n=\xi-H_n$ and $y=\xi-H$, we get also
 that
 $\lim_{n\to\infty}\norm{y_n-y}_{L^\infty}=0$.
 Let us look at the bounds on the decay of $Y$. We
 assume $Y\in F^{e}$. We have
 \begin{align*}
   \int_\Real
   \e^{\abs{\xi}}\abs{h_n}\,d\xi&=\frac{1}{\Delta\xi}\sum_{i=-N}^{N+1}\int_{\xi_i}^{\xi_{i+1}}\int_{\xi_i}^{\xi_{i+1}}\e^{\abs{\xi}}\abs{h(\eta)}\,d\eta
   d\xi\\
   &=\frac{1}{\Delta\xi}\sum_{i=-N}^{N+1}\int_{\xi_i}^{\xi_{i+1}}\int_{\xi_i}^{\xi_{i+1}}\e^{\abs{\xi}}\e^{-\abs{\eta}}\e^{\abs{\eta}}\abs{h(\eta)}\,d\eta
   d\xi\\
   &\leq\frac{1}{\Delta\xi}\sum_{i=-N}^{N+1}\int_{\xi_i}^{\xi_{i+1}}\int_{\xi_i}^{\xi_{i+1}}\e^{\abs{\xi-\eta}}\e^{\abs{\eta}}\abs{h(\eta)}\,d\eta
   d\xi\\
   &\leq
   \e^{\Delta\xi}\sum_{i=-N}^{N+1}\int_{\xi_i}^{\xi_{i+1}}\e^{\abs{\eta}}\abs{h(\eta)}\,d\eta\leq
   3\norm{Y}_{F^e},
 \end{align*}
 after assuming, without loss of generality, that
 $\Delta\xi\leq1$.  Similarly one proves that
 $\int_\Real\e^{\abs{\xi}}w^2\,d\xi\leq
 C\norm{Y}_{F^e}$. It remains to estimate
 $\int_\Real U_n^2\e^{\abs{\xi}}\,d\xi$. For any
 $\eta,\xi\in[\xi_i,\xi_{i+1}]$, we have
 \begin{align*}
   U^2(\eta)&=U^2(\xi)+2\int_{\xi}^{\eta}UU_\xi(\bar\xi)\,d\bar\xi\\
   &\leq
   U^2(\xi)+\int_{\xi_i}^{\xi_{i+1}}(U^2+(U_\xi)^2)(\bar\xi)\,d\bar\xi=U^2(\xi)+\int_{\xi_i}^{\xi_{i+1}}(U^2+w^2)(\bar\xi)\,d\bar\xi.
 \end{align*}
 Hence,
 \begin{align*}
   U_{i,n}^2&=\left(\frac{\int_{\xi_i}^{\xi_{i+1}}q^2(\eta)U(\eta)\,d\eta}{\int_{\xi_i}^{\xi_{i+1}}q^2(\eta)\,d\eta}\right)^2\\
   &\leq\frac{\int_{\xi_i}^{\xi_{i+1}}q^2(\eta)U^2(\eta)\,d\eta}{\int_{\xi_i}^{\xi_{i+1}}q^2(\eta)\,d\eta}\quad\text{ (by Cauchy-Schwarz)}\\
   &\leq
   U^2(\xi)+\int_{\xi_i}^{\xi_{i+1}}(U^2+w^2)(\bar\xi)\,d\bar\xi
 \end{align*}
 for any $\xi\in[\xi_i,\xi_{i+1}]$. Then,
 \begin{align*}
   \int_{\Real}\e^{\abs{\xi}}U_n^2\,d\xi&=\sum_{i=-N}^{N-1}\int_{\xi_i}^{\xi_{i+1}}\e^{\abs{\xi}}U_{i,n}^2\,d\xi\\
   &\leq
   \sum_{i=-N}^{N-1}\int_{\xi_i}^{\xi_{i+1}}\big(\e^{\abs{\xi}}\big(U^2(\xi)+\int_{\xi_i}^{\xi_{i+1}}(U^2+w^2)(\bar\xi)\,d\bar\xi\big)\big)\,d\xi\\
   &\leq\int_\Real\e^{\abs{\xi}}U^2(\xi)\,d\xi+\sum_{i=-N}^{N-1}\int_{\xi_i}^{\xi_{i+1}}\int_{\xi_i}^{\xi_{i+1}}(U^2+w^2)(\bar\xi)\e^{\abs{\xi}}\,d\xi
   d\bar\xi\\
   &\leq
   \norm{Y}_{F^e}+\sum_{i=-N}^{N-1}\e^{\Delta\xi}\int_{\xi_i}^{\xi_{i+1}}\int_{\xi_i}^{\xi_{i+1}}(U^2+w^2)(\bar\xi)\e^{\abs{\bar
       \xi}}\,d\xi d\bar\xi\\
   &\leq(1+2\Delta\xi\e^{\Delta\xi})\norm{Y}_{F^e}\leq
   7\norm{Y}_{F^e}.
 \end{align*}
 Thus we have proved that $\norm{Y_n}_{F^e}\leq
 C\norm{Y}_{F^\alpha}$ for a constant $C$ which
 does not depend on $Y$ and $n$. One proves in the
 same way that $\norm{Y_n}_{F^\alpha}\leq
 C\norm{Y}_{F^\alpha}$.
\end{proof}

\subsection{Convergence of the Semi-Discrete
  solutions}
\label{subsec:convsdc}

Let $Y(t)$ and $Y_{\{\Delta\xi,R\}}(t)$ be
respectively the solution of \eqref{eq:sysban}
with initial data $Y_0$ and the solution of
\eqref{eq:sysbarr} with initial data
$Y_{0,\Delta\xi,R}$. We assume $Y_0\in F^e$. Given
$T>0$, we consider the fixed time interval
$[0,T]$. Since $Y_0\in\G$, the solution $Y(t)$
exists globally and
\begin{equation*}
  \sup_{t\in[0,T]}\norm{Y(t,\cdot)}_{F^e}\leq M
\end{equation*}
for a constant $M$ which depends only on $T$ and
$\norm{Y_0}_{F^e}$, see Theorems \ref{th:globsol}
and \ref{th:decay}. The solution
$Y_{\{\Delta\xi,R\}}$ does not necessarily exist
globally in time. However, we claim that there
exists $n>0$ such that for any $\Delta\xi$ and $R$
such that $\Delta\xi+\frac1R\leq\frac{1}{n}$, we
have
\begin{equation}
  \label{eq:boundapp}
  \sup_{t\in[0,T]}\norm{Y_{\{\Delta\xi,R\}}(t,\cdot)}<2M.
\end{equation}
It implies in particular that the solution
$Y_{\{\Delta\xi,R\}}$ is defined on $[0,T]$. Let
us assume the opposite. Then, there exists a
sequence $\Delta\xi_k$, $R_k$ and $t_k<T$ such
that $\lim_{k\to\infty}\Delta\xi_k=0$,
$\lim_{k\to\infty}R_k=\infty$,
\begin{equation*}
  \sup_{t\in[0,t_k]}\norm{Y_{\{\Delta\xi,R\}}(t,\cdot)}=2M.
\end{equation*}
From \eqref{eq:compares}, we get
\begin{equation}
  \label{eq:supY}
  \sup_{t\in[0,t_k]}\norm{Y(t,\cdot)-
    Y_{\{\Delta\xi_k,R_k\}}(t,\cdot)}_{F}\leq 
  C(M)\big(\norm{Y_0-Y_{0,\Delta\xi_k,R_k}}_{F}+\sqrt{\Delta\xi}+\e^{-R_k}\big).
\end{equation}
The constant $C(M)$ depends on $M$ but not on
$\Delta\xi_k$ and $R_k$.  Thus, we have
\begin{align*}
  2M=\sup_{t\in[0,t_k]}\norm{Y_{\{\Delta\xi_k,R_k\}}(t,\cdot)}&\leq\norm{Y(t_k,\cdot)}+\norm{Y(t_k,\cdot)-Y_{\{\Delta\xi_k,R_k\}}(t_k,\cdot)}\\
  &\leq
  M+C\big(\norm{Y_0-Y_{0,\Delta\xi,R}}_F+\sqrt{\Delta\xi_k}+\e^{-R_k}\big)
\end{align*}
which leads to a contradiction as the right-hand
side in the last inequality above tends to $M$
when $k$ tends to infinity. Once
\eqref{eq:boundapp} is established, Theorem
\ref{thm:convsemidisc} follows from
\eqref{eq:compares}. The same estimates can be
obtained for $Y_0\in F^{\alpha}$. Without loss of
generality, we assume that the approximating
sequence satisfies
$\norm{Y_0-Y_{0,\Delta\xi,R}}_F\leq\frac{C(M)}{2M}$
where $C(M)$ is given in \eqref{eq:supY}, so that 
$Y_{\{\Delta\xi,R\}}$ exists on $[0,T]$. Then,
we have the following theorem.
\begin{theorem}\label{thm:convsemidisc}
  Given $Y_0\in F^e$, for any $T>0$, there exists
  a constant $n>0$ such that, for all $\Delta\xi$
  and $R$ such that
  $\Delta\xi+\frac1R\leq\frac{1}n$, we have
  \begin{equation*}
    \sup_{t\in[0,T]}\norm{Y(t,\cdot)- Y_{\{\Delta\xi,R\}}(t,\cdot)}_{F}\leq 
C\big(\norm{Y_0- Y_{0,\Delta\xi,R}}_{F}+\sqrt{\Delta\xi}+\e^{-R}\big).
  \end{equation*}
  The constant $C$ depends only on
  $\norm{Y_0}_{F^e}$ and $T$. Correspondingly,
  given $Y_0\in F^\alpha$, we have
  \begin{equation*}
    \sup_{t\in[0,T]}\norm{Y(t,\cdot)- Y_{\{\Delta\xi,R\}}(t,\cdot)}_{F}\leq 
    C\big(\norm{Y_0- Y_{0,\Delta\xi,R}}_{F}+\sqrt{\Delta\xi}+\frac{1}{R^{\alpha/2}}\big)
  \end{equation*}
  and $C$ depends only on $\norm{Y_0}_{F^\alpha}$
  and $T$.
\end{theorem}

\section{Discretisation in time}
\label{sect-time}

In this section, we deal with the numerical
integration in time of the system of differential
equations \eqref{eq:sysbarr} which corresponds to
the semi-discretisation in space of system
\eqref{eq:sysban}. The flow of this system of 
differential equations has some geometric 
properties and it is of interest 
to derive numerical schemes that preserve 
these properties. Such 
integrators are called geometric 
numerical schemes, see for example the 
monograph \cite{hlw}. Thus we will look for
numerical schemes preserving the invariants
\eqref{eq:IDR} of our system of differential equations. 
Moreover, this last property will
enable us to show that the numerical schemes preserve the
positivity of the energy density. These invariants
are quartic functions of $Y$ and we are not aware
of schemes preserving quartic polynomials, this
is why we first split the system of equations
\eqref{eq:sysbarr} into two pieces. Each
sub-system will then have quadratic invariants and
we can use a numerical scheme preserving these
invariants. The following sub-systems read
\begin{align}\label{eq:sys1}
  \zeta_{i,t}&=0\nonumber\\
  U_{i,t}&=0\nonumber\\
  H_{i,t}&=0\nonumber\\
  q_{i,t}&=\gamma w_i\hspace{4cm} i=-N,\ldots,N-1\\
  w_{i,t}&=\sfrac{\gamma}{2}h_i+
  \bigl(\sfrac{3-2\gamma}{2}U_i^2-P_i\bigr)q_i\nonumber\\
  h_{i,t}&=\bigl(3U_i^2-2P_i\bigr)
  w_i,\nonumber
\end{align}
or shortly
\begin{equation*}
  \bar Y_{t}=\bar G_{1}(\bar Y),
\end{equation*}
where $\bar
Y(t)=\bigl(Y_{\{\Delta\xi,R\}}(t,\xi_i)\bigr)_{i=-N}^{N-1}$
and similarly for $\bar G_{1}$. 
We also define the system of differential
equations
\begin{align}\label{eq:sys2}
\zeta_{i,t}&=\gamma U_i\nonumber\\
U_{i,t}&=- Q_i\nonumber\\
H_{i,t}&= U_i^3-2 P_i U_i
\nonumber\\
q_{i,t}&=0\hspace{4cm} i=-N,\ldots,N-1\\
w_{i,t}&=0\nonumber\\
h_{i,t}&=-2 Q_i U_i q_i,\nonumber
\end{align}
or shortly
\begin{equation*}
  \bar Y_t= \bar G_2(\bar Y).
\end{equation*}
The space $F_{\{\Delta\xi,R\}}$ is finite
dimensional. We denote $\bar F=\Real^{2N\times 6}$.
The mapping from $\bar F$ to $F_{\{\Delta\xi,R\}}$
\begin{equation*}
  \left\{\bar Y_i=(\bar\zeta_i,\bar U_i,\bar H_i,\bar q_i,\bar w_i,\bar h_i)
\right\}_{i=-N}^{N-1}\mapsto Y=(\zeta,U,H,q,w,h)
\end{equation*}
is a bijection, where we define
\begin{equation*}
  \zeta(\xi)=\sum_{i=-N}^{N-1}\left(\bar\zeta_i\chi_{[\xi_i,\xi_{i+1})}(\xi)
  \right)+\bar\zeta_{-N}\chi_{(-\infty,-R]}(\xi)+\bar\zeta_{N}\chi_{[R,\infty]}(\xi)
\end{equation*}
and similar definitions for the other components
of $Y$. This mapping is in addition an isometry if
we consider the norm
\begin{multline}
  \label{eq:defnorm2N}
  \norm{\bar Y}_{\bar F}=\norm{\bar
    \zeta}_{l^\infty(\Real^{2N})}+\norm{\bar
    U}_{l^2(\Real^{2N})}+\norm{\bar
    U}_{l^\infty(\Real^{2N})}+\norm{\bar
    H}_{l^\infty(\Real^{2N})}\\+ \norm{\bar
    v}_{l^2(\Real^{2N})}+\norm{\bar
    w}_{l^2(\Real^{2N})}+\norm{\bar
    h}_{l^2(\Real^{2N})},
\end{multline}
where 
\begin{equation*}
  \norm{\bar z}_{l^2(\Real^{2N})}=(\Delta\xi\sum_{i=-N}^{N-1}
\bar z_i^2)^{\frac12}
\end{equation*}
for any $\bar z\in \Real^{2N}$. In the remaining,
we will always consider the norm given by
\eqref{eq:defnorm2N} for $\bar F$ so that the
bounds found in the previous sections directly
apply. In particular, we have the following lemma,
which is a consequence of Proposition
\ref{prop:LipG} and the same arguments that lead
to Lemmas \ref{lem:LipGbar} and
\ref{lem:LipGbarr}.
\begin{lemma}
  \label{lem:LipG1G2}
  The mappings $\bar G_1:\bar F\to \bar F$ and $\bar G_2:\bar
  F\to \bar F$ belong to $C^1(\bar F,\bar F)$ and
  \begin{equation*}
    \norm{\bar G_1(\bar Y)}_{\bar F}+\norm{\fracpar{\bar G_1}{\bar Y}(\bar Y)}_{L(\bar F,\bar F)}\leq C(M),
  \end{equation*}
  and
  \begin{equation*}
    \norm{\bar G_2(\bar Y)}_{\bar F}+\norm{\fracpar{\bar G_2}{\bar Y}(\bar Y)}_{L(\bar F,\bar F)}\leq C(M),
  \end{equation*}
  for any $\bar Y\in \bar B_M$, where
  \begin{equation*}
    \bar B_M=\{\bar Y\in \bar F\ |\ \norm{\bar Y}_{\bar F}\leq M\}.
  \end{equation*}
\end{lemma}
As this was done in the last sections, one can
show that both systems posses $\bar
I_i(Y)=U_i^2q_i^2+w_i^2-q_ih_i$, see
\eqref{eq:IDR}, as first integrals. That is $\bar
I_i^\prime(Y)\bar G_k(Y)=0$ for all $Y$, for
$k=1,2$ and for $i=-N,\ldots,N-1$. In particular,
this implies that every solutions of
\eqref{eq:sys1} or \eqref{eq:sys2} satisfy $\bar
I_i(\bar Y(t))= \bar I_i(\bar Y(0))$ for
$i=-N,\ldots,N-1$ and $t\geq0$. Having a closer
look at the differential equations \eqref{eq:sys1}
and \eqref{eq:sys2}, one sees that the invariants
are now quadratic functions ($\bar U$ is constant
for \eqref{eq:sys1} and $\bar q$ is constant for
\eqref{eq:sys2}) and we therefore use a numerical
scheme that preserves quadratic invariants.

\begin{proposition}\label{prop:RKinv}
Let us apply a Runge-Kutta scheme with coefficients 
satisfying 
\begin{equation}\label{eq:RKcoeff}
b_ia_{ij}+b_ja_{ji}=b_ib_j
\quad\:\:\:\text{for all}
\:\:\:\quad i,j=1,\ldots,s
\end{equation}
to the system \eqref{eq:sys1}, then it conserves
exactly the invariants $\bar I_i(Y)=U_i^2q_i^2+w_i^2-q_ih_i$ 
for $i=-N,\ldots,N-1$.  The same holds if we apply
the scheme to \eqref{eq:sys2}.
\end{proposition}
\begin{proof}
  The proof of this proposition is a simple
  adaptation of the proof of Theorem 2.2 from
  \cite[Chapter IV]{hlw}. Let us start with system
  \eqref{eq:sys1}.  Dropping the indexes and the
  bars for ease of notations, we first write the
  invariant $I(Y)$ as
  $$
  I(Y)=Y^TD(Y)Y+d(Y)^TY
  $$
  with $Y=(\zeta,U,H,q,w,h)$, 
  $D(Y)=
  \begin{pmatrix} 
    0 & 0 & 0 & 0 & 0 & 0\\
    0 & 0 & 0 & 0 & 0 & 0\\
    0 & 0 & 0 & 0 & 0 & 0\\
    0 & 0 & 0 & U^2 & 0 & -1/2\\
    0 & 0 & 0 & 0 & 1 & 0\\
    0 & 0 & 0 & -1/2 & 0 & 0
  \end{pmatrix}$ and $d(Y)=0^T$.
  For the Runge-Kutta method, we write
  $Y_1=Y_0+h\sum_{j=1}^sb_jK_j$ with
  $K_i=G_1(Y_0+h\sum_{j=1}^sa_{ij}K_j)$.  From the
  definition of the method, of the matrix $D(Y)$
  and of the vector $d(Y)$, it follows that
  \begin{align*}
    I(Y_1)&=Y_1^TD(Y_1)Y_1+d(Y_1)^TY_1=
    (Y_0+h\sum_{i=1}^sb_iK_i)^TD(Y_0)(Y_0+h\sum_{j=1}^sb_jK_j)\\
    &=Y_0^TD(Y_0)Y_0+h\sum_{i=1}^sb_iK_i^TD(Y_0)Y_0
    +h\sum_{j=1}^sb_jY_0^TD(Y_0)K_j\\
    &+h^2\sum_{i,j=1}^sb_ib_jK_i^TD(Y_0)K_j.
  \end{align*}
  Writing $K_i=G_1(\tilde Y_i)$ with 
  $\tilde Y_i=Y_0+h\sum_{j=1}^sa_{ij}K_j$,  
  we obtain that
  \begin{align*}
    I(Y_1)&=Y_0^TD(Y_0)Y_0+
    2h\sum_{i=1}^sb_i\tilde Y_i^TD(Y_0)G_1(\tilde Y_i)\\
    &+h^2\sum_{i,j=1}^s(b_ib_j-b_ia_{ij}-b_ja_{ji})K_i^TD(Y_0)K_j.
  \end{align*}
  The last term in the above equation vanishes due
  to condition \eqref{eq:RKcoeff}. By definition
  of the problem and of the matrix $D(Y)$, we have
  $D(Y_0)=D(\tilde Y_i)$ because $U$ is preserved
  and since $I(Y)$ is a first integral for
  \eqref{eq:sys1}, we get $\tilde Y_i^TD(\tilde
  Y_i)G_1(\tilde Y_i)=0$.  It thus follows
  $$
  I(Y_1)=Y_0^TD(Y_0)Y_0+0=I(Y_0)
  $$
  and the Runge-Kutta scheme applied to 
  \eqref{eq:sys1} conserves the invariant $I(Y)$.
  
  The proof for system \eqref{eq:sys2} is similar, take 
  $D(Y)=
  \begin{pmatrix} 
    0 & 0 & 0 & 0 & 0 & 0\\
    0 & q^2 & 0 & 0 & 0 & 0\\
    0 & 0 & 0 & 0 & 0 & 0\\
    0 & 0 & 0 & 0 & 0 & 0\\
    0 & 0 & 0 & 0 & 1 & 0\\
    0 & 0 & 0 & 0 & 0 & 0
  \end{pmatrix} 
  $ and $d(Y)=(0,0,0,0,0,-q)^T$.
\end{proof}
Let us consider the following differential equation 
$y_t(t)=f(y(t))$. The implicit midpoint rule
$$
y_1=y_0+\Delta t f(\sfrac{y_1+y_0}{2})
$$
satisfies the condition \eqref{eq:RKcoeff} and thus 
preserves quadratic invariants. The implicit midpoint 
rule will be the building block for the construction of 
the schemes we will use for the numerical 
experiments in Section~\ref{sect-numexp}. 
For other schemes preserving quadratic 
invariants, we refer to \cite{hlw} for example.

As a direct consequence of Proposition~\ref{prop:RKinv}, we 
have the following result. 
\begin{theorem}\label{th:RKinv}
Let us apply a Runge-Kutta scheme $\Phi^1_{\Delta t}$, 
resp. $\Phi^2_{\Delta t}$, with 
coefficients satisfying \eqref{eq:RKcoeff} 
to the system \eqref{eq:sys1}, resp. \eqref{eq:sys2}, 
with time step size $\Delta t$. 
Then the Lie-Trotter splitting 
\begin{equation*}
  \Phi_{\Delta t}:=\Phi^2_{\Delta t}\circ\Phi^1_{\Delta t}
\end{equation*}
has order of convergence one and preserves all the
invariants $\bar I_i$ for $i=-N,\ldots,N-1$. The
Strang splitting
\begin{equation*}
  \Phi_{\Delta t}:=\Phi^1_{\Delta t/2}\circ\Phi^2_{\Delta t}\circ\Phi^1_{\Delta t/2}
\end{equation*}
is symmetric, has thus order of convergence two and 
preserves all the invariants 
$\bar I_i$ for $i=-N,\ldots,N-1$.
\end{theorem}
If we take for $\Phi^i_{\Delta t}$, $i=1,2$, the
implicit midpoint rule, we obtain a first order 
splitting scheme for \eqref{eq:sysbarr} that 
preserve exactly the invariants (a second order scheme 
is obtained using the Strang splitting). 
This will be the schemes that we will
consider in the numerical experiments of
Section~\ref{sect-numexp}.

\section{Full discretisation}\label{sect-full}

Our concern is now to combine the results from the
last two sections and to show that our numerical
schemes converge to the exact solution of the
system of equations \eqref{eq:sysban}. We
integrate $\bar Y(t)$ on the time interval $[0,T]$
and obtain $\bar Y_{j}$ for the time steps
$j\Delta t$, $j=0,\ldots,N_T$ where $\Delta
t=\frac{T}{N_T}$. We have the following
convergence result.
\begin{theorem}\label{th:full} Given initial
  values $Y_0$ in $F^e$ and $\bar Y_0\in F_R$, for
  the Lie-Trotter splitting we have
  \begin{equation}
    \label{eq:convfull}
    \max_{j\in\{0,\ldots,N_T\}}\norm{S_{j\Delta t}(Y_0)- \Phi_{j\Delta t}(\bar Y_0)}_F\leq
    C\bigl(
    \norm{Y_0-\bar Y_0}_{F}+\sqrt{\Delta\xi}+\e^{-R}+\Delta t\bigr),
  \end{equation}
  where we recall that $S_t$ stands for the
  semigroup of solutions to \eqref{eq:sysban} and,
  where the constant $C$ depends only on
  $\norm{Y_0}_{F^e}$, $\norm{\bar Y_0}_{F^e}$ and
  $T$. Correspondingly, given initial values $Y_0$
  in $F^\alpha$ and $\bar Y_0\in F_R$, we have
  \begin{equation}
    \label{eq:convfull2}
    \max_{j\in\{0,\ldots,N_T\}}\norm{S_{j\Delta t}(Y_0)- \Phi_{j\Delta t}(\bar Y_0)}_F\leq
    C\bigl(
    \norm{Y_0-\bar Y_0}_{F}+\sqrt{\Delta\xi}+\frac{1}{R^{\alpha/2}}+\Delta t\bigr),
  \end{equation}
  where the constant $C$ depends only on
  $\norm{Y_0}_{F^\alpha}$, $\norm{\bar
    Y_0}_{F^\alpha}$ and $T$. The same results
  hold for the Strang splitting with second order
  accuracy in time, that is, when we replace
  $\Delta t$ with $\Delta t^2$ in
  \eqref{eq:convfull}.
\end{theorem}
Let us denote $Y(t)=S_t(Y_0)$ and
\begin{equation*}
  \Phi_t(\bar Y_0)=\frac{((j+1)\Delta t-t)\Phi_{j\Delta t}(\bar Y_0)+(t-j\Delta t)\Phi_{(j+1)\Delta t}(\bar Y_0)}{\Delta t}
\end{equation*}
for $t\in[j\Delta t,(j+1)\Delta t]$. We can
rewrite \eqref{eq:convfull} as
\begin{equation*}
  \max_{t\in[0,T]}\norm{S_{t}(Y_0)- \Phi_{t}(\bar Y_0)}_F\leq
  C\bigl(
  \norm{Y_0-\bar Y_0}_{F}+\sqrt{\Delta\xi}+\e^{-R}+\Delta t\bigr).
\end{equation*}
\begin{proof}[Proof of Theorem \ref{th:full}]
To estimate the total error 
$$
\norm{S_{j\Delta t}(Y_0)- \Phi_{j\Delta t}(\bar Y_0)}_F   
$$
we split it in time and in space.  Let us start
with the error in time. The proof follows
basically the steps of the standard proof of the
convergence of numerical scheme for ordinary
differential equations. The crucial point is that
we guarantee here that the convergence rate in
time is independent of the discretisation step in
space. Let us first prove the following claim:
Given $M>0$, for any $\bar Y\in\bar B_M$ and $\bar
Z\in \bar B_M$, we have
  \begin{multline}
    \label{eq:claimconv}
    \Phi_{\Delta t}(\bar Y)-\phi_{\Delta t}(\bar
    Z)= \bar Y-\bar Z\\+\Delta t\left(\bar
      G_1(\bar Y)-\bar G_1(\bar Z)+\bar
      G_2(\bar Y)-\bar G_2(\bar
      Z)\right)+\bigo{\Delta t^2},
  \end{multline}
  where $\phi_{\Delta t}(\bar Z)$ stands for the 
exact flow of \eqref{eq:sysbarr} at time $\Delta t$ with starting 
values $\bar Z$. 
  Here, and in the following, the
  $\mathcal{O}$-notation stands for an element in
  $\bar F$ satisfying
  \begin{equation*}
    \norm{\bigo{\epsi}}_{\bar F}\leq C(M)\epsi
  \end{equation*}
  for all $\epsi>0$, where the constant $C(M)$
  depends on $M$ but is independent on $R$ and on
  the space grid size $\Delta\xi$. We first show
  that the midpoint rule
  \begin{equation*}
    \Phi^j_{\Delta t}(\bar Y)= \bar
    Y+\Delta t\bar G_j\left(\frac{\Phi^j_{\Delta
          t}(\bar Y)+\bar Y}{2}\right),
  \end{equation*}
  applied to equation \eqref{eq:sys1},
  resp. \eqref{eq:sys2}, is at least first order
  accurate. To do this, 
  let us introduce the mapping $K:\bar F\times\bar
  F\to \bar F$ given by
  \begin{equation*}
    K(\bar Z,\bar Y)=\bar Z-\bar Y-\Delta t\bar G_{1}
     \bigl(\frac{\bar Z+\bar Y}2\bigr).
  \end{equation*}
  We have $K(\Phi^1_{\Delta t}(\bar Y),\bar
  Y)=0$. Since
  \begin{equation*}
    \fracpar{K}{\bar Z}(\bar Y)=\id-\frac{\Delta t}{2}
     \fracpar{\bar G_1}{\bar Y}\bigl(\frac{\bar Z+\bar Y}2\bigr)
  \end{equation*}
  and $\norm{\fracpar{\bar G_1}{\bar Y}(\bar
    Y)}_{\bar F}\leq C(M)$ (by
  Lemma~\ref{lem:LipG1G2}), there exist $C(M)$
  such that, for $\Delta t\leq\frac{1}{C(M)}$, we
  have that $\fracpar{K}{\bar Z}(\bar Y)$ is
  invertible. By the implicit function Theorem, we
  get that $\Phi^1_{\Delta t}(\bar Y)$ is
  well-defined. Moreover, also following from the
  implicit function Theorem, we get that
  \begin{equation*}
    \norm{\Phi^1_{\Delta t}(\bar Y)}_{\bar F}\leq C(M).
  \end{equation*}
  Then,
  \begin{align*}
    \Phi^1_{\Delta t}(\bar Y)&=\bar Y+\Delta
    t\bar G_1\left(\bar Y+\frac{\Delta t}2\bar
      G_1\left(\frac{\Phi^1_{\Delta t}(\bar
          Y)+\bar Y}{2}\right)\right)\\
    &=\bar Y+\Delta t\bar G_1(\bar
    Y)+\bigo{\Delta t^2}
  \end{align*}
  by Lemma \ref{lem:LipG1G2}. Using Lemma
  \ref{lem:LipG1G2} again, we obtain for the exact
  flow of \eqref{eq:sys1} that
  \begin{equation*}
    \phi^1_{\Delta
      t}(\bar Z)=\bar
    Z+\Delta t\bar G_1(\bar Z)+\bigo{\Delta t^2}.
  \end{equation*}
  Following the same arguments, we obtain that
  \begin{equation*}
    \notag \Phi^2_{\Delta t}(\Phi^1_{\Delta
      t}(\bar Y))=\Phi^1_{\Delta t}(\bar
    Y)+\Delta t \bar G_2(\Phi^1_{\Delta t}(\bar
    Y))+\bigo{\Delta t^2}
  \end{equation*}
  and for the composition of the exact flows
  \begin{equation*}
    \phi^2_{\Delta
      t}(\phi^1_{\Delta t}(\bar Z))=\phi^1_{\Delta t}(\bar Z)+\Delta t\bar G_2(\phi^1_{\Delta t}(\bar Z))+\bigo{\Delta t^2}.
  \end{equation*}
  Hence,
  \begin{align}
    \notag \Phi^2_{\Delta t}(&\Phi^1_{\Delta
      t}(\bar Y))-\phi^2_{\Delta
      t}(\phi^1_{\Delta t}(\bar Z))\\
    \notag&=\Phi_{\Delta t}^1(\bar Y)+\Delta t\bar
    G_2(\Phi_{\Delta t}^1(\bar Y))-\phi_{\Delta
      t}^1(\bar Z)-\Delta t\bar G_2(\phi_{\Delta
      t}^1(\bar
    Z))+\bigo{\Delta t^2}\\
    \notag&=\bar Y-\bar Z+\Delta t(\bar G_1(\bar
    Y)-\bar G_1(\bar
    Z))\\
    \notag&\quad+\Delta t\left(\bar G_2(\bar Y+\Delta
      t\bar G_1(\bar Y )+\bigo{\Delta t^2})-\bar
      G_2(\bar Z+\Delta t\bar G_1(\bar Z
      )+\bigo{\Delta t^2})\right)\\
    \notag&\quad+\bigo{\Delta t^2}\\
    \label{eq:Phimphi}&=\bar Y-\bar Z+\Delta t\bigl(\bar G_1(\bar
    Y)-\bar G_1(\bar Z)+\bar G_2(\bar
    Y)-\bar G_2(\bar Z)\bigr)+\bigo{\Delta t^2}.
  \end{align}
  We consider now the splitting error. We have
  \begin{equation*}
    \phi_{\Delta t}(\bar Z)-\bar Z=
    {\Delta t} \bar G(\bar Z) +\bigo{\Delta t^2}
  \end{equation*}
  and
  \begin{equation*}
    \phi^1_{\Delta t}(\bar Z)-\bar Z=
    \Delta t\bar G_1(\bar Z)+\bigo{\Delta t^2}
  \end{equation*}
  and thus 
  \begin{align*}
    \phi^2_{\Delta t}(\phi^1_{\Delta t}(\bar
    Z))=\phi^1_{\Delta t}(\bar Z)+ \Delta t\bar
    G_2(\phi^1_{\Delta t}(\bar Z))+\bigo{\Delta
      t^2}.
  \end{align*}
  Hence,
  \begin{align}\notag
    \phi^2_{\Delta t}(\phi^1_{\Delta t}(\bar
    Z))-\phi_{\Delta t}(\bar Z)&=\Delta t\bar
    G(\bar Z)-\Delta t\bar G_1(\bar
    Z)\\\notag&\quad -\Delta t\bar G_2(\bar
    Z+\Delta t\bar G_1(\bar Z)+\bigo{\Delta
      t^2})+\bigo{\Delta
      t^2}\\
    \notag&=\Delta t(\bar G(\bar Z)-\bar
    G_1(\bar
    Z)-\bar G_2(\bar Z))+\bigo{\Delta t^2}\\
    \label{eq:splitest}&=\bigo{\Delta t^2},
  \end{align}
  as $\bar G=\bar G_1+\bar G_2$. Combining
  \eqref{eq:splitest} and \eqref{eq:Phimphi}, we
  obtain \eqref{eq:claimconv} and the claim is
  proved. Let us now set
  \begin{equation*}
    M=\sup_{t\in[0,T]}\norm{\phi_t(\bar Y_0)}_F.
  \end{equation*}
  For a given $\Delta t$, we define
  \begin{equation}
    \label{eq:jdeltat}
    j_{\Delta t}=\max\{j\in\{0,\ldots, N_T-1\}\ |\ 
    \norm{\Phi_{\bar j\Delta t}(\bar Y_0)}_{\bar F}
     \leq 2M\text{ for all }\bar j\leq j\}.
  \end{equation}
  For $j\leq j_{\Delta t}$, we get from \eqref{eq:claimconv} that 
  \begin{equation*}
    \norm{\Phi_{(j+1)\Delta t}(\bar Y_0)-
    \phi_{(j+1)\Delta t}(\bar Y_0)}_{F}\leq(1+C(M)\Delta t)
    \norm{\Phi_{(j)\Delta t}(\bar Y_0)-\phi_{(j)\Delta t}(\bar Y_0)}_{F}
    +\bigo{\Delta t^2}.
  \end{equation*}
  By induction, it follows that
  \begin{align*}
    \norm{\Phi_{(j+1)\Delta t}(\bar Y_0)-\phi_{(j+1)\Delta t}
    (\bar Y_0)}_{F}&\leq\norm{\bigo{\Delta t^2}}\sum_{k=0}^{j}(1+C(M)\Delta t)^k\\
    &\leq\norm{\bigo{\Delta t^2}}\frac{1}{C(M)\Delta t}
  \end{align*}
  and therefore
  \begin{equation}
    \label{eq:bdPhiphiglob}
    \Phi_{(j+1)\Delta t}(\bar Y_0)=\phi_{(j+1)\Delta t}(\bar Y_0)+\bigo{\Delta t}.
  \end{equation}
  We claim that there exists a constant $C(M)$
  such that for all $\Delta t\leq \frac1{C(M)}$,
  we have $j_{\Delta t}=N_T-1$ and therefore
  \eqref{eq:bdPhiphiglob} holds for all $j\leq
  N_T-1$. Let us assume the opposite. Then, there
  exists $\Delta t_k$ such that
  $\lim_{k\to\infty}\Delta t_k=0$ and $j_{\Delta
    t_k}<N_T-1$. By definition \eqref{eq:jdeltat},
  we have $\norm{\Phi_{(j_{\Delta t_k}+1)\Delta
      t_k}(\bar Y_0)}_{\bar F}>2M$. Then,
  \eqref{eq:bdPhiphiglob} implies
  \begin{align*}
    2M&\leq \norm{\Phi_{(j_{\Delta t_k}+1)\Delta t_k}(\bar Y_0)
    -\phi_{(j_{\Delta t_k}+1)\Delta t_k}(\bar Y_0)}_{F}+
    \norm{\phi_{(j_{\Delta t_k}+1)\Delta t_k}(\bar Y_0)}_{F}\\
    &\leq\bigo{\Delta t_k}+M
  \end{align*}
  which leads to a contradiction when $k$ tends to
  $\infty$. Finally, for the total error in space
  and time, we have:
  \begin{equation*}
    \norm{S_{j\Delta t}(Y_0)- \Phi_{j\Delta t}(\bar Y_0)}_F\leq \norm{S_{j\Delta t}(Y_0)- \phi_{j\Delta t}(\bar Y_0)}_F+\norm{\phi_{j\Delta t}(\bar Y_0)- \Phi_{j\Delta t}(\bar Y_0)}_F,
  \end{equation*}
  where all the functions are evaluated at time
  $j\Delta t$ for $j\leq N_T$.  The first term can
  be estimate using Theorem~\ref{thm:convsemidisc}
  and we thus obtain
  \begin{equation*}
    \max_{j\in\{0,\ldots,N_T\}}\norm{S_{j\Delta t}(Y_0)- \phi_{j\Delta t}(\bar Y_0)}_F\leq
    C\bigl(
    \norm{Y_0-\bar Y_0}_{F}+\sqrt{\Delta\xi}+\e^{-R}\bigr).
  \end{equation*}
  For the second one we use
  \eqref{eq:bdPhiphiglob} and this concludes the
  proof of the theorem for the Lie-Trotter
  splitting. If we had taken the Strang splitting
  instead, we would have obtained an error in time
  of order two since this scheme is symmetric. The
  proof for initial data in $F^{\alpha}$ is the 
  same.
\end{proof}

Our next task will be to show that our schemes
preserve the positivity of the particle density
and of the energy density as does the exact
solution of \eqref{eq:sysban} with initial data
given by Theorem~\ref{th:init}.  In order to prove
this result, we introduce $F^{\infty}$ defined as
\begin{equation*}
  F^\infty=\{Y=(y,U,H,q,w,h)\in F\ |\ \norm{q}_{L^\infty}+\norm{w}_{L^\infty}+\norm{h}_{L^\infty}<\infty\}
\end{equation*}
with the norm
\begin{equation*}
  \norm{Y}_{F^{\infty}}=\norm{Y}_{F}+\norm{q}_{L^\infty}+\norm{w}_{L^\infty}+\norm{h}_{L^\infty}.
\end{equation*}
We know that the space $F^{\infty}$ is preserved
by the governing equations \eqref{eq:sysban}, see
Lemma \ref{lem:presprop}. Using the semilinear
structure of \eqref{eq:sysban4}-\eqref{eq:sysban6}
with respect to $q$, $w$, $h$, one can show in the
same way that \eqref{eq:bdderG} was shown, that,
for a given $M>0$,
\begin{equation}
  \label{eq:LipbdFinfty}
  \norm{G(Y)}_{F^\infty}+\norm{\fracpar{G}{Y}(Y)}_{L(F^\infty,F^\infty)}\leq C(M)
\end{equation}
for any $Y\in B_M^\infty=\{Y\in F^\infty\ |\
\norm{Y}_{F^\infty}\leq M\}$. The same result
holds for the mappings $G_{\Delta\xi}$,
$G_{\Delta\xi,R}$, $\bar G_1$ and $\bar G_2$. In
particular we can prove, as in Theorem
\ref{th:full} for the proof of
\eqref{eq:claimconv}, that
\begin{multline*}
  \Phi_{\Delta t}(\bar Y)-\phi_{\Delta t}(\bar Z)=
  \bar Y-\bar Z\\+\Delta t\left(\bar G_1(\bar
    Y)-\bar G_1(\bar Z)+\bar G_2(\bar Y)-\bar
    G_2(\bar Z)\right)+\bigo{\Delta t^2},
\end{multline*}
where the definition of $\bigo{\cdot}$ is replaced
by
\begin{equation*}
  \norm{\bigo{\epsi}}_{\bar F^{\infty}}\leq C(M)\epsi.
\end{equation*}
Here, $\bar F^\infty= \bar F=\Real^{2N\times 6}$
but equipped with the norm derived from
$\norm{\cdot}_{F^\infty}$, see
\eqref{eq:defnorm2N}.

\begin{theorem}\label{th:fullinv} We consider an
  initial data which satisfy
  \begin{equation*}
    q_i^0h_i^0\geq(U_i^0q_i^0)^2+(w_i^0)^2,\quad q_i^0\geq 0,\quad h_i^0\geq 0\text{ and }q_i^0+h_i^0\geq c
  \end{equation*}
  for all $i=-N,\ldots, N-1$, for some constant
  $c>0$. Then, given $T>0$, there exists $n>0$,
  which depends only on $c$, $\norm{\bar
    Y^0}_{F^\infty}$ and $T$, such that if,
  $\Delta\xi+\frac{1}R+\Delta t<\frac 1{n}$,the
  positivity of the particle density $1/q$ and of
  the energy density $h$ are preserved by our numerical discretisation, that is,
  \begin{equation*}
    q_i^j\geq 0\quad\text{ and }\quad h_i^j\geq 0,
  \end{equation*}
  for $i=-N,\ldots, N-1$ and $j=1,\ldots,N_T$.
\end{theorem}

\begin{proof} The main idea of the proof is to
  control the growth of $1/(q^k_i+h^k_i)$. To do
  so we adapt the proof of
  Lemma~\ref{lem:presprop} to this discrete
  situation.  Let $\displaystyle
  M=2\sup_{t\in[0,T]}\norm{\phi_t(\bar
    Y_0)}_{F^\infty}$.  As in the proof of Theorem
  \ref{th:full}, we can prove that for $\Delta t$
  small enough (the bound depending only on $M$),
  we have
  \begin{equation*}
    \norm{\Phi_{k\Delta t}(\bar Y_0)}_{F^\infty}\leq 2M
  \end{equation*}
  for all $k=0,\ldots,N_T$.  For $k<N_T$, we have, by definition of our 
  scheme, that 
  \begin{align*}
    \frac{1}{q_i^{k+1}+h_i^{k+1}}-\frac{1}{q_i^{k}+h_i^{k}}&=-\frac{q_i^{k+1}-q_i^k+h_i^{k+1}-h_i^k}{(q_i^{k+1}+h_i^{k+1})(q_i^{k}+h_i^{k})}\\
    &=-\frac{\Delta t\bigl(\gamma
      w_i^k-2Q(Y^k)U_i^kq_i^k+(3(U_i^k)^2-2P(Y^k))w_i^k\bigr)+\bigo{\Delta
        t^2}}{(q_i^{k+1}+h_i^{k+1})(q_i^{k}+h_i^{k})}.
  \end{align*}
  Hence, using the bounds \eqref{eq:LipbdFinfty}, we get 
  \begin{equation}
    \label{eq:iterstep1}
    \abs{\frac{1}{q_i^{k+1}+h_i^{k+1}}-\frac{1}{q_i^{k}+h_i^{k}}}\leq \frac{\Delta t C(M)}{\abs{q_i^{k+1}+h_i^{k+1}}}\left(\frac{\abs{w_i^k}+\abs{q_i^k}+\Delta t}{\abs{q_i^{k}+h_i^{k}}}\right).
  \end{equation}
  Let us prove by induction that, for $\Delta t$
  small enough (depending only $M$),
  \begin{equation}
    \label{eq:indstatement}
    \frac1{q_i^k+h_i^k}\leq \frac1c\e^{2C(M)T}+1,\quad 
    q_i^k\geq 0\text{ and }h_i^k\geq 0
  \end{equation}
  for $i=-N,\ldots, N-1$, all $k=0,\ldots,N_T$ and
  where $C(M)$ is the constant given in
  \eqref{eq:iterstep1}. By definition of our
  initial data, these assumptions hold for $k=0$.
  We assume now that \eqref{eq:indstatement} holds
  for $k=0,\ldots,j$ and we want to prove that it
  also holds for $j+1$. We set $\bar
  M=\frac1c\e^{2C(M)T}+1$. Since the numerical
  schemes preserve the invariant
  $q_i^kh_i^k=(U_i^kq_i^k)^2+(w_i^k)^2$, we obtain
  in particular that
  \begin{equation}
    \label{eq:posforallk}
    q_i^kh_i^k\geq (U_i^kq_i^k)^2+(w_i^k)^2
  \end{equation}
  for all $k=0,\ldots,N_T$. From this, it follows that 
  $\abs{w_i^k}\leq\frac1{\sqrt{2}} (q_i^k+h_i^k)$
  as $q_i^k\geq 0$ and $h_i^k\geq0$. For $k\leq
  j$, we get from \eqref{eq:iterstep1} and our induction hypothesis that
  \begin{equation}
    \label{eq:iterstep2}
    \abs{\frac{1}{q_i^{k+1}+h_i^{k+1}}-\frac{1}{q_i^{k}+h_i^{k}}}\leq 
\frac{\Delta t C(M)}{\abs{q_i^{k+1}+h_i^{k+1}}}\left(1+\frac1{\sqrt{2}}+\bar M\Delta t\right).
  \end{equation}
  From the above equation, we get
  \begin{equation*}
    \abs{\frac{1}{q_i^{k+1}+h_i^{k+1}}}\leq\frac{1}{1-2C(M)\Delta t-\bar MC(M)\Delta t^2}\abs{\frac{1}{q_i^{k}+h_i^{k}}}
  \end{equation*}
  and therefore
  \begin{align*}
    \label{eq:estuntijp1}
    \abs{\frac{1}{q_i^{j+1}+h_i^{j+1}}}&\leq
    \frac{1}{(1-2C(M)\Delta t-\bar MC(M)\Delta
      t^2)^{j}}
    \abs{\frac{1}{q_i^{0}+h_i^{0}}}\\
    &\leq\frac{1}{c(1-2C(M)\Delta t-\bar
      MC(M)\Delta t^2)^{\frac{T}{\Delta t}}}.
  \end{align*}
  We have
  \begin{equation*}
    \lim_{\Delta t\to0}\frac{1}{c(1-2C(M)\Delta t-
      \bar MC(M)\Delta t^2)^{\frac{T}{\Delta t}}}=\frac{1}{c}\e^{2C(M)T}<\bar M.
  \end{equation*}
  Therefore, by taking $\Delta t$ small enough,
  depending only on the value of $M$ and not on
  the number of induction steps $j$, we get
  \begin{equation*}
    \abs{\frac{1}{q_i^{j+1}+h_i^{j+1}}}\leq\bar M.
  \end{equation*}
  Using the above inequality and \eqref{eq:iterstep2}, we obtain 
  $$
  -\frac{1}{q_i^{j+1}+h_i^{j+1}}+\frac{1}{q_i^{j}+h_i^{j}}\leq 
  \bar M\Delta t C(M)\left(1+\frac1{\sqrt{2}}+\bar M\Delta t\right)  
  $$
  so that $\frac{1}{q_i^{j+1}+h_i^{j+1}}\geq 0$ for a sufficiently 
  small $\Delta t$. By
  \eqref{eq:posforallk}, we have that
  $q_i^{j+1}h_i^{j+1}\geq 0$ and therefore 
  \begin{equation*}
    q_i^{j+1}\geq 0 \text{ and }h_i^{j+1}\geq 0,
  \end{equation*}
  which concludes our proof by induction.
\end{proof}

Now we go back to the original set of
coordinates. Given an initial data $u_0\in
H^{1,e}(\Real)$ or $H^{1,\alpha}(\Real)$, we
construct the initial data $Y_0$ as given by
\eqref{eq:initdatex}. Then the function $u(t,x)$
defined as
\begin{equation}
  \label{eq:defuex}
  u(t,x)=U(t,\xi)\quad\text{ for }\quad y(t,\xi)=x
\end{equation}
is well-defined, is a weak solution to
\eqref{eq:hr2} which corresponds to the global
conservative solution.  The definition
\eqref{eq:defuex} of $u(t,x)$ means that for any
given time $t$ the set of points
\begin{equation*}
  (y(t,\xi),U(t,\xi))\in\Real^2\text{ for }\xi\in\Real
\end{equation*}
is the graph of $u(t,x)$. Let
$\frac1{n}=\Delta\xi+\frac1{R}+\Delta t$ so that
$n$ tends to infinity if and only if $\Delta\xi$,
$\Delta t$ tend to zero and $R$ tends to
infinity. We consider an approximating sequence
$Y_{0,n}$ which satisfies the conditions
\eqref{eq:convYn} and \eqref{eq:bdappsedec} of the
sequence of initial values which is constructed in
Section \ref{sect-appdata}. Let
$Y_n(t)=\Phi(Y_{0,n})$. From Theorem
\ref{th:full}, we obtain the following convergence
theorem.
\begin{theorem} The full discretised scheme
  provide us with points which converge to the
  graph of the exact conservative solution
  $u(t,x)$. Indeed, if $u_0\in H^{1,e}(\R)$, we have
  \begin{multline*}
    \max_{\substack{i=-N,\ldots,N-1\\j=0,\ldots,N_T}}
        \abs{(y_n(t_j,\xi_i),U_n(t_j,\xi_i))
       -(y(t_j,\xi_i),U(t_j,\xi_i))}\\\leq
    C\bigl(\norm{Y_0-\bar
      Y_0}_{F}+\sqrt{\Delta\xi}+\e^{-R}+\Delta
    t\bigr),
  \end{multline*}
  where the constant $C$ depends only on
  $\norm{u_0}_{H^{1,e}}$ and, if $u_0\in
  H^{1,\alpha}(\R)$,
  \begin{multline}
    \label{eq:errorest}
    \max_{\substack{i=-N,\ldots,N-1\\j=0,\ldots,N_T}}
    \abs{(y_n(t_j,\xi_i),U_n(t_j,\xi_i))
      -(y(t_j,\xi_i),U(t_j,\xi_i))}\\\leq
    C\bigl(\norm{Y_0-\bar
      Y_0}_{F}+\sqrt{\Delta\xi}+\frac{1}{R^{\alpha/2}}+\Delta
    t\bigr),
  \end{multline}
  where the constant $C$ depends only on
  $\norm{u_0}_{H^{1,\alpha}}$.
\end{theorem}
Since
\begin{align*}
  \abs{y(t,\xi_{i+1})-y(t,\xi_i)}=
  \abs{\int_{\xi_i}^{\xi_{i+1}}q(t,\xi)\,d\xi}\leq
  C\Delta\xi,
\end{align*}
where $C$ depends only on $\norm{Y_0}_{F^\infty}$,
we have an apriori upper bound on the density of
points of the graph of $u$ we can approximate by
our scheme.

In the case where $u_0$ does not belong to
$H^{1,\alpha}(\R)$, we can approximate $u_0$ by
functions $u_{0,k}\in H^{1,\alpha}(\R)$, which
converge to $u_0$ in $H^1(\R)$. From \cite{horay07},
we know that the change of variable
\eqref{eq:initdatex} produces sequences $Y_{0,k}$
and $Y_0$ such that
$\lim_{k\to0}\norm{Y_{0,k}-Y_0}_F=0$. In this way,
by using the results done for functions in
$F^{\alpha}$, we can approximate the exact
solution $Y(t)$ and prove convergence. However,
since $\norm{Y_{0,k}}_{F^{\alpha}}$ is not
uniformly bounded with respect to $k$, we lose the
control on the error rate (the term
$\frac{1}{R^{\alpha/2}}$) which is given by
\eqref{eq:errorest}.

\section{Numerical experiments}\label{sect-numexp}

In this section, we present 
some numerical experiments for the
hyperelastic rod wave equation \eqref{eq:hr}. 
In order to demonstrate the efficiency of our 
schemes, we will numerically compute three 
types of traveling waves with decay, 
see Figure~\ref{fig:init}.
\begin{figure}[h]
\begin{center}
\includegraphics*[height=3cm,keepaspectratio]
{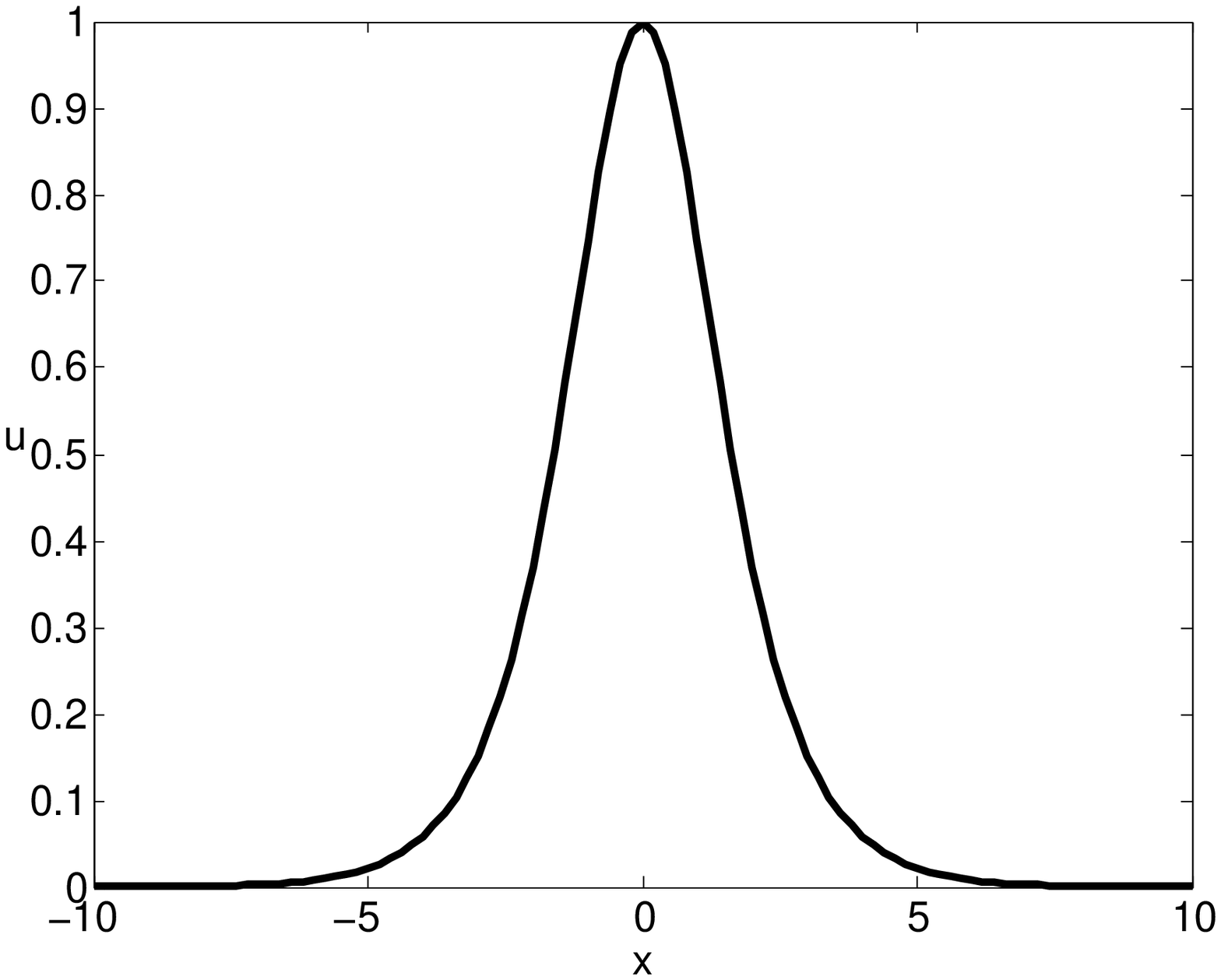}
\includegraphics*[height=3cm,keepaspectratio]
{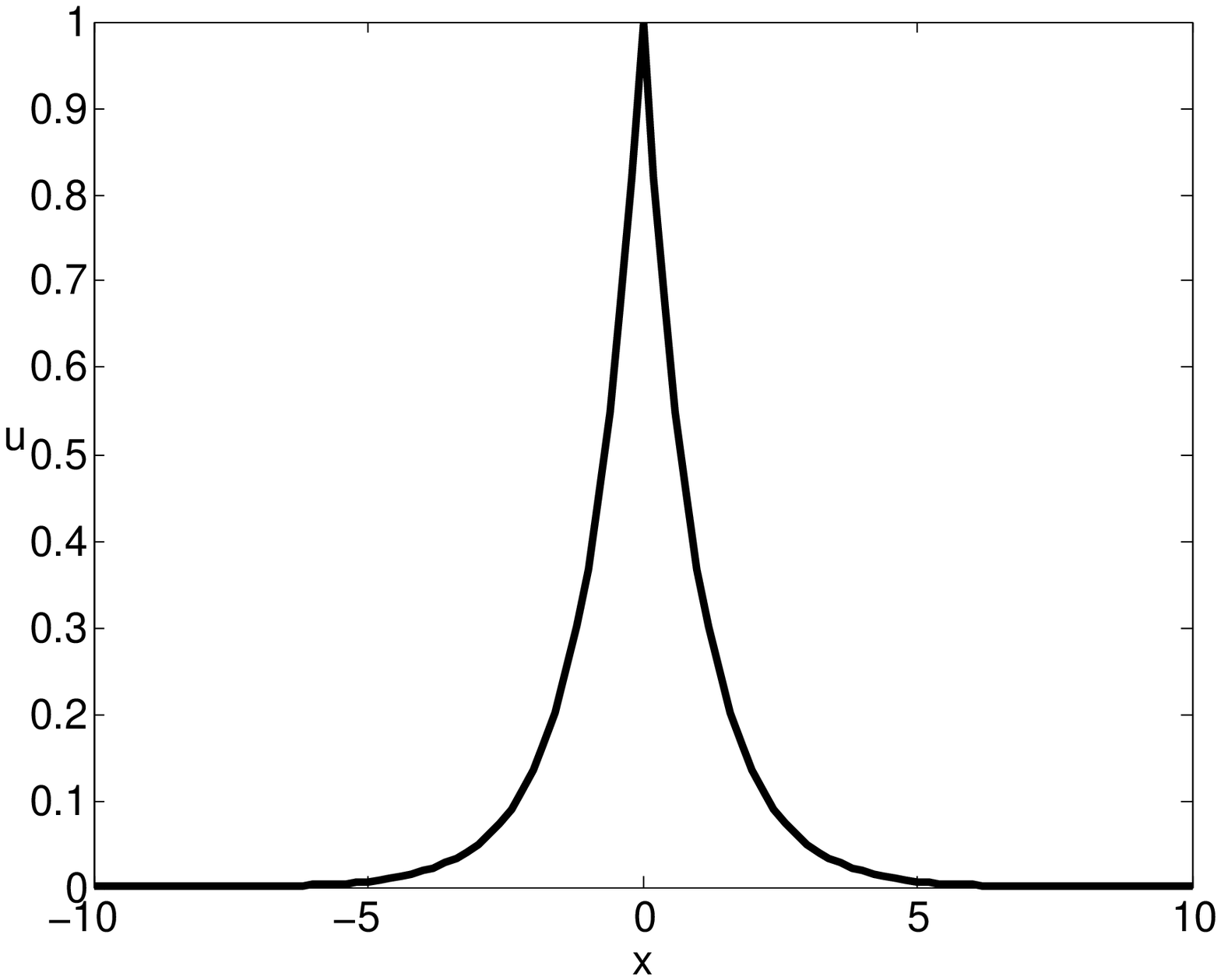}
\includegraphics*[height=3cm,keepaspectratio]
{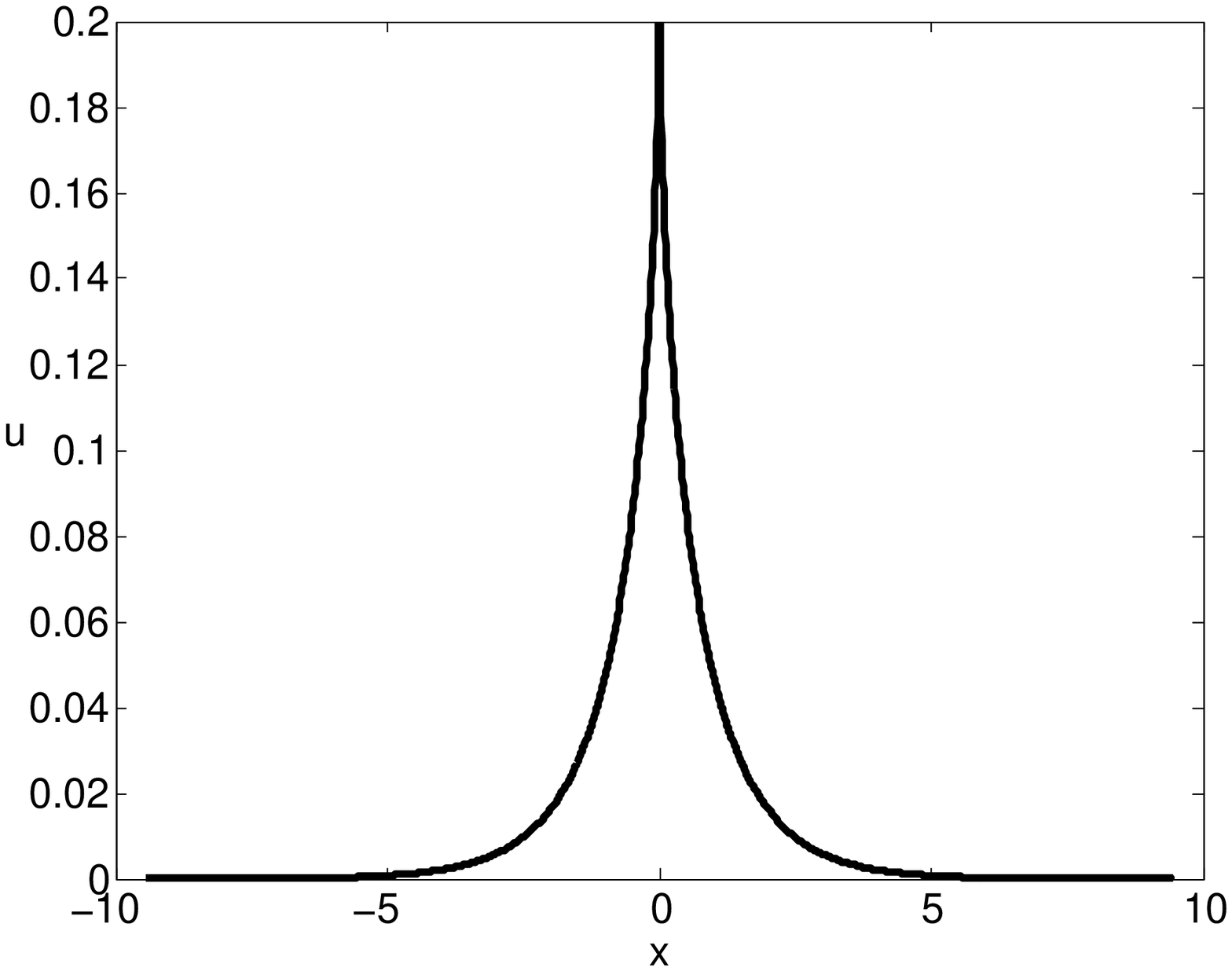}
\caption{Traveling waves with decay with speed $c=1$: 
smooth ($\gamma=0.2$), peakon ($\gamma=1$), cuspon ($\gamma=5$).}
\label{fig:init}
\end{center}
\end{figure}
The derivation of the cusped ($\gamma>1$),
resp. smooth ($\gamma<1$), solutions follows the
lines of \cite{lenells:06}.  We refer for example 
to \cite{ray:06} for a thorough discussion on the
peakon case (i.e. $\gamma=1$).

\subsection{Smooth traveling waves with decay ($\gamma<1$)}
\label{subsect-smooth}

According to the classification presented in
\cite{lenells:06}, for a fixed $\gamma\neq0$,
traveling waves $u(x-ct)$ are parametrised by
three parameters, $M$, $m$ and the speed
$c$. Moreover, they are solutions of the following
differential equation
\begin{equation}\label{eq:phitrav}
u_x^2=F(u)=
\sfrac{(M-u)(u-m)(u-z)}{c-\gamma u}.
\end{equation}
For positive values of $\gamma$, a smooth
traveling wave with decay with
$m=\inf_{x\in\R}u(x)$ and $M=\max_{x\in\R}u(x)$ is
obtained if $z=m<M<c/\gamma$, where $z:=c-M-m$.
For our purpose, we have to set $m=0$ so that the
solution decays at infinity. This gives us the
conditions $c=M$ and $\gamma<1$. We thereby
obtain the initial values for our system of
differential equations \eqref{eq:sysbarr} by
solving \eqref{eq:phitrav} numerically.  To do
this, some care has to be taken as
$u\mapsto\sqrt{F(u)}$ is not Lipschitz. We instead
solve $u_{xx}=F'(u)/2$. Once this is appropriately
done we get the initial values $U_0=u$,
$w_0=u_x$. We then set $y_0=\xi$, $q_0=1$,
$h_0=U_0^2+w_0^2$ and $H_0=\int_{-\infty}^{y_0}
h_0$. These initial values do not correspond to
the ones defined by \eqref{eq:initdatex} but they
are equivalent via relabeling and one can check
that
$(y_0,U_0,H_0,q_0,w_0,h_0)\in\G$. Figure~\ref{fig:smooth1}
displays the exact solution together with the
numerical solutions given by the ODE45 solver from
Matlab, the explicit Euler scheme, the Lie-Trotter
and the Strang splitting schemes at time $T=7$. 
We plot the points
\begin{equation*}
  (y(t,\xi_i),U(t,\xi_i)),\quad\text{ for }i=-N,\ldots,N-1,
\end{equation*}
which approximate the graph of the exact solution
$u(t,x)$ for $t=T$. The initial value is a
smooth traveling wave with parameters
$\gamma=0.2,m=0,M=c=1$, see Figure~\ref{fig:init}. 
We took relatively large
discretisation parameters $\Delta\xi=0.25$ and
$\Delta t=0.1$. We observe that the explicit Euler
scheme gives a less accurate solution than the
other schemes.  We also observe that, even for
these large discretisation parameters, the
splitting schemes have the same high as the exact
solution, thus following it at the same speed. We
do not observe any dissipation. Since both
splitting schemes give relative similar results,
in what follows, we will only display the results
given by the Strang splitting scheme.  We finally
note that all schemes preserve the positivity of
the particle density but only the splitting
schemes conserve exactly the invariants from
Section~\ref{sect-time} (these results are not
displayed).
\begin{figure}[h]
\begin{center}
\includegraphics*[height=4cm,keepaspectratio]
{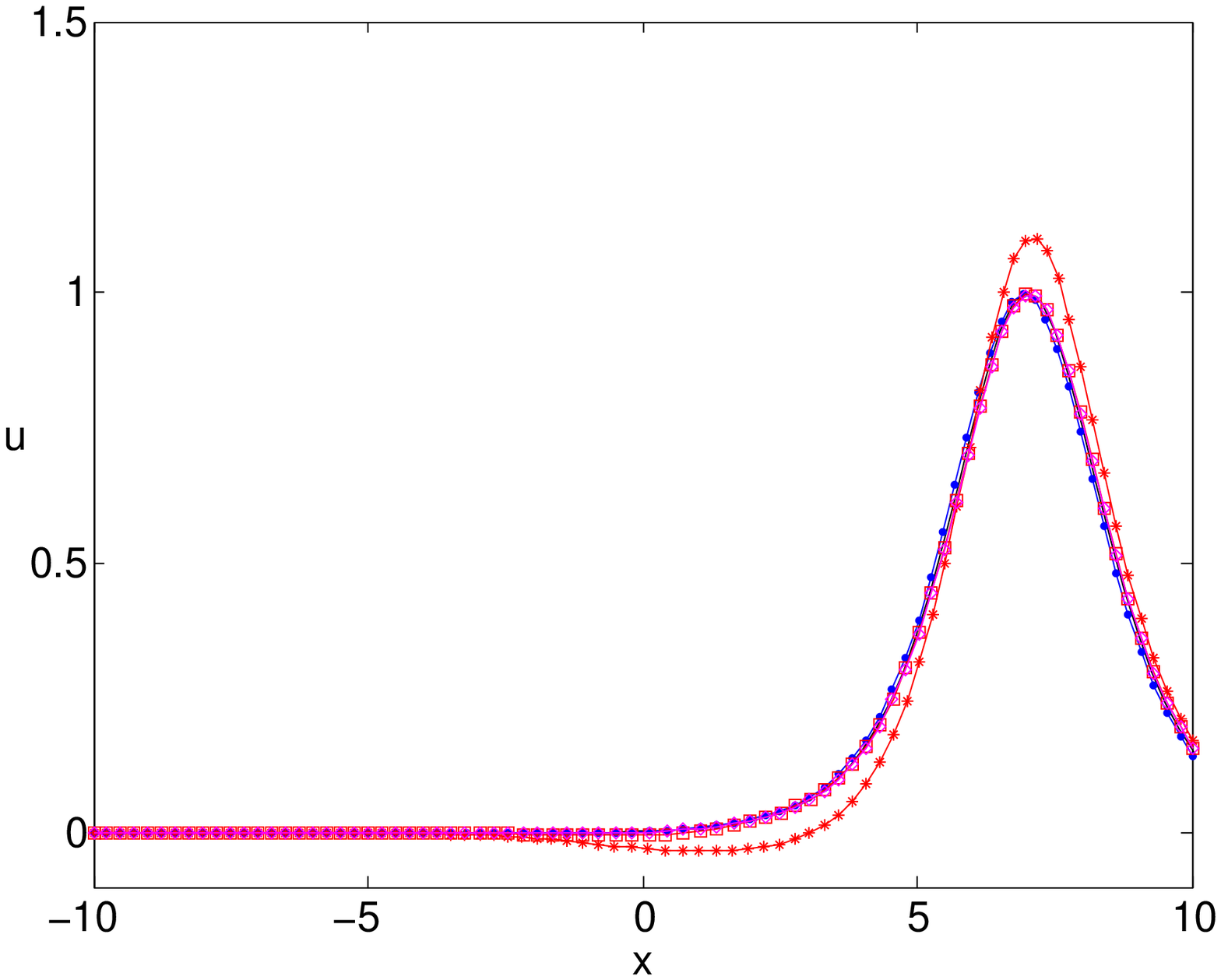}
\includegraphics*[height=4cm,keepaspectratio]
{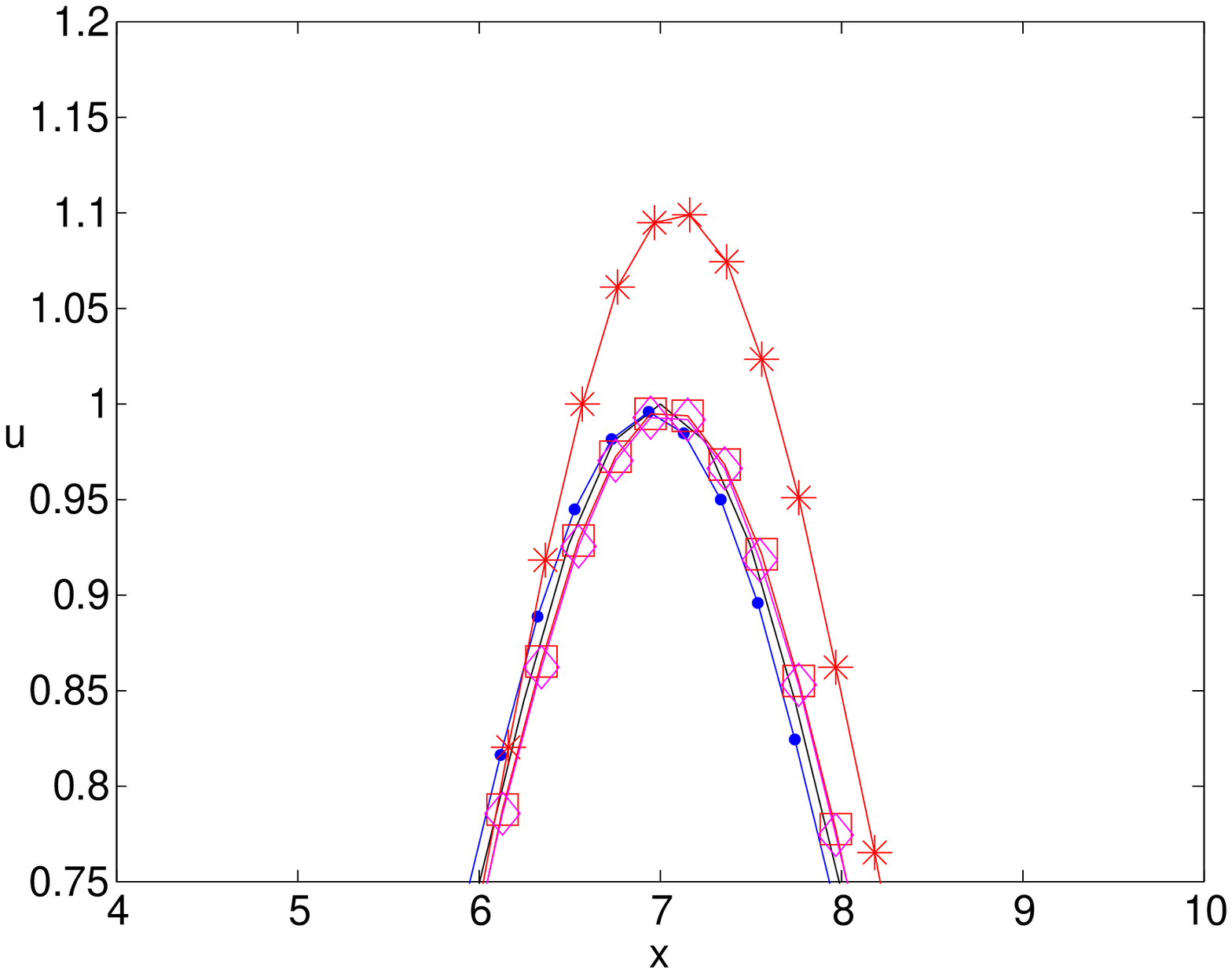}
\caption{Exact and numerical solutions of a smooth 
traveling wave with decay. Solid line=exact, dashdotted line=ODE45, 
Stars=Explicit Euler, Square=Lie-Trotter, Diamond=Strang.}
\label{fig:smooth1}
\end{center}
\end{figure}
We finally want to mention that 
for negative values of $\gamma$, 
smooth traveling waves with 
decay also exist. They are obtained 
if $c/\gamma<m=M<z$. 

\subsection{Peakon ($\gamma=1$)}
\label{subsect-peakon}

The Camassa--Holm equation, i.e. equation \eqref{eq:hr} with $\gamma=1$, 
possesses solutions with a particular shape: the peakons. 
A single peakon is a traveling wave which is given by
\begin{equation*}
u(t,x)=c\,\e^{-\abs{x-ct}}.
\end{equation*}
We note, that at the peak, the derivative of this
particular solution is discontinuous. The initial values are then
\begin{align*}
  &y_0(\xi)=\xi,\quad U_0(\xi)=u(0,\xi),\quad w_0(\xi)=u_x(0,\xi),\\
  &q_0=1,\quad h_0=U_0^2+w_0^2,\quad
  H_0=\int_{-\infty}^{\xi}h_0(\eta)\,d\eta
\end{align*}
In Figure~\ref{fig:peak}, we display the numerical
solutions given by the explicit Euler scheme and
the Strang splitting for a single peakon traveling
from left to right with speed $c=1$, see
Figure~\ref{fig:init}.  For readability reason, we
do not display the solution given by the ODE45
solver, but we note that this numerical solution
is very similar to the one given by the splitting
scheme.  Due to the discontinuity of the
derivative, we have to take smaller (in space)
discretisation parameters: $\Delta\xi=0.05$ and
$\Delta t=0.2$.  We note more grid-points before
the peak and very few just after it, but the speed
of the wave is still relatively close to the exact
one.
\begin{figure}[h]
\begin{center}
\includegraphics*[height=8cm,keepaspectratio]
{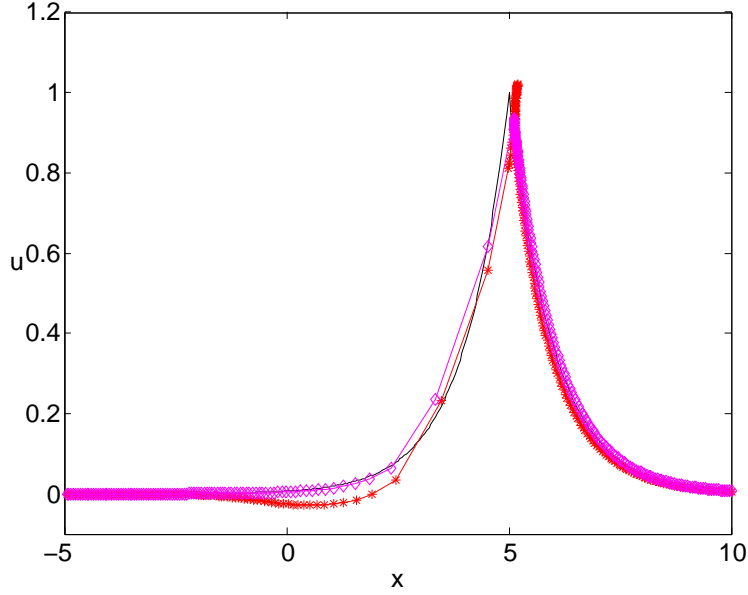}
\caption{Exact and numerical solutions at time $T=5$. 
Solid line=exact, Stars=Explicit Euler, Diamond=Strang.}
\label{fig:peak}
\end{center}
\end{figure}
As in the preceding case, only the splitting schemes 
preserve exactly the invariants of our problem.

\subsection{Cusped traveling waves with decay ($\gamma>1$)}
\label{subsect-cusped}

Let us now turn our attention to cusped traveling
waves. For $\gamma>0$, according to the
classification given in \cite{lenells:06}, cusped
solutions with $c/\gamma=\max_{x\in\R}u(x)$ and
$m=\inf_{x\in\R}u(x)$ are obtained if
$z=m=0<c/\gamma<M$.  This gives us the condition
$c=M$ and thus $\gamma>1$. The cuspon $u(x)$
satisfies \eqref{eq:phitrav}, which yields for
the indicated values of the parameters
\begin{equation}
  \label{eq:ueqcuspon}
  u_x=-\sqrt{F(u)}=-\left(\frac{M-u}{c-\gamma u}\right)^{\frac12}u
\end{equation}
for $x\geq 0$ and with the boundary value at zero
given by $u(0)=\frac{c}{\gamma}$. For such
boundary value, the differential equation
\eqref{eq:ueqcuspon} is not well-posed and the
slope at the top of the cuspon (that is $x=0$) is
indeed equal to infinity. However, we can find a
triplet $X=(y,U,H)$ in $\F$ which corresponds to
this curve, that is, such that
$(u,u^2+u_x^2\,dx)=M(X)$, see \eqref{eq:umudef}
for the definition of the map $M$. Due to the
freedom or relabeling, the representation of the
curve $(x,u(x))$ is not unique: For any
diffeomorphism $(\phi(\xi),u(\phi(\xi)))$, we
obtain an other parameterization of the same
curve. Here, we look for a smooth $\phi(\xi)$ (and
we set $y(\xi)=\phi(\xi)$) such that
$U=u(\phi(\xi))=u(y(\xi))$ is smooth, even if $u$
is not. We introduce the function
\begin{equation*}
  g(u)=-\int_{\frac{c}{\gamma}}^{u}\frac{dz}{\sqrt{F(z)}}.
\end{equation*}
Since $\frac{dx}{du}=-\frac1{\sqrt{F(u)}}$, by
\eqref{eq:ueqcuspon}, if we choose
\begin{equation*}
  U(\xi)=\frac{c}\gamma-\xi,\quad y(\xi)=g(U(\xi))
\end{equation*}
then we get, at least for
$\xi\in[0,\frac{c}\gamma]$, a triplet for which
$U(\xi)=u(y(\xi))$. We set the energy density by
using \eqref{eq:lagcoord3} and get
\begin{equation*}
  H_\xi=U^2y_\xi+\frac{U_\xi^2}{y_\xi}.
\end{equation*}
However, in this case, 
\begin{equation*}
  y_\xi=g'(U)U_\xi=\left(\frac{c-\gamma U}{M-U}\right)^{\frac12}\frac1{U}
\end{equation*}
so that $H_\xi(0)=\infty$ and it is incompatible
with the requirement that all the derivatives in
Lagrangian coordinates are bounded in $L^\infty(\R)$,
see \eqref{eq:lagcoord1}. Thus, we take
\begin{equation*}
  U(\xi)=\frac{c}\gamma-\xi^2,\quad y(\xi)=g(U(\xi)),\quad
  H_\xi=U^2y_\xi+\frac{U_\xi^2}{y_\xi}.
\end{equation*}
In this case, we have
\begin{equation*}
  y_\xi(\xi)=g'(U)U_\xi=\frac2{U(\xi)}\left(\frac{c-\gamma U(\xi)}{M-U(\xi)}\right)^{\frac12}\xi=\frac{2\sqrt{\gamma}}{U(\xi)(M-U(\xi))^{\frac12}}\xi^2
\end{equation*}
and 
\begin{equation*}
  H_\xi(0)=\frac{2c}{\gamma^2}(M\gamma-c)^{\frac12}
\end{equation*}
is finite. The problem we face now is that the
functions are given only on the interval
$[0,\frac{c}\gamma)$ and
$\lim_{\xi\to\frac{c}{\gamma}}y(\xi)=\infty$. We
know that the tail of the cuspon behaves as
$u(x)\approx\frac{c}\gamma \e^{-\sqrt{\frac{M}c}x}$
as $x$ tends to $\infty$, see
\cite{lenells:06}. Since we require that
$y(\xi)-\xi$ remains bounded, we would like to
have $U(\xi)\approx\frac{c}\gamma
\e^{-\sqrt{\frac{M}c}\xi}$ for large
$\xi$. Therefore we introduce the following
partitions functions $\chi_1$ and $\chi_2$ defined
as
\begin{equation*}
  \chi_1(\xi)=
  \begin{cases}
    1&\text{ if }\xi<a\\
    -\frac{1}{b-a}(\xi-b)&\text{ for }\xi\in[a,b]\\
    0&\text{ if }x>b
  \end{cases}
\end{equation*}
and $\chi_2(\xi)=1-\chi_1$, where $a<b$ are two
parameters. We finally set
\begin{equation*}
  U(\xi)=\chi_1(\xi)(\frac{c}\gamma-\xi^2)+\chi_2(\xi)\frac{c}\gamma \e^{-\sqrt{\frac{M}c}\xi}
\end{equation*}
and
\begin{equation*}
  y(\xi)=g(U(\xi)),\quad H_\xi=U^2y_\xi+\frac{U_\xi^2}{y_\xi}.
\end{equation*}
By a proper choice of the parameters $a$ and $b$,
we can guarantee that $y_\xi(\xi)\geq0$ for all
$\xi\geq0$. We extend
$X(\xi)=(y(\xi),U(\xi),H(\xi))$ on the whole axis
by parity and we obtain an element in $\F$ such
that \eqref{eq:umudef} is
satisfied. Figure~\ref{fig:cuspedInityU} displays 
$y(\xi)$ and $U(\xi)$.
\begin{figure}[h]
\begin{center}
\includegraphics*[height=4cm,keepaspectratio]
{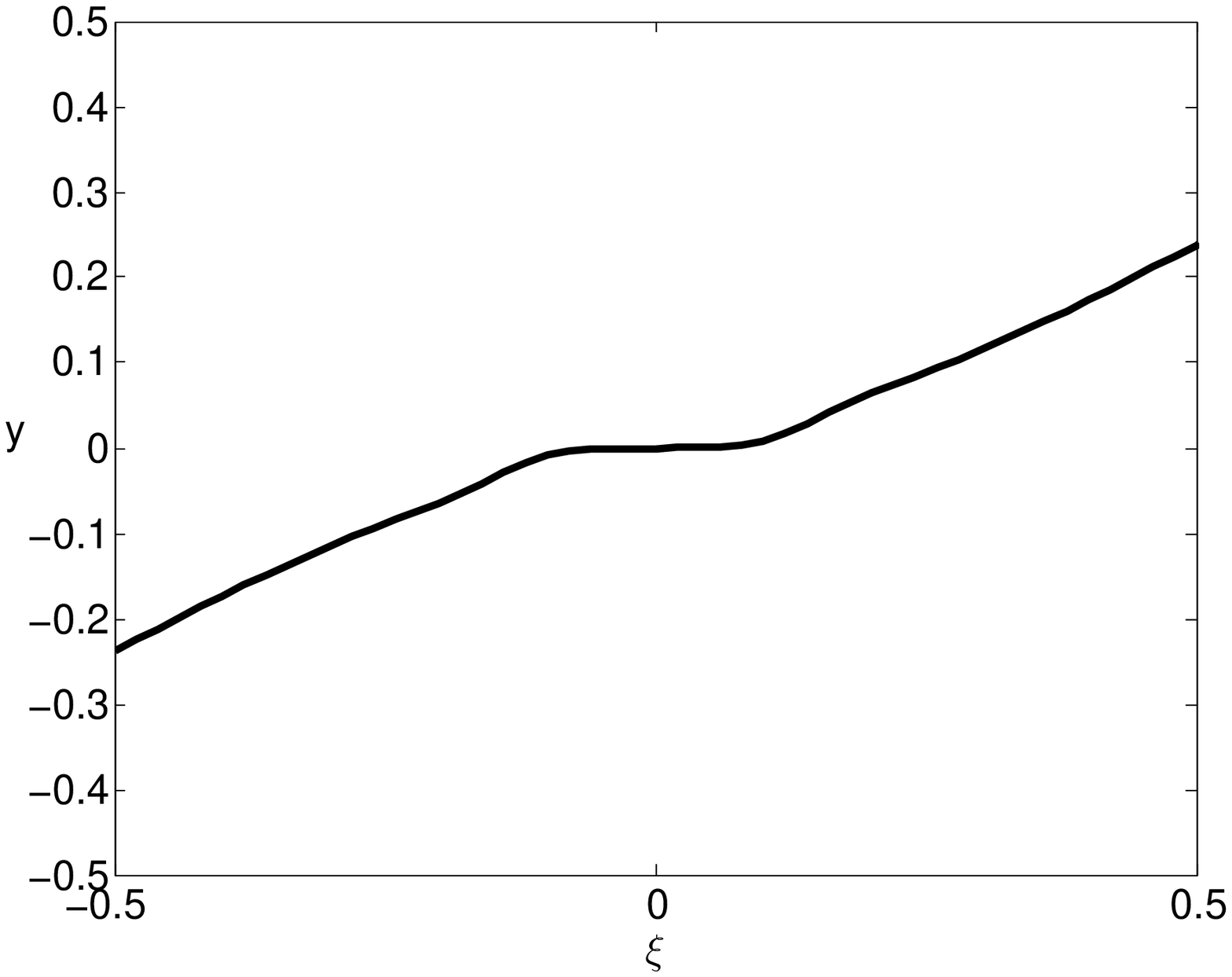}
\includegraphics*[height=4cm,keepaspectratio]
{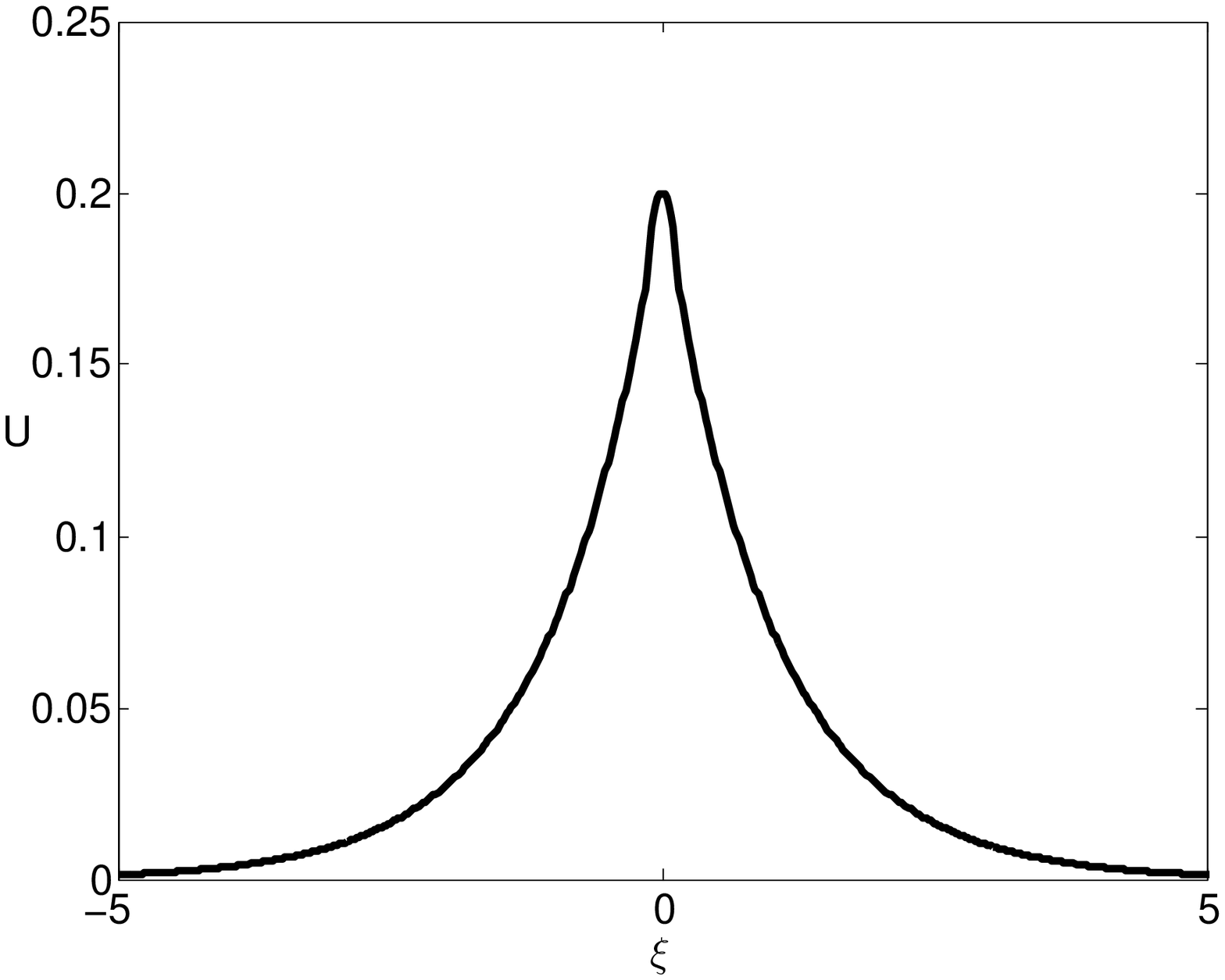}
\caption{The function $y(\xi)$ (left picture) and
  the function $U(\xi)$. Note that these functions
  are smooth while $u_0(x)$ is not Lipschitz, see
  Figure \ref{fig:init}.}
\label{fig:cuspedInityU}
\end{center}
\end{figure}
Figure~\ref{fig:cusped} displays the exact
solution together with the numerical solutions
given by the explicit Euler scheme and the Strang
splitting scheme at time $T=6$. As before, we
note that the numerical solution given by the
ODE45 solver is very similar to the one given by
our splitting scheme.  The initial value is a
cusped traveling wave with parameters
$\gamma=5,m=0,M=c=1$, see Figure~\ref{fig:init}.
For the discretisation parameters, we take
$\Delta\xi=0.1$ and $\Delta t=0.1$.  We see that,
even for initial data with infinite derivative
$u_x(0)=\pm\infty$, the spatial discretisation
converges. For the time discretisation, as
expected, explicit Euler is less accurate than the
other schemes.  We also remark that only the
splitting schemes preserve the positivity of the
particle density and conserve the invariants.
\begin{figure}[h]
\begin{center}
\includegraphics*[height=8cm,keepaspectratio]
{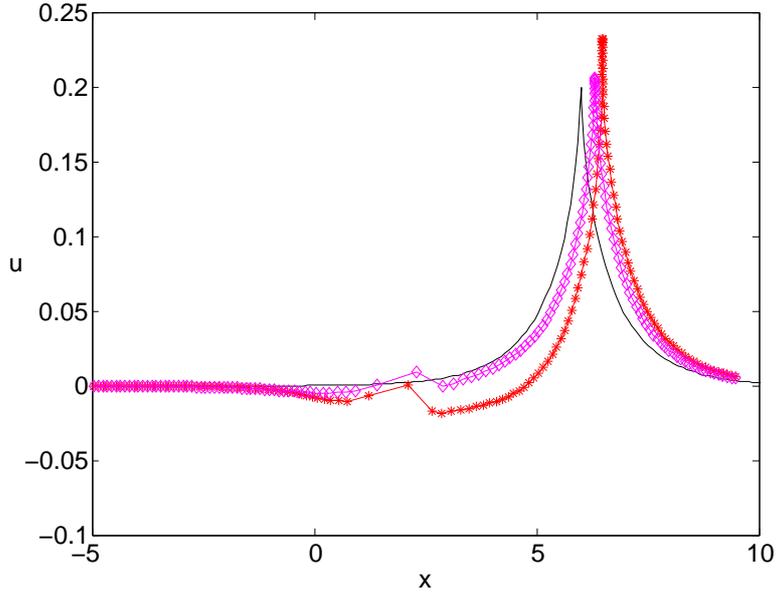}
\caption{Exact and numerical solutions of a cusped  
traveling wave with decay. Solid line=exact, Stars=Explicit Euler, 
Diamond=Strang.}
\label{fig:cusped}
\end{center}
\end{figure}
We finally note that, for negative values of $\gamma$, 
an anticusped traveling wave with $c/\gamma=\min_{x\in\R}u(x)$ 
and $m=\sup_{x\in\R}u(x)$ is obtained if $c/\gamma<m=M<z$.

\subsection{Peakon-antipeakon collisions}
\label{subsect-peakcoll}

In Figure~\ref{fig:collpeak1} we display a collision between a peakon 
and an antipeakon for $\gamma=1$. For this problem, 
the initial value is given by
$$
u(0,x)=\e^{-\abs{x}}-\e^{-\abs{x-1}}.
$$
The numerical solutions are computed with 
grid parameters $\Delta\xi=0.1$ and $\Delta t=0.1$ 
until time $T=8$. Once again we notice that 
the spatial discretisation converges.
\begin{figure}[h]
\begin{center}
\includegraphics*[height=8cm,keepaspectratio]
{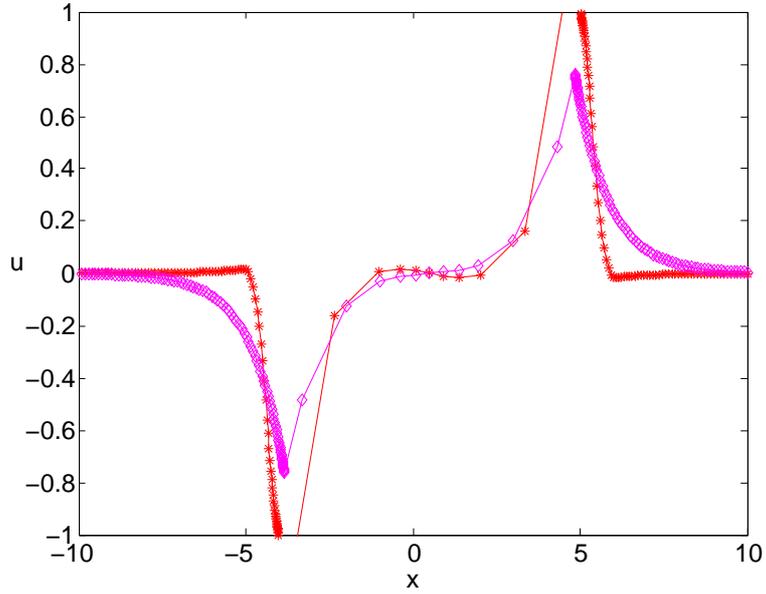}
\caption{Peakon-antipeakon collision for $\gamma=1$. 
Stars=Explicit Euler, Diamond=Strang.}
\label{fig:collpeak1}
\end{center}
\end{figure}
Let us now see what happens for a peakon-antipeakon collision 
with $\gamma\neq1$. In Figure~\ref{fig:collpeak5} we present a similar
experiment as the above one, but where we use $\gamma=5$ and $T=2$. 
Here, we plot the graph given by the points
\begin{equation*}
  (y(t,\xi_i),\frac{h}{q}(t,\xi_i)),\quad\text{ for }i=-N,\ldots,N-1
\end{equation*}
for $t=T$. From the right part of Figure~\ref{fig:collpeak5}
we see that only the splitting schemes preserve
the positivity of the energy density. 
As always, only the splitting
schemes conserve exactly the invariants.

\begin{figure}[h]
\begin{center}
\includegraphics*[height=4cm,keepaspectratio]
{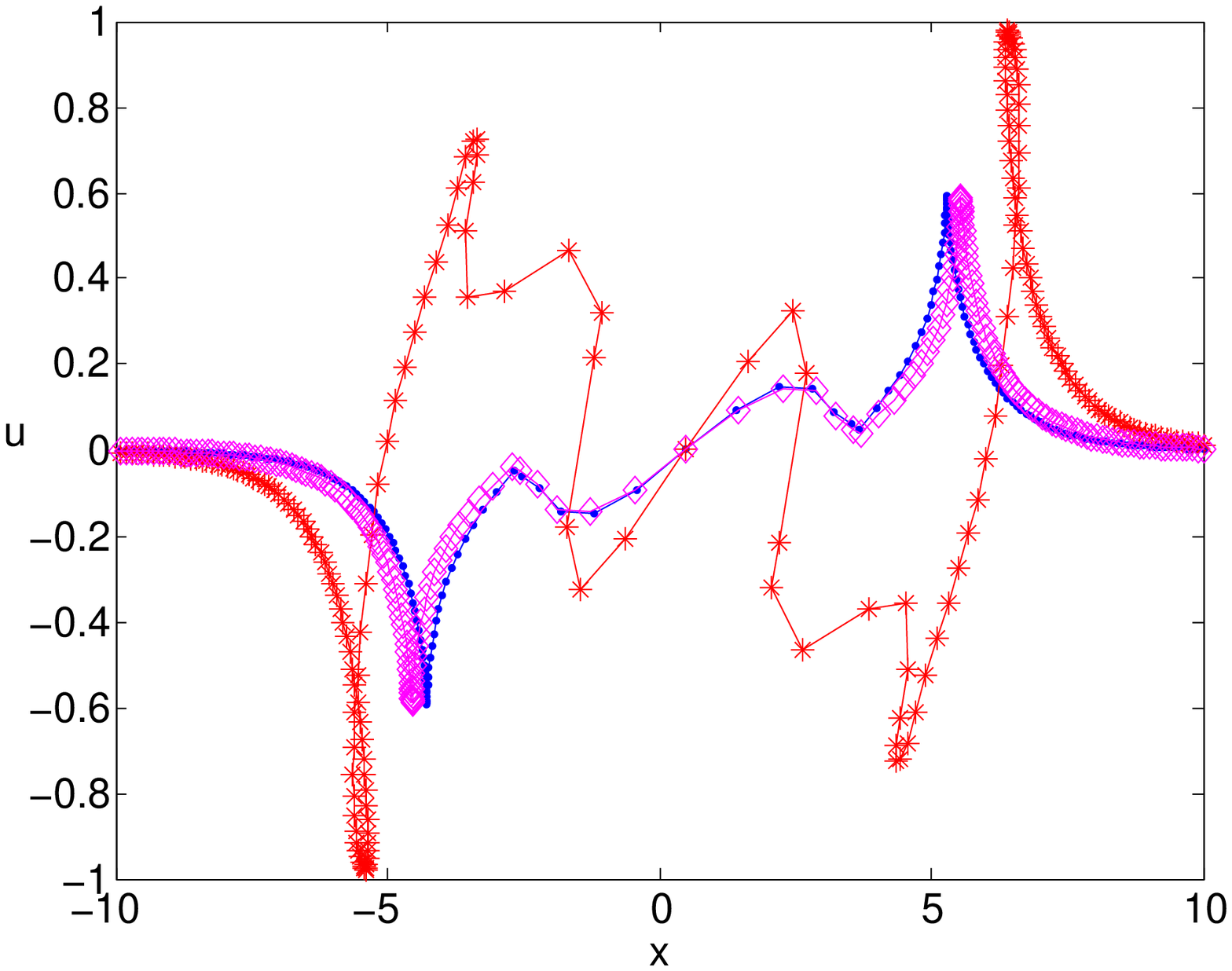}
\includegraphics*[height=4cm,keepaspectratio]
{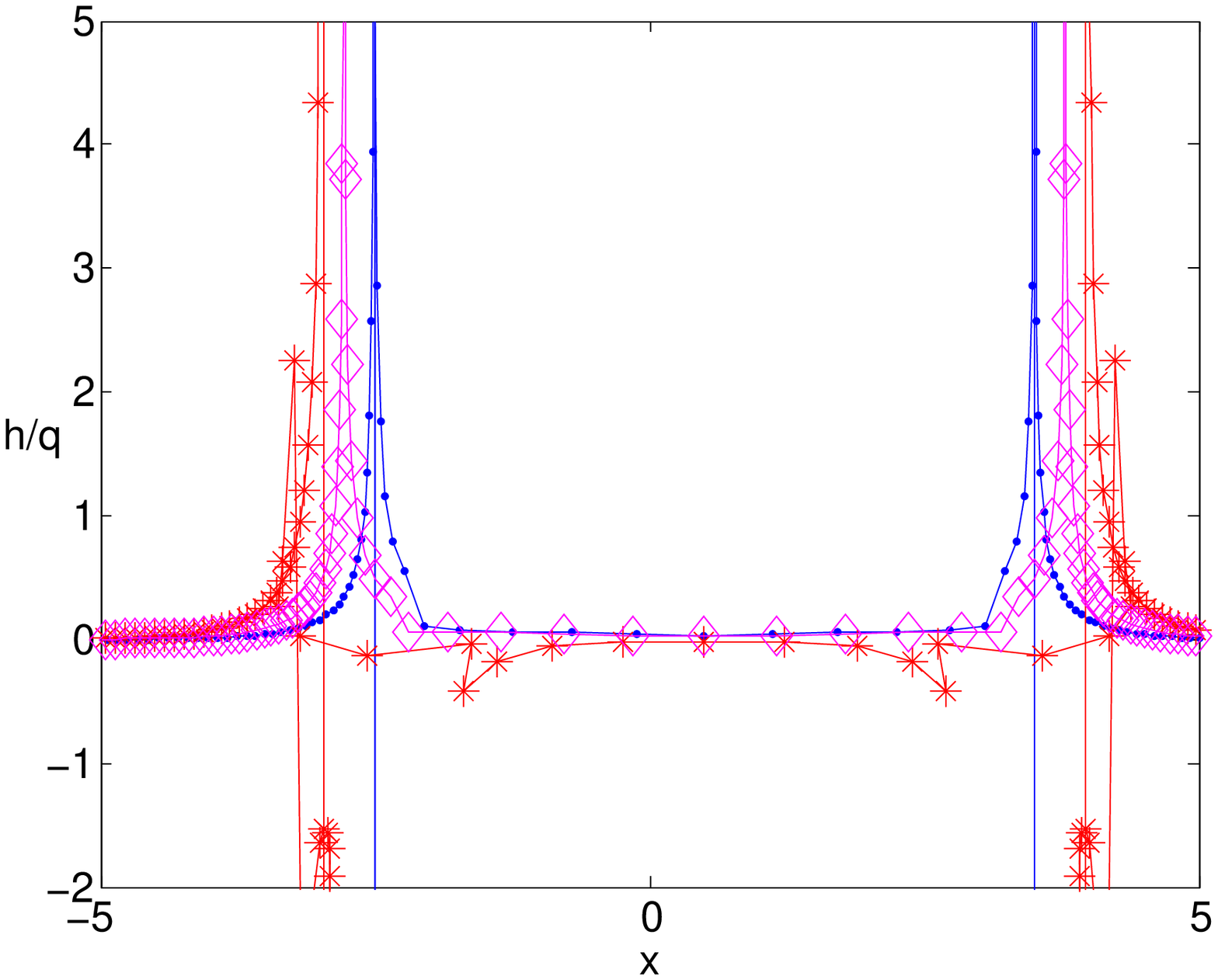}
\caption{Peakon-antipeakon collision for $\gamma=5$ 
at time $T=2$ (left) and 
energy density (right) at the first time, 
where the numerical solution given 
by ODE45 is not positive ($q=-1.7394e-05$). 
Dashdotted line=ODE45, Stars=Explicit Euler, 
Diamond=Strang.}
\label{fig:collpeak5}
\end{center}
\end{figure}

\subsection{Collision of smooth traveling waves}
\label{subsect-collision}

We want now to study the behaviour of the
numerical schemes when dealing with a collision of
smooth traveling waves, as this in an important
feature of our numerical scheme to be able to
handle such configuration. To do so, we consider
the following initial value
$$
u(0,x)=-x\e^{-x^2/2}.
$$
Figure~\ref{fig:col} displays 
the exact solution (i.e. the numerical 
solution with very small 
discretisation parameters) for $\gamma=0.8$. 
It is remarkable to see that even for 
such solution, our scheme performs very well. 
\begin{figure}[h]
\begin{center}
\includegraphics*[height=4cm,keepaspectratio]
{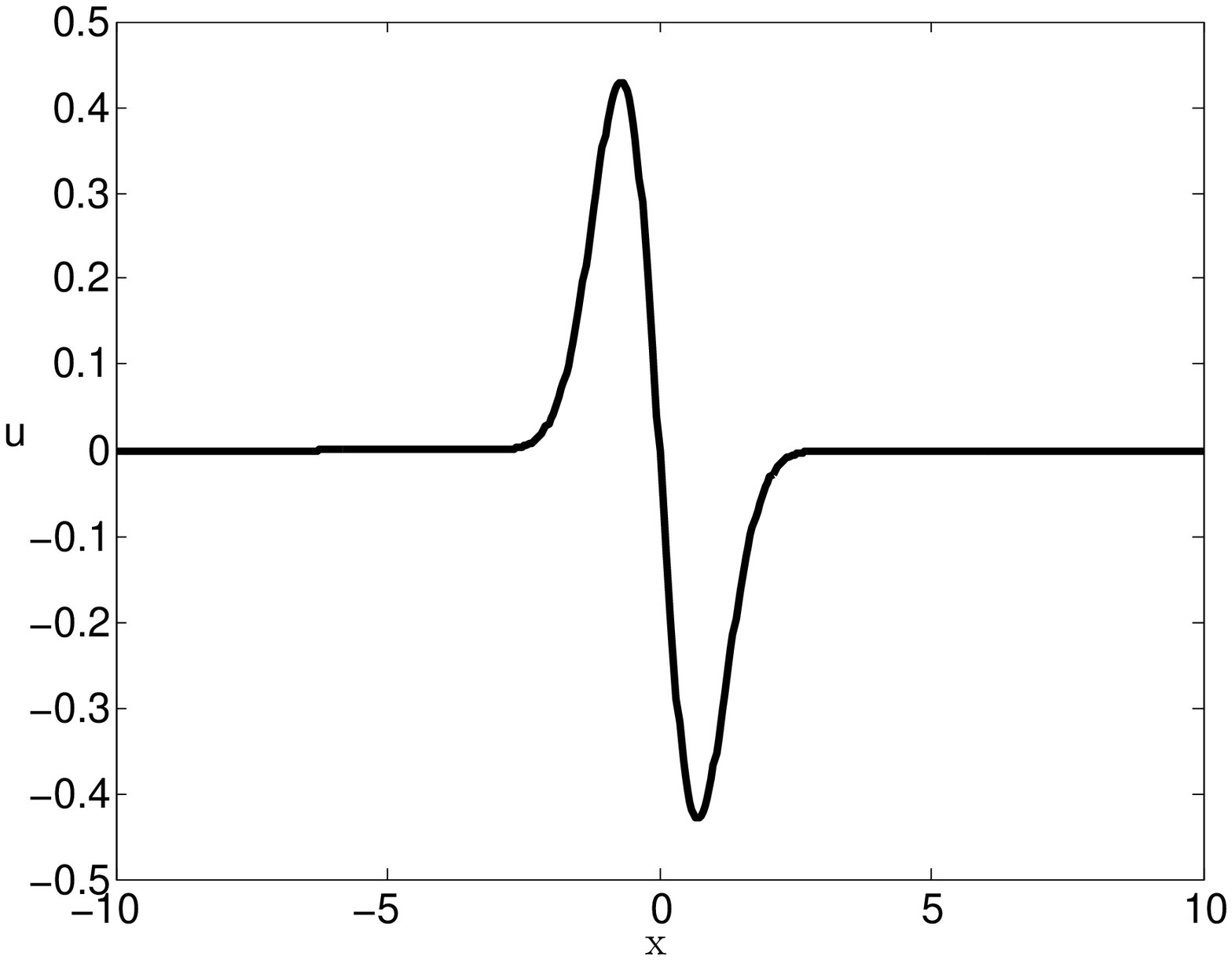}
\includegraphics*[height=4cm,keepaspectratio]
{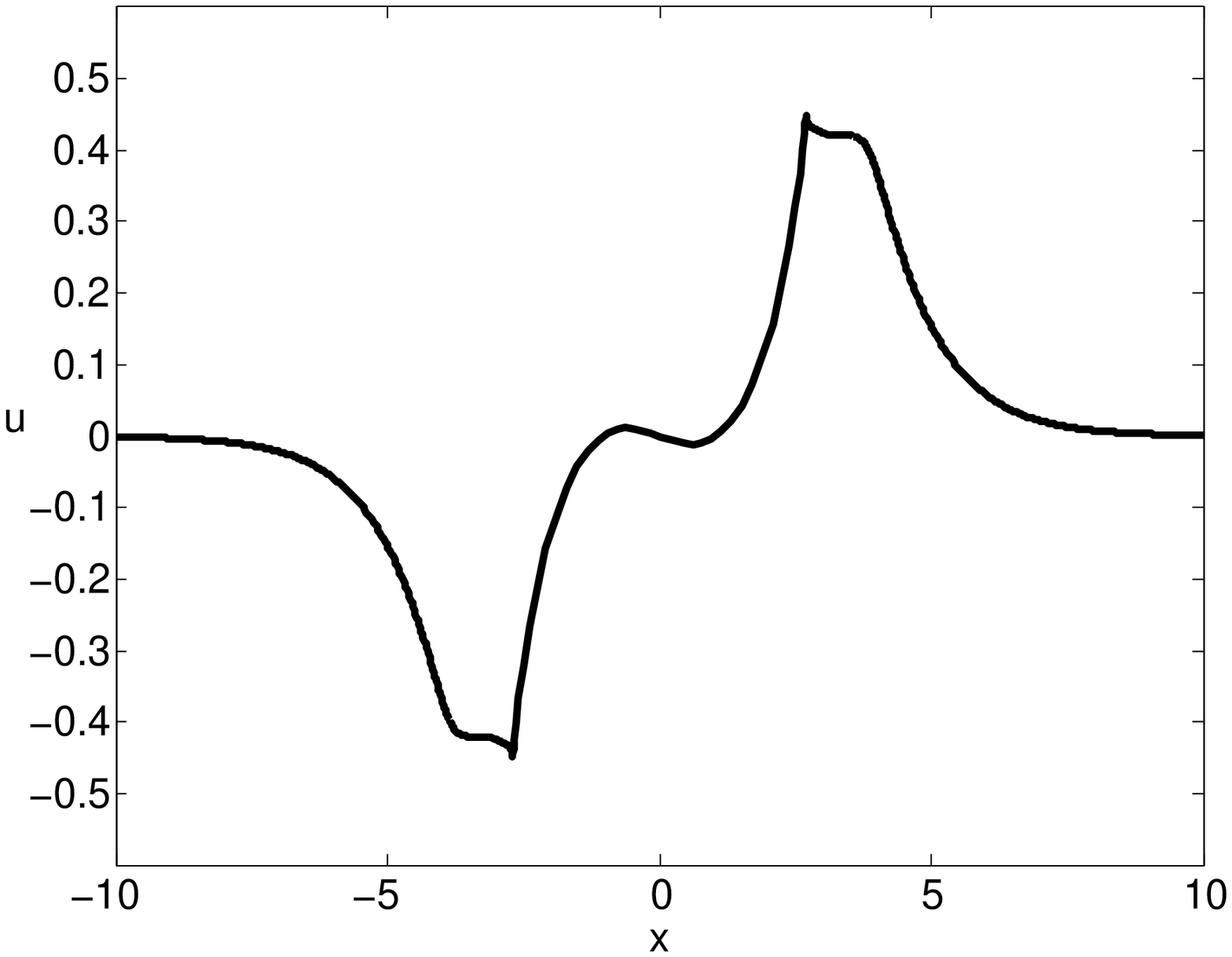}
\caption{Collision of smooth traveling waves: 
Initial data (left) and exact solution at 
time $T=11$.}
\label{fig:col}
\end{center}
\end{figure}
In order to get a better understanding of this problem, 
we look at the evolution of the waves with time. 
Figure~\ref{fig:smoothwaterfall} shows 
this evolution together  
with a zoom close to the collision time.
\begin{figure}[h]
\begin{center}
\includegraphics*[height=4cm,keepaspectratio]
{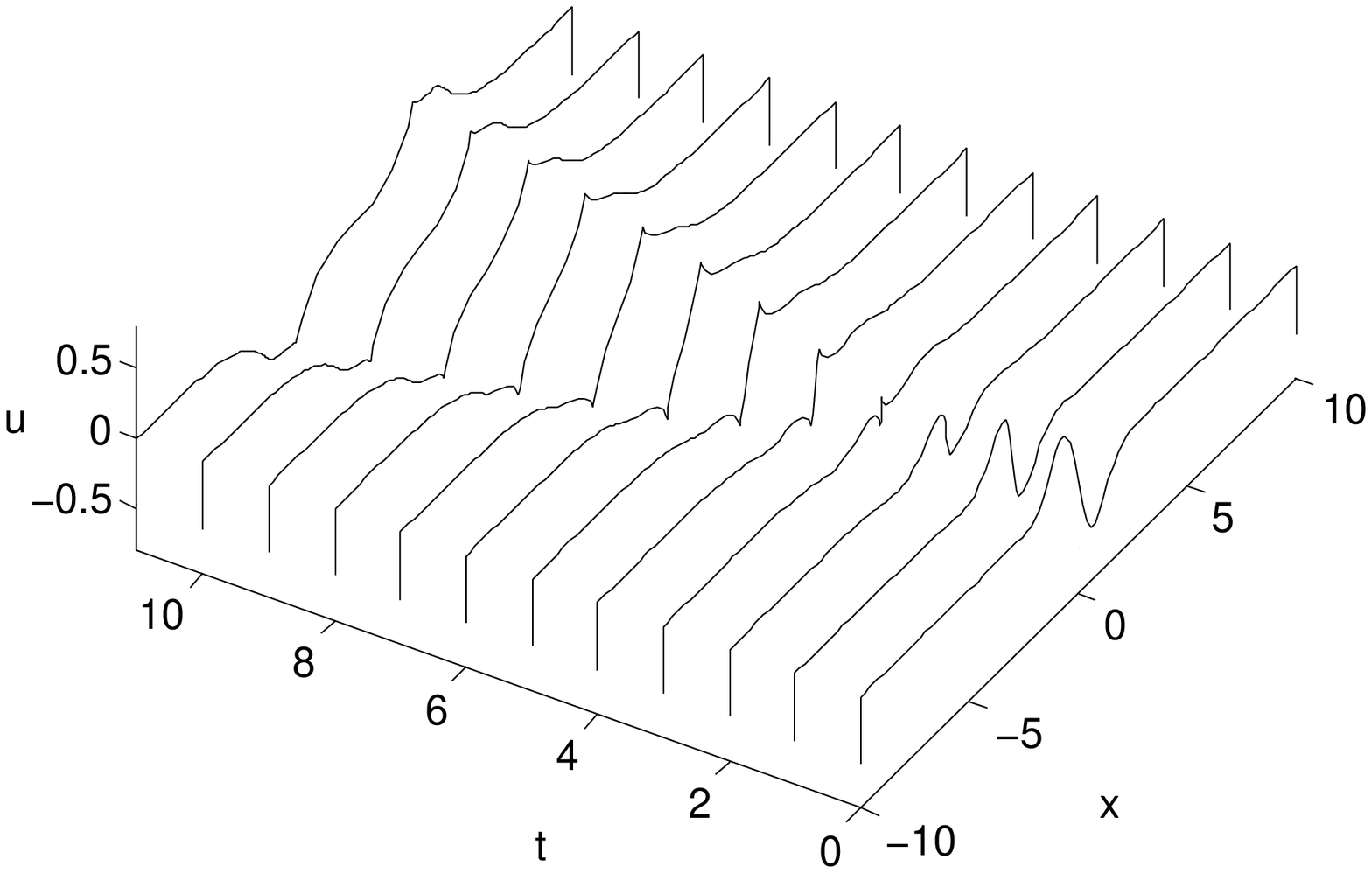}
\includegraphics*[height=4cm,keepaspectratio]
{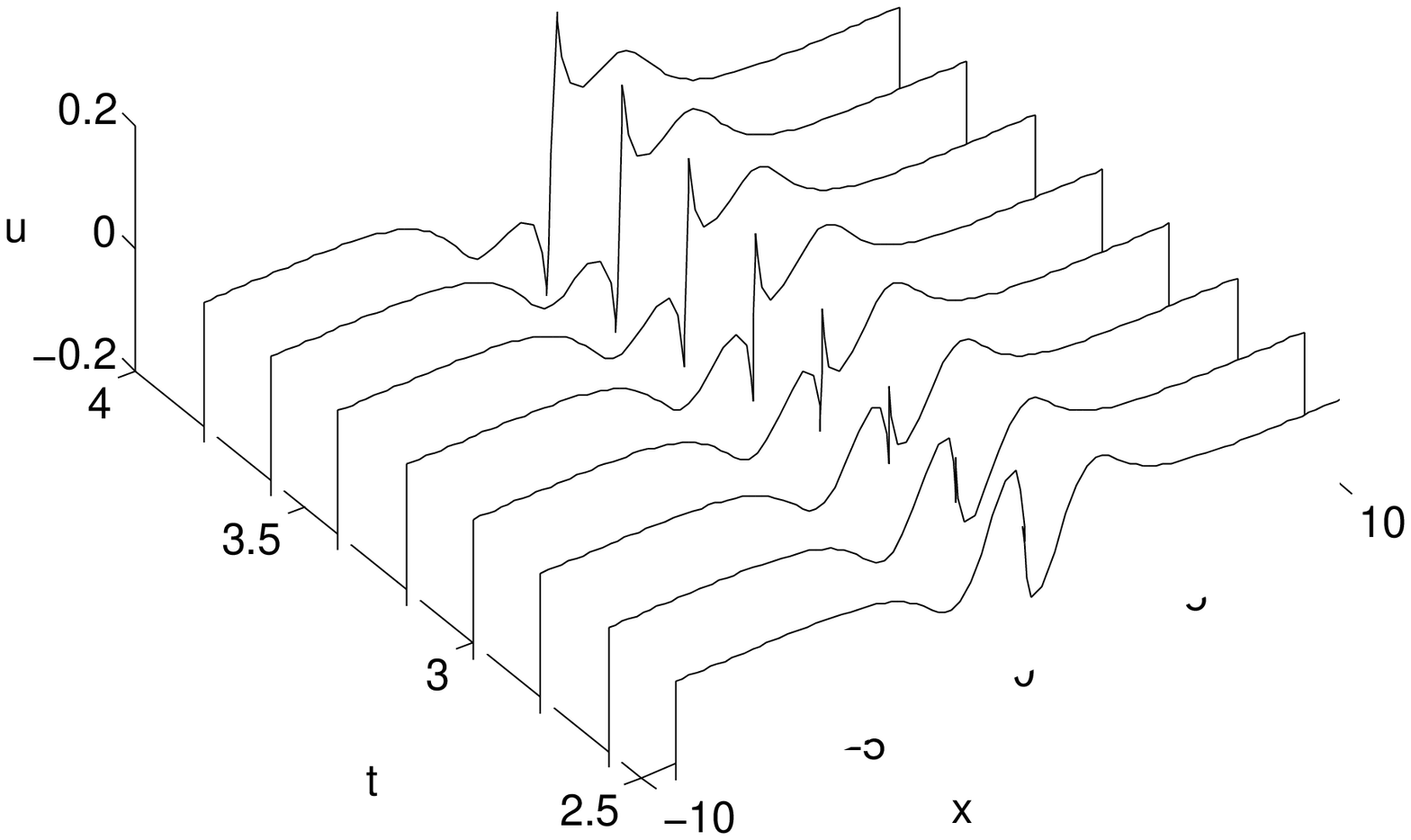}
\caption{Collision of smooth traveling waves: 
Evolution in time (left) and zoom of the 
evolution close to the collision.}
\label{fig:smoothwaterfall}
\end{center}
\end{figure}
We now present the results given by the numerical
schemes with grid parameters $\Delta\xi=0.25$ and
$\Delta t=0.1$ in Figure~\ref{fig:colnum}. 
\begin{figure}[h]
\begin{center}
\includegraphics*[height=8cm,keepaspectratio]
{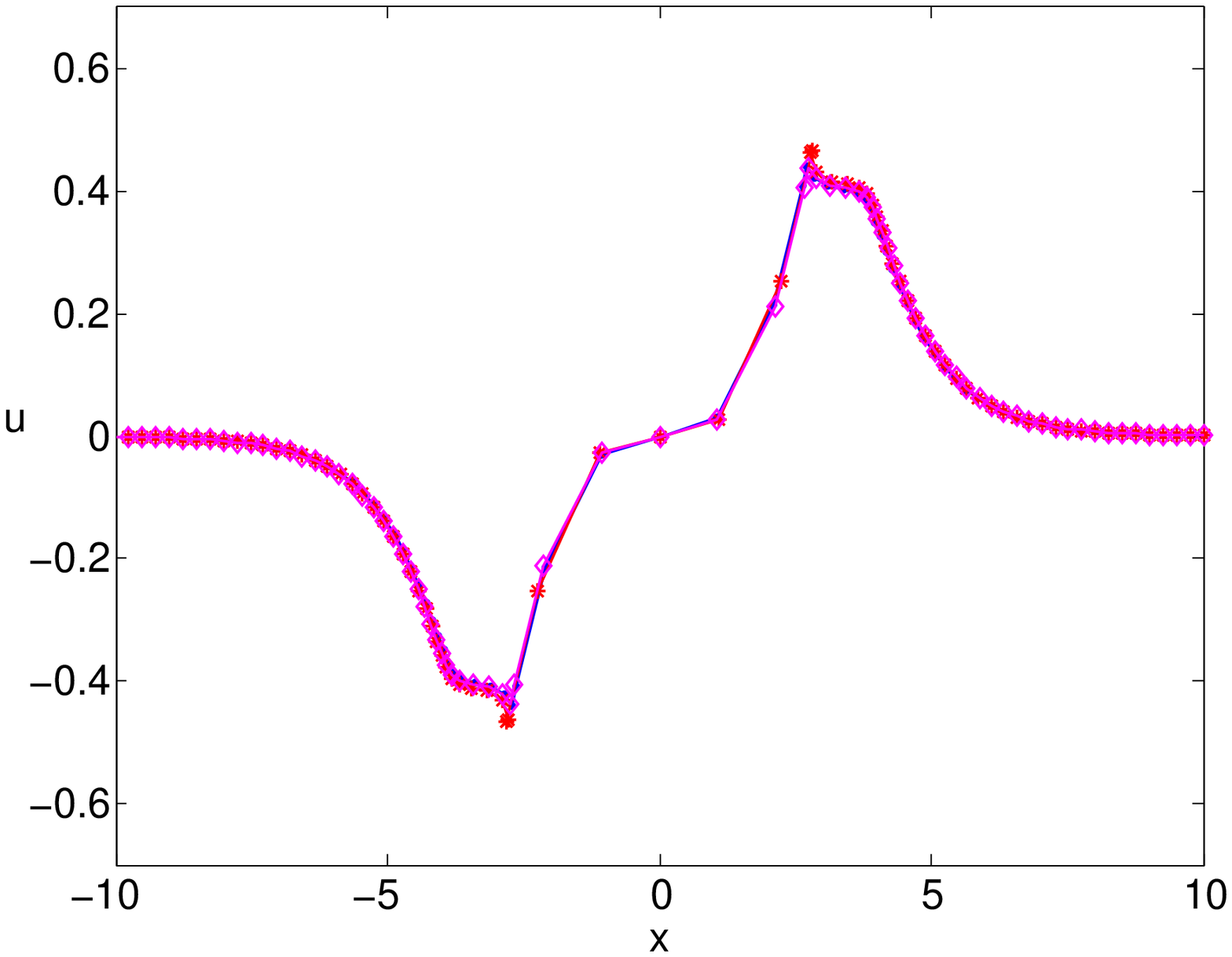}
\caption{Collision of smooth traveling waves: 
numerical solutions at 
time $T=11$.  
Dashdotted line=ODE45, Stars=Explicit Euler, 
Diamond=Strang.}
\label{fig:colnum}
\end{center}
\end{figure}
We have also checked that only the splitting
schemes preserve the positivity of the particle
density and conserve the invariants of our
problem. Finally, in Figure~\ref{fig:colposiwater}
we display, with the same parameter values as
above, the evolution in time of the energy density
along the numerical solution given by the Strang
splitting scheme.  We can observe the
concentration of the energy and then its
separation in two parts, following the waves.
\begin{figure}[h]
\begin{center}
\includegraphics*[height=4cm,keepaspectratio]
{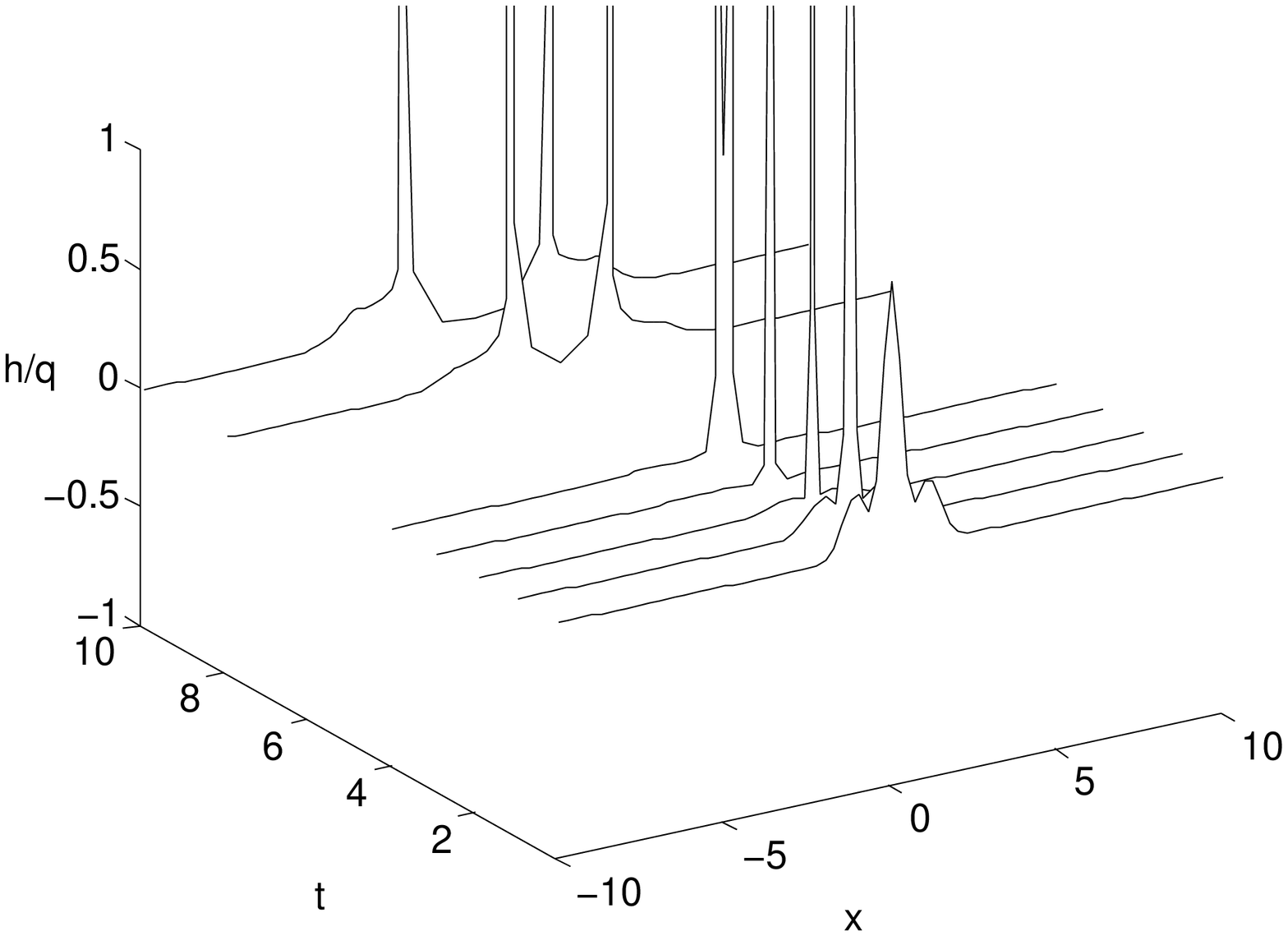}
\includegraphics*[height=4cm,keepaspectratio]
{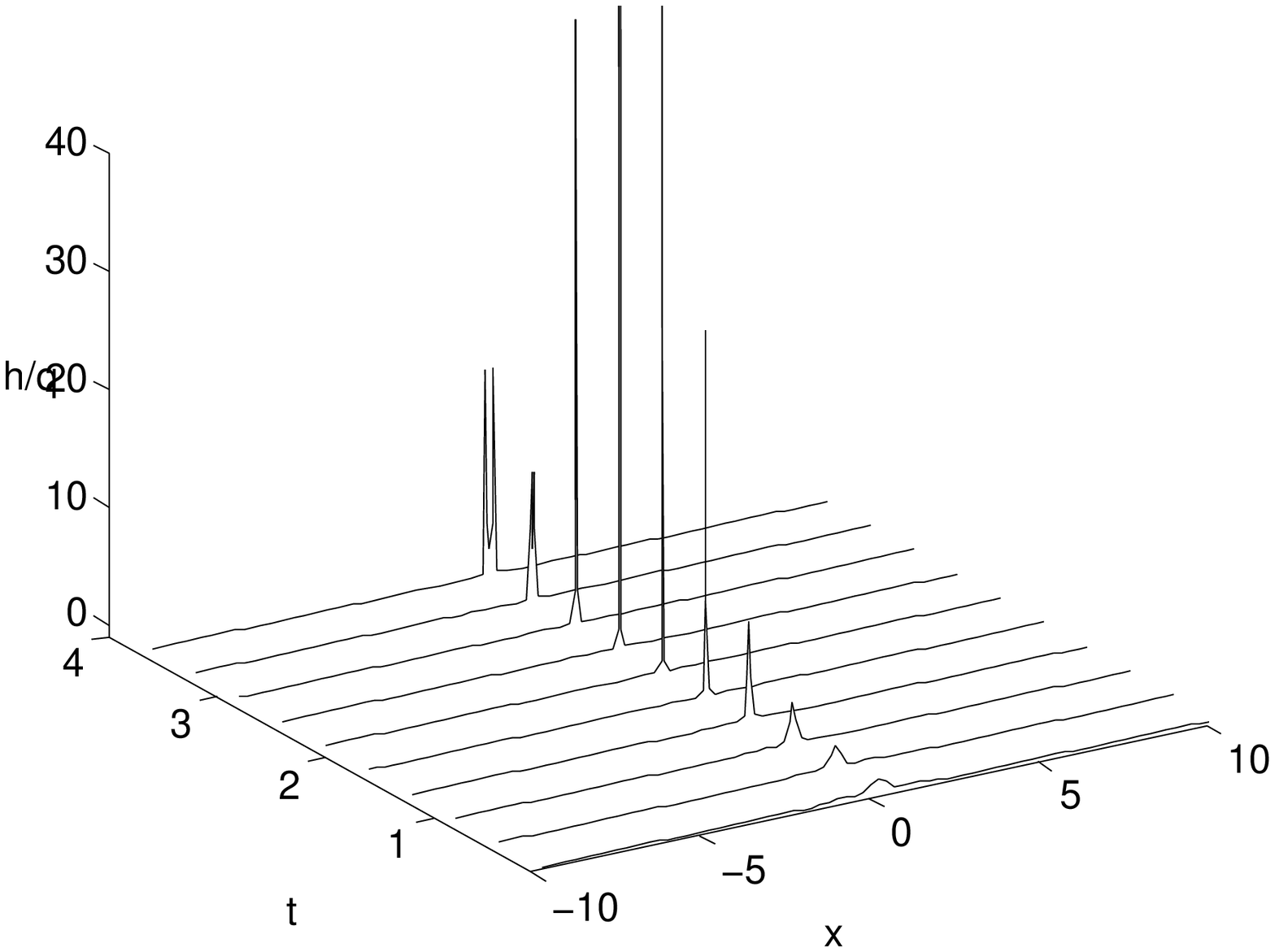}
\caption{Evolution of the energy density (left picture) 
along the numerical solution given by the Strang splitting and 
close up look at the blow up time (right).}
\label{fig:colposiwater}
\end{center}
\end{figure}
With all these numerical observations, we 
can conclude that the proposed 
spatial discretisation is 
robust and qualitatively correct. 
The time integrators are relatively comparable 
but only the splitting schemes have the 
additional properties of maintaining the positivity 
of the energy density and conserve exactly 
the invariants of our partial differential 
equation.

%

\bibliographystyle{plain}
\bibliography{bibnumhr}

\end{document}